\documentclass[10pt,a4paper,fleqn]{article}

\usepackage[paper=a4paper,left=25mm,right=25mm,top=25mm,bottom=25mm]{geometry}

\usepackage[utf8]{inputenc}

\usepackage{authblk}
\usepackage[english]{babel}
\usepackage[T1]{fontenc}
\usepackage{lmodern}
\usepackage{amsmath, amstext, amsfonts, mathrsfs}
\usepackage{amsfonts}
\usepackage{amssymb}
\usepackage{dsfont}
\usepackage{mathrsfs}
\usepackage{graphicx}
\usepackage{epsfig}
\usepackage[numbers]{natbib}
\usepackage{color}
\usepackage{bm}
\usepackage{verbatim}
\usepackage{stmaryrd}
\usepackage{appendix}
\usepackage{soul}	
\usepackage{bbm}			
\usepackage{enumerate} 
\usepackage{nicefrac}
\usepackage{xfrac}
\usepackage{units} 
\usepackage{mathtools}
\usepackage{pdfcomment}
\usepackage{url}
\usepackage{tikz}
\usepackage{cancel} 
\usepackage[labelfont={bf,sf},font={small},%
  labelsep=space]{caption}
\usepackage{acronym}
\usepackage{mathrsfs}
\usepackage{fancyhdr}
\usepackage[T1]{fontenc}
\usepackage{caption, booktabs}%
\usepackage{textcomp}
\usepackage{multirow}
\usepackage{tikz}
\usetikzlibrary{arrows.meta}
\usetikzlibrary{positioning,arrows}
\usetikzlibrary{matrix}
\usepackage{hologo}
\usepackage{xcolor}
\usepackage{forest}
\usetikzlibrary{matrix,fit}
\tikzset{%
  mleftdelimiter/.style={inner ysep=0pt, inner xsep=1ex,left delimiter=\{,label={[label distance=3mm]left:#1}}
}

\usepackage[Bjarne]{fncychap}

\usepackage{setspace}
\usepackage[squaren]{SIunits}
\usepackage{overpic}
\usepackage{subcaption}

\definecolor{light-gray}{gray}{0.95}
\definecolor{darkblue}{rgb}{0,0,.5}

\newcommand{\E}{\mathbb{E}}

\newcommand{\N}{\mathbb{N}}

\newcommand{\R}{\mathbb{R}}

\newcommand{\uu}{\bold{u}}

\newcommand{\zz}{\bold{z}}

\newcommand{\cM}{\mathcal{M}}
\newcommand{\cC}{\mathcal{C}}

\newcommand{\cF}{\mathcal{F}}

\newcommand{\cA}{\mathcal{A}}
\newcommand{\cB}{\mathcal{B}}

\newcommand{\cR}{\mathcal{R}}

\newcommand{\cT}{\mathcal{T}}

\providecommand{\keywords}[1]{\textbf{Keywords } #1}
\newcommand{\eqd}{\stackrel{\mathrm{d}}=}

\newcommand{\bN}{N}
\newcommand{\bE}{E}
\newcommand{\bT}{T}

\newcommand{\1}{\mathds{1}}

\newcommand{\de}{\mathsf{\,d}}

\newcommand{\Ran}{\mathsf{Ran}}

\newcommand{\law}{\mathsf{law}}
\newcommand{\MTP}{\mathrm{\text{MTP}_2}}

\newcommand{\Cor}{\mathrm{Cor}}

\newcommand{\CIS}{\textnormal{CIS} }

\newcommand{\CI}{\textnormal{CI} }

\newcommand{\Gauss}{\textnormal{Ga}}

\newcommand{\Cl}{\textnormal{Cl}}
\newcommand{\SCl}{\textnormal{SCl}}

\DeclareMathOperator{\Cov}{Cov}

\DeclareMathOperator{\v@r}{V@R}

\DeclareMathOperator{\av@r}{AV@R}

\newtheorem{theorem}{Theorem}[section]
\newtheorem{proposition}[theorem]{Proposition}
\newtheorem{corollary}[theorem]{Corollary}
\newtheorem{definition}[theorem]{Definition}
\newtheorem{lemma}[theorem]{Lemma}
\newtheorem{example}[theorem]{Example}
\newtheorem{remark}[theorem]{Remark}

\newenvironment{proof}[1][{Proof:}]{\begin{trivlist}
\item[\hskip \labelsep {\bfseries #1}]}{\hfill$\blacksquare$ \end{trivlist}}

\selectfont

\author[1]{Jonathan Ansari\thanks{Corresponding author: \texttt{jonathan.ansari@plus.ac.at}}}
\author[2]{Moritz Ritter\thanks{\texttt{moritz.ritter@stochastik.uni-freiburg.de}}}

\affil[1]{University of Salzburg (Austria)}
\affil[2]{University of Freiburg (Germany)}
\title{Comparison results for positive supermodular dependent \\ Markov tree distributions}

\begin{document}

\maketitle

\begin{abstract}
Positive dependencies have been compared in the literature under rather strong assumptions such as equality of conditional distributions, exchangeability, or stationarity. 
We establish supermodular ordering results for distributions that are Markov with respect to a tree structure. 
Our comparison results rely on simple stochastic monotonicity conditions and a pointwise ordering of bivariate copulas associated with the edges of the underlying tree.
We also study flexibility of the marginal distributions in stochastic and convex order.
As a consequence, we obtain first- and second-order stochastic dominance results for extreme order statistics and sums of positively dependent random variables.
As an application, we investigate distributional robustness of the maximum of a perturbed random walk under model uncertainty.
Several examples and a detailed discussion of the assumptions demonstrate the generality of our results and reveal deeper insights into non-intuitive positive dependence properties of multidimensional distributions.

\keywords{
bivariate copula,
conditional independence,
distributional robustness,
factor model,
hidden Markov model,
graphical model,
model uncertainty,
pair-copula construction,
perturbed random walk,
positive dependence,
tree-based dependence,
stochastic dominance,
supermodular order,
vine copula model
}
\end{abstract}

\maketitle

\section{Introduction}\label{secintro}

Positive dependencies play a fundamental role in both the theoretical and applied aspects of statistics. They arise naturally in a wide range of models and applications, including physics \cite{Propp-1996}, economics \cite{Milgrom-Roberts-1990}, finance and risk management \cite{Embrechts-2002,Denuit-2006}, reliability theory \cite{Barlow-1975}, and network analysis \cite{Dorp-1999,Daduna-1995}.
There is a large literature on concepts of positive dependence such as total positivity \cite{Fallat-2017,Karlin-Rinott-1981,Karlin-Rinott-1980}, conditional increasingness \cite{Mueller-Scarsini-2001}, association \cite{Fortuin-1971,Esary-1967}, positive supermodular dependence \cite{Shaked-1997}, and positive quadrant dependence \cite{Lehmann-1966}. 
All these concepts involve the comparison of positively dependent random vectors with independent random vectors of the same Fr{\'e}chet class.

The literature on comparing strengths of positive dependencies in a given Fr{\'e}chet class includes well known parametric families of distributions, e.g. the multivariate normal distribution \cite{Mueller-2001}, elliptical distributions \cite{Ansari-Rueschendorf-2021b,Yin-2023}, Archimedean copulas \cite{Mueller-Scarsini-2005}, or multivariate exponential distributions \cite{Marshall-Olkin-2011}.
In contrast, for non-parametric classes of positively dependent distributions, ordering results are studied under quite restrictive assumptions such as 
equality of conditional distributions \cite{Bauerle-1997,Tong-1977,Tong-1989}, exchangeability \cite{Shaked-1985}, or stationarity \cite{Hu-2000,Bauerle-1997b}.
In this paper, we establish supermodular ordering results for random vectors that are conditionally independent with respect to a tree structure. Our ordering conditions are both flexible and easy to verify, since they rely only on simple positive dependence assumptions and a pointwise ordering of the bivariate copulas 
associated with the edges of the underlying tree.

In many applications including reliability analysis, finance and insurance, particular important statistics are the minimum, the maximum, and the sum over the components of a random vector. The minimum may describe the first time a unit breaks down, leading to system failure. The maximum of a stochastic process is closely linked to the first passage time across a specified threshold.
The aggregated components of a random vector may represent the value of a portfolio of assets or the total insurance loss.
It is well known that comonotonicity describes the most extreme positive dependence for the maximum\footnote{The minimum and the maximum are related by symmetry, so all results carry over to the minimum, e.g. \(\min_n\{X_n\} \leq_{st} \min_n\{X_n^c\}.\) We will focus on the maximum in the sequel.} in stochastic order and the most risky dependence structure for the sum in convex order,\footnote{For real-valued random variables \(S\) and \(T\), the stochastic order \(S\leq_{st} T\) and the convex order \(S\leq_{cx} T\) are defined by \(\E \varphi(S) \leq \E \varphi(T)\) for all increasing resp. convex function \(\varphi\) such that the expectations exist; see Definition \ref{defdepord}.}
 i.e.,
 \begin{align*}
     \max_n\{X_n\}\geq_{st} \max_n\{X_n^c\} \quad \text{and} \quad \sum_n X_n \leq_{cx} \sum_n X_n^c
 \end{align*}
 for any random vector \((X_1,\ldots,X_d),\) where \(X_n^c := F_{X_n}^{-1}(U)\) for \(U\) uniform on \((0,1);\) see \cite{Tchen-1980,Wang-2025}.
However, in practice, comonotonicity is too extreme to be useful, 
which motivates to study stochastic orderings for comparing positively dependent random variables.

In the two-dimensional case and for fixed marginals, a sufficient condition for ordering maxima and sums is the upper orthant order \(\leq_{uo}\) (i.e., the pointwise order of survival functions).
As a direct consequence of the ordering results in \cite{Tchen-1980}, \((X_1,X_2)\leq_{uo} (Y_1,Y_2)\) with \(\law(X_i) = \law(Y_i)\) implies
\begin{align}\label{eqbivcase}
      \max\{X_1,X_2\} \geq_{st} \max\{Y_1,Y_2\} \quad \text{and} \quad  X_1+X_2 \leq_{cx} Y_1+Y_2.
\end{align}
These conditions are satisfied by many families of bivariate distributions and can often easily be checked. For dimension \(3\) and larger, a pointwise comparison of survival functions still implies a comparison of maxima in stochastic order, but no more a comparison of aggregated components in convex order; see \cite[Theorem 2.6]{Mueller-2000}. 

A stronger dependence order that also implies a convex ordering of component sums is the supermodular order. 
Largest elements in supermodular order are comonotonic random vectors reflecting the worst case behavior of sums as discussed above \cite{Tchen-1980}. The supermodular order is an integral stochastic order defined by comparing expectations of supermodular functions (see Definition \ref{defdepord} and Table \ref{table_functions}).
In the bivariate case, it reduces to the lower/upper orthant order, for which a verification is straightforward; see \eqref{eqbivord}.
In contrast, for dimension \(3\) and larger, checking the supermodular order is difficult because no small class of functions generating this order is known.

Well known construction methods for random vectors that are comparable in supermodular order are based on common mixture models, majorization, sequential Markov constructions, independence, or comonotonicity \cite{Bauerle-1997,Bauerle-1997b,Bauerle-1998,Hu-2000, Shaked-Shanthikumar-2007}. 
The results in these references generalize the inequalities in \eqref{eqbivcase} to arbitrary dimension. However,
the model assumptions are quite restrictive, as they involve equality of conditional distributions
or stationarity. 
For exchangeable random vectors, comparison results for \(\max_n\{X_n\}\) in stochastic order and for the variance of \(\sum_n X_n\) are derived in \cite{Shaked-1985}. 
Note that exchangeability is closely related to conditional independence. 
More precisely, by de Finett's theorem, any finite subset of an infinite sequence of exchangeable random variables is conditionally i.i.d. given a latent factor variable; see \cite{Hewitt-Savage-1955}. For (not necessarily conditionally i.i.d.) factor models, a supermodular ordering result was recently established in \cite{Ansari-Rueschendorf-2023}. This ordering result only requires stochastic monotonicity conditions and a pointwise ordering of bivariate copulas.
In particular, it extends \eqref{eqbivcase} to arbitrary dimensions for conditionally independent random variables. In this paper, we generalize the supermodular ordering result for factor models to the class of Markov tree distributions which model conditional independence along tree structures; see Definition \ref{defMTD}.

\begin{figure}[t]
	\begin{tikzpicture}[scale=.9, transform shape, remember picture]
            
        \begin{scope}[shift={(0,0)}] 
			\tikzstyle{every circle node} = []
			\node[circle] (a) at (0, 0) {\(X_0\)};
			\node[circle] (r) at +(-1,0) {a)};
			\node[circle] (b) at +(0: 1.7) {\(X_1\)};
			\node[circle] (c) at +(0: 3.4) {\(X_2\)};
			\node[circle] (d) at +(0: 5.1) {\(X_3\)};
            \node[circle] (e) at +(0: 5.8) {\(\cdots\)};
			\draw [-] (a) -- (b) node[pos=0.5,above]{\(B_{01}\)};
			\draw [-] (b) -- (c) node[pos=0.5,above]{\(B_{12}\)};
			\draw [-] (c) -- (d) node[pos=0.5,above]{\(B_{23}\)};
		\end{scope}

		\begin{scope}[shift={(11,-1.5)}] 
			\tikzstyle{every node} = []
			\node (a) at (-1.1, 1.6) {\(X_1\)};
			\node[circle] (r) at +(-3,1.5) {b)};
			\node (b) at +(1.25,1.75) {\(X_2\)};
			\node (c) at +(-.7,-1) {\(\ddots\)};
			\node (d) at +(2,-.5) {\(X_3\)};
			\node (e) at +(0.2,0.3) {\(X_0\)};
            \node (f) at +(0.25,-1.5) {\(X_4\)};
            \node (g) at +(-1.6,-.4) {\(X_d\)};
			\draw [-] (a) -- (e) node[pos=0.4, right=0.05]{\(B_{01}\)};
			\draw [-] (b) -- (e) node[pos=0.5, right=-0.03]{\(B_{02}\)};
			\draw [-] (d) -- (e) node[pos=0.4, above]{\(B_{03}\)};
            \draw [-] (f) -- (e) node[pos=0.4, right]{\(B_{04}\)};
            \draw [-] (g) -- (e) node[pos=0.6, below]{\(B_{0d}\)};
		\end{scope}

        \begin{scope}[shift={(2,-1.0)}] 
			\tikzstyle{every node} = []
			\node (a) at (0,0) {\(X_0\)};
            \node (b) at +(-1.2*1.2,-1.2*1.2) {\(X_1\)};
			\node (c) at +(1.2*1.2,-1.2*1.2) {\(X_7\)};
			\node (d) at +(-2*1.2,-2.4*1.2) {\(X_2\)};
			\node (e) at +(-0.5*1.2,-2.4*1.2) {\(X_3\)};
			\node (f) at +(0.5*1.2,-2.4*1.2) {\(X_8\)};
            \node (g) at +(2*1.2,-2.4*1.2) {\(X_9\)};
            \node (h) at +(-1.7*1.2,-3.8*1.2) {\(X_4\)};
            \node (i) at +(-.5*1.2,-3.8*1.2) {\(X_5\)};
            \node (j) at +(.7*1.2,-3.8*1.2) {\(X_6\)};
            \node (k) at +(1.5*1.2,-3.8*1.2) {\(X_{10}\)};
            \node (l) at +(2.7*1.2,-3.8*1.2) {\(X_{11}\)};
            \node (m) at +(2.1*1.2,-4.3*1.2) {\(\vdots\)};
            \node (n) at +(-.6*1.2,-4.3*1.2) {\(\vdots\)};
			\draw [-] (a) -- (b) node[pos=0.4, left]{\(B_{01}\)};
			\draw [-] (a) -- (c) node[pos=0.4, right=-0.03]{\(B_{07}\)};
            \draw [-] (b) -- (e) node[pos=0.4, right]{\(B_{13}\)};
			\draw [-] (b) -- (d) node[pos=0.4, left]{\(B_{12}\)};
            \draw [-] (c) -- (f) node[pos=0.4, left]{\(B_{78}\)};
            \draw [-] (c) -- (g) node[pos=0.4, right]{\(B_{79}\)};
            \draw [-] (g) -- (k) node[pos=0.52, left]{\(B_{9,10}\)};
            \draw [-] (g) -- (l) node[pos=0.52, right]{\(B_{9,11}\)};
            \draw [-] (e) -- (h) node[pos=0.5, left]{\(B_{34}\)};
            \draw [-] (e) -- (i) node[pos=0.5, right=-.1]{\(B_{35}\)};
            \draw [-] (e) -- (j) node[pos=0.5, right]{\(B_{36}\)};
            \node[circle] (r) at +(-3,0) {c)};
		\end{scope}

        \begin{scope}[shift={(8,-4.2)}] 
			\tikzstyle{every circle node} = []
			\node[circle] (a) at (0, 0) {\(X_0\)};
			\node[circle] (r) at +(-1,0) {d)};
			\node[circle] (b) at +(0: 1.7) {\(X_2\)};
			\node[circle] (c) at +(0: 3.4) {\(X_4\)};
			\node[circle] (d) at +(0: 5.1) {\(X_6\)};
            \node[circle] (ya) at (a |- 0,-1.5) {\(X_1\)};
            \node[circle] (yb) at (b |- 0,-1.5) {\(X_3\)};
            \node[circle] (yc) at (c |- 0,-1.5) {\(X_5\)};
            \node[circle] (yd) at (d |- 0,-1.5) {\(X_7\)};
			\draw [-] (a) -- (b) node[pos=0.5,above]{\(B_{02}\)};
			\draw [-] (b) -- (c) node[pos=0.5,above]{\(B_{24}\)};
			\draw [-] (c) -- (d) node[pos=0.5,above]{\(B_{46}\)};
            \draw [-] (a) -- (ya) node[pos=0.5,right]{\(B_{01}\)};
			\draw [-] (b) -- (yb) node[pos=0.5,right]{\(B_{23}\)};
			\draw [-] (c) -- (yc) node[pos=0.5,right]{\(B_{45}\)};
            \draw [-] (d) -- (yd) node[pos=0.5,right]{\(B_{67}\ \ \cdots\)};
		\end{scope}
	\end{tikzpicture}
	\caption{Examples of Markov tree distributions:
 a) A Markov process in discrete time, where the underlying tree is a chain.
 b) Random variables \(X_1,\ldots,X_d\) that are conditionally independent given the common factor variable \(X_0,\) where the underlying tree has a star-like structure. 
 c) A tree-indexed Markov process with a general underlying tree structure.
 d) A hidden Markov model with hidden nodes $(X_{i})_{i\in 2\mathbb N_0}$ and observable nodes $(X_i)_{i\in 2\mathbb N_0+1}$.
 Each model is uniquely determined by the univariate marginal distributions specifying the nodes and by the bivariate copulas \((B_{ij})_{(i,j)\in E}\) specifying the edges of the underlying tree \(T=(N,E);\) see Proposition \ref{prop: existence of markov tree dirstribution}. 
 }
	\label{5dimlabDvine}
\end{figure}

Markov tree distributions are a flexible and interpretable class of models because they
admit a simple representation in terms of bivariate distributions.
More precisely,
for a directed tree \(T = (N,E)\) of nodes \(N\) and edges \(E\subset N \times N,\) 
a fundamental property of Markov tree distributions is that they can be constructed by a sequence \(F=(F_n)_{n\in N}\) of univariate distribution functions and a sequence \(B = (B_e)_{e\in E}\) of bivariate copulas. 
While \(F_n\) determines the distribution of \(X_n,\) the copula \(B_e,\) \(e = (i,j)\in E,\) describes the dependence structure between the random variables \(X_i\) and \(X_j\) that are adjacent in \(T\); see Figure \ref{5dimlabDvine} and Proposition \ref{prop: existence of markov tree dirstribution}. 
Hence, Markov tree distributions are pair-copula constructions---and analyzing their distributional properties reduces to studying their specifications.
We write \(\cM(F,T,B)\) for the Markov tree distribution with these specifications.
Note that Markov tree distributions are the cornerstone of vine copulas, which have become very popular in the last decade \cite{Bedford-Cooke-2002,Czado-2019,Czado-2022}. More precisely, if Lebesgue densities exist, Markov tree distribution are vine copula models with only one layer.

For continuous marginal distribution functions, it is clear that Markov-tree distributed random variables \((X_n)_{n\in N}\sim \cM(F,T,B)\) are independent if and only if all the bivariate copulas \(B_e\) coincide with the independence copula. Similarly, 
\((X_n)_{n\in N}\) are comonotonic if and only if all \(B_e\) coincide with the comonotonicity copula. 
The focus of this paper is to derive simple conditions on the specifications \((F_n)_{n\in N}\) and \((B_e)_{e\in E}\) to model and order positive dependencies between the extreme cases of independence and comonotonicity.

We contribute to the literature on positive dependence orderings as follows.
\begin{enumerate}[1.]
    \item Our main result, Theorem \ref{thm: supermodular order of tree specifications generalized}, compares Markov tree distributions in supermodular order. In particular, it extends \eqref{eqbivcase} to conditionally independent models along tree structures noting that
    \((X_n) \leq_{sm} (Y_n)\) implies\footnote{We generally assume integrable random variables so that a comparison of the aggregated components in convex order makes sense.}
    \begin{align*}
    \max_{ n} \{X_n\} \geq_{st} \max_{n} \{Y_n\}\quad \text{and} 
     \quad \sum_{n} X_n \leq_{cx} \sum_{n} Y_n.
\end{align*}
    The non-intuitive positive dependence assumptions and the ordering conditions in Theorem \ref{thm: supermodular order of tree specifications generalized} can easily be checked because they rely only on the bivariate distributions related to the edges of the underlying tree; see Remark \ref{remdisasu} \eqref{remdisasu1} and \eqref{remdisasu2}.
    In Corollary \ref{corsm}, we formulate a version of our main result based on the bivariate copula specifications.
    \item In Theorem \ref{thm: maintwo}, we additionally allow for flexibility of the marginal distributions in stochastic order and establish a comparison of Markov tree distributions with respect to the increasing supermodular order. 
    This ordering continues to imply first- and second-order stochastic dominance results for extreme value statistics and sums of random variables; see \eqref{eqminmaxsum}.
    Similarly, Theorem \ref{themaindcx} permits flexibility of the marginal distributions in convex order, leading to a comparison result in the directionally convex order. The latter order implies a comparison of the aggregated components in convex order.
    \item Theorem \ref{thm: supermodular order of tree specifications generalized} is general in the sense that none of the positive dependence conditions can be omitted or weakened to positive supermodular dependence; see Proposition \ref{prop: justification of the SI assumptions}. Hence, our results reveal new insights into ordering properties of multidimensional distributions.
    Theorem \ref{thm: supermodular order of tree specifications generalized} also allows us to compare some negative dependencies, as not all bivariate dependence specifications are required to satisfy stochastic monotonicity conditions. 
    In Section \ref{secSM}, we investigate a particular property of star structures, for which arbitrary dependence specifications can be compared with positive dependencies in supermodular order; see Proposition \ref{propSch}. 
    Interestingly, the underlying ordering conditions are consistent with dependence measures such as Chatterjee's rank correlation \cite{Ansari-Langthaler-2025,chatterjee2021}.
    As we show, this ordering result for star structures does not extend to chains and, consequently, cannot be generalized to Markov tree distributions.
    \item Our contributions admit various applications to Markovian structures, including hidden Markov models, as we study in Section \ref{secHidMarkov}. Corollary \ref{thmHMM} is a version of our comparison results in the context of hidden Markov models, where the observations are modeled as conditionally independent given the latent process. 
    As an illustrative example, we consider the maximum of the observations and study its distributional robustness in stochastic order; see Theorem \ref{thedisrw}. For simplicity, the hidden process is modeled as a random walk---a setting commonly referred to as a perturbed random walk. 
    While the existing literature on maxima of perturbed random walks focuses on their asymptotic behavior \cite{Araman-2006,Iksanov-2016,Iksanov-2025}, our robustness results pertain to a finite number of observations.
\end{enumerate}

\subsection{Main results}

To motivate the non-intuitive positive dependence assumptions in Theorem \ref{thm: supermodular order of tree specifications generalized} below, we first formulate comparison results for chain and star structures, both of which are based on asymmetric positive dependence conditions.

Regarding the notation, 
we write \(V\uparrow_{st} U\) for random variables \(U\) and \(V,\) if \(V\) is stochastically increasing (SI) in \(U;\) see \eqref{defconin}. 
The supermodular order \(X\leq_{sm} Y\) is defined by \(\E f(X)\leq \E f(Y)\) for all \(b\)-bounded, supermodular functions \(f\); see Definition \ref{defdepord}, where a function \(f\colon \R^d \to \R\) is said to be \emph{supermodular} if \(f(x)+f(y) \leq f(x\wedge y) + f(x\vee y)\) for all \(x,y\in \R^n.\) Here, \(\wedge\) and \(\vee\) denote the componentwise minimum and maximum, respectively.
Further, a random vector is said to be positive supermodular dependent (PSMD) if it exceeds an independent random vector in supermodular order; see Definition \ref{defdepnotcop}.
The following result is a direct extension of \cite[Theorem 3.2]{Hu-2000} from stationary discrete-time Markov processes to a non-stationary setting.
In the language of Markov tree distributions, the underlying tree structure corresponds to a chain of nodes; see Figure \ref{5dimlabDvine}a).

\begin{proposition}[Supermodular ordering of Markov processes]\label{propMKprocess}~\\
    Let $X=(X_i)_{i\in \N_0}$ and $Y=(Y_i)_{i\in \N_0}$ be Markov processes 
    in discrete time. Assume for all \(i\in \N_0\) that 
    \begin{enumerate}[(i)]
        \item \label{propMKprocess1} \(X_{i+1}\uparrow_{st}X_i\),
        \item \label{propMKprocess2} \(Y_{i}\uparrow_{st} Y_{i+1},\)
        \item \label{propMKprocess3} \((X_i,X_{i+1})\leq_{sm} (Y_i,Y_{i+1})\) (resp.\@ $ \geq_{sm}$).
    \end{enumerate}
    Then it follows that $X\leq_{sm} Y $ (resp.\@ $ \geq_{sm}$). In particular, \(X\) and \(Y\) are PSMD.
\end{proposition}
Assumptions \eqref{propMKprocess1} -- \eqref{propMKprocess3} in the above proposition refer only to bivariate distributions and are therefore easy to verify; see Remark \ref{remdisasu}\,\eqref{remdisasu1} and \eqref{remdisasu2}. 
Due to the first two conditions,  
\(X\) and \(Y\) are SI in 'opposite directions'. As we discuss in the Proposition \ref{prop: justification of the SI assumptions} and Example \ref{exp:a}, these rather non-intuitive positive dependence assumptions cannot be weakened to PSMD, nor can they be replaced with SI in the 'same direction'. 
Note that, under the assumptions of Proposition \ref{propMKprocess}, the entire processes \(X\) and \(Y\) are PSMD.

For the proof of our main result, we use the following recently established ordering result for factor models where conditionally independent random variables are compared with respect to the supermodular order; see \cite[Corollary 4(i)]{Ansari-Rueschendorf-2023}. Such factor models can be interpreted as Markov tree distributions where the underlying tree has a star-like structure; see Figure \ref{5dimlabDvine}b).
\begin{lemma}[Supermodular ordering of Markovian star structures]\label{lemfacmod}~\\
Let \(X=(X_0,\ldots,X_d)\) and \(Y=(Y_0,\ldots,Y_d)\) be random vectors. Assume that \(X_1,\ldots,X_d\) are conditionally independent given \(X_0\), and that \(Y_1,\ldots,Y_d\) are conditionally independent given \(Y_0.\) Assume for all \(i\in \{1,\ldots,d\}\) that
\begin{enumerate}[(i)]
    \item\label{lemfacmod1} \(X_i\uparrow_{st} X_0,\) 
    \item\label{lemfacmod2} \(Y_i\uparrow_{st} Y_0,\) 
    \item\label{lemfacmod3} \((X_0,X_i)\leq_{sm} (Y_0,Y_i).\)
\end{enumerate}
Then it follows that \(X\leq_{sm} Y.\) In particular, \(X\) and \(Y\) are PSMD. 
\end{lemma}

While for chain structures, \(X\) and \(Y\) are required to satisfy opposite SI conditions, it is important to emphasize that the above result demands \(X\) and \(Y\) to fulfill identical SI conditions; see Example \ref{exp:b}.

To extend Proposition \ref{propMKprocess} and Lemma \ref{lemfacmod} to Markov tree distributions, it becomes necessary to combine the SI conditions in a suitable way. 
To this end, let $T=(N,E)$ be a directed tree with root \(0\in N\), which may have finitely or countably many nodes; see Figure \ref{5dimlabDvine}c).
Denote by \(L\subset N\) the set of leaves of \(T.\) 
Let \(P = \{\ell_1,\ell_2,\ell_3\ldots\}\subseteq N\setminus\{0\} \) with \((0,\ell_1),(\ell_1,\ell_2),(\ell_2,\ell_3),\ldots \in E\) be a path of nodes that starts with a child of the root and
either terminates at a leaf node \(\ell\in L\) or has infinitely many nodes.
 Further, let \(k^*\) be a child of the root \(0\) that does not belong to the path \(P\), unless it is the only child, i.e., \((0,k^*)\in E\)  with  $k^*\notin P$ whenever $\operatorname{deg}(0)\geq 2$.
 We refer to Section \ref{secdeftree} and 
 Definition \ref{def:childs_descendants_ancestors} for a formal specification of the notation and the terms.
 The following theorem is our main result.

\begin{theorem}[Supermodular ordering of Markov tree distributions]\label{thm: supermodular order of tree specifications generalized}~\\
Let $X=(X_n)_{n\in N}$ and $Y=(Y_n)_{n\in N}$ be sequences  of random variables that follow a Markov tree distribution with respect to $T$. Assume for all $e=(i,j)\in E$ that
	\begin{enumerate}[(i)]
        \item \label{thm11} \(X_j\uparrow_{st} X_i\) if \(j\ne k^*,\)
        \item \label{thm12} \(Y_i \uparrow_{st} Y_j\) if \(j \notin L,\) and \(Y_j\uparrow_{st} Y_i\) if \(j\notin P,\)
        \item \label{thm13} $(X_i,X_j) \leq_{sm}	(Y_i,Y_j)$ (resp.\@ $ \geq_{sm}$).
	\end{enumerate}
Then it follows that $ X \leq_{sm} Y $ (resp.\@ $ \geq_{sm}$). If additionally $(X_0,X_{k^*})$ is PSMD, then \(X\) is PSMD. Moreover, if additionally $(Y_i,Y_j)$ is PSMD for $(i,j)\in E$ with $j\in P\cap L,$ then  \(Y\) is PSMD.
\end{theorem}

The non-intuitive SI conditions in Theorem \ref{thm: supermodular order of tree specifications generalized} are illustrated in Figure \ref{fig: example main theorem}.
A detailed discussion of the assumptions provided in Appendix \ref{secSI} proves the generality of the above theorem and establishes that none of the SI conditions can be omitted or relaxed to PSMD. 
The technical proof of Theorem \ref{thm: supermodular order of tree specifications generalized}, presented in Appendix \ref{appendix: maintheorem}, builds on the proof for ordering stationary Markov processes in \cite{Hu-2000} and extends it by incorporating star structures. This requires a series of auxiliary results and additional notation to navigate appropriately through the tree. 
Note that Theorem \ref{thm: supermodular order of tree specifications generalized} compares dependence structures of Markov tree distributions for arbitrary univariate marginals. Related results for Poisson and Binomial random vectors are given in \cite{Cote-2024,Kizildemir-Privault-2018}.

As a consequence of the above result, the following corollary provides simple sufficient conditions on the bivariate copula specifications for a supermodular comparison of Markov tree distributions. We denote by \(\leq_{lo}\) the lower orthant (i.e., pointwise) order of distribution functions.
Further, a bivariate distribution function is said to be SI (CI) if there exists a bivariate random vector \((U,V)\) having this distribution function such that \(V\uparrow_{st} U\) (and \(U\uparrow_{st} V\)); see Definition \ref{defdepnotcop}. 

 \begin{corollary}[Supermodular ordering based on bivariate tree specifications]\label{corsm}~\\
       	Let $X\sim\cM(F,T,B)$ and $Y\sim\cM(F,T,C).$ Assume for all $e\in E$ that
	\begin{enumerate}[(i)]
		\item \label{corsm1} $B_{e}$\text{ is \textnormal{SI},}
		\item \label{corsm2} $C_{e}$\text{ is \textnormal{CI},}
		\item \label{corsm3} $B_{e} \leq_{lo}	C_{e}$ (resp.\@ $\geq_{lo}$).
	\end{enumerate}
	Then it follows that $X\leq_{sm}Y$ (resp.\@ $\geq_{sm}$).  In particular, \(X\) and \(Y\) are PSMD.
   \end{corollary}


\begin{figure}[t]
\centering
\begin{tikzpicture}[level distance=12mm,
  level 1/.style={sibling distance=25mm},
  level 2/.style={sibling distance=15mm},
  level 3/.style={sibling distance=10mm}]
  \node {$X_0$}
    child[->]  {node {$X_1$}
      child {node {$X_2$}}
      child {node {$X_3$}
        child {node {$X_4$}}
        child {node {$X_5$}}
        child {node {$X_6$}} 
        }
    }
    child  {node {$X_7$}
      child[->] {node {$X_8$}}
      child[->] {node {$X_{9}$}
        child {node {$X_{10}$}}
        child {node {$X_{11}$}}} 
    };
\end{tikzpicture}
\quad
\begin{tikzpicture}[level distance=12mm,
  level 1/.style={sibling distance=25mm},
  level 2/.style={sibling distance=15mm},
  level 3/.style={sibling distance=10mm}]
  \node {$Y_0$}
    child[<-] {node {$Y_1$}
      child[->] {node {$Y_2$}}
      child[<-] {node {$Y_3$}
        child[-] {node {$Y_4$}}
        child[->] {node {$Y_5$}}
        child[->] {node {$Y_6$}} 
        }
    }
    child[<->] {node {$Y_7$}
      child[->] {node {$Y_8$}}
      child[<->] {node {$Y_{9}$}
        child[->] {node {$Y_{10}$}}
        child[->] {node {$Y_{11}$}}} 
    };
\end{tikzpicture}
\caption{An example that illustrates the positive dependence conditions \eqref{thm11} and \eqref{thm12} in Theorem \ref{thm: supermodular order of tree specifications generalized} for a tree $T=(N,E)$ on \(12\) nodes with root \(0\).
An arrow \(U\longrightarrow V\) indicates \(V\uparrow_{st} U;\) similarly, an arrow \(U\longleftrightarrow V\) indicates \(V \uparrow_{st} U\) and \(U\uparrow_{st} V.\) 
Note that there is no positive dependence condition between \(X_0\) and \(X_7\) and between \(Y_3\) and \(Y_4.\)
The leaves consist of the set \(L=\{2,4,5,6,8,10,11\}.\) The set \(P\subseteq N\) consists of the leaf \(\ell=4\) and the path \(p(0,\ell)\) between the root \(0\) and the leaf \(\ell,\) i.e., \(P=\{1,3,4\}.\) The root's child $k^*$ is given by $k^*=7$.
}\label{fig: example main theorem}
\end{figure}

\begin{remark}\label{remdisasu}
\begin{enumerate}[(a)]
    \item \label{remdisasu1} For random variables \(U\) and \(V,\) the SI condition \(V\uparrow_{st} U\) is a positive dependence concept, satisfied by many well known bivariate distributions.
    For example, extreme value distributions \cite{Garralda-2000}, various Archimedean copulas \cite{Mueller-Scarsini-2005}, and the bivariate normal distribution for non-negative correlation \cite{Rueschendorf-1981} are SI; see \cite{Ansari-Rockel-2024} for an overview. 
    Further, the uniquely determined increasing rearrangement of a bivariate copula, recently studied in the context of dependence measures, is by construction SI; see \cite{Ansari-Rueschendorf-2021,strothmann2022}.
    Since SI random vectors are invariant under increasing transformations of the components (i.e., \(U\uparrow_{st} V\) implies \(f(U)\uparrow_{st} g(V)\) for all increasing functions \(f\) and \(g\); see \cite[Proposition 2.12]{Cai-2012}) the SI property is a copula-based dependence concept. Hence, it suffices to analyze copulas; see \eqref{eqsklar} for the notion of copula. 
    \item \label{remdisasu2} For bivariate distributions, the supermodular order can be verified by a pointwise comparison of the underlying copulas. This follows from \eqref{eqbivord} and from the invariance of the supermodular order under increasing transformations \cite[Theorem 9.A.9(a)]{Shaked-Shanthikumar-2007}.
    In contrast, for dimensions greater than \(2,\) the supermodular order is strictly stronger than the orthant orders, making a verification challenging since no small class of functions generating the supermodular order is known \cite{Mueller-Stoyan-2002}.
    This is why Theorem \ref{thm: supermodular order of tree specifications generalized} is particularly meaningful: it provides a novel method for constructing and comparing PSMD multivariate distributions based on bivariate building blocks, leveraging well known ordering results for bivariate distributions.
    For instance, in the case of the bivariate normal distribution, the supermodular order corresponds to an ordering of the correlation parameter \cite[Theorem 3.13.5]{Mueller-Stoyan-2002}; see \cite{Ansari-Rueschendorf-2021b} for an extension to elliptical distributions.
    For bivariate Archimedean copulas, a characterization of the supermodular order follows from \cite[Theorem 4.4.2]{Nelsen-2006}. For bivariate extreme value copulas, the supermodular order is characterized in \cite[Theorem 3.4]{Ansari-Rockel-2024}.
\item The SI assumption in Theorem \ref{thm: supermodular order of tree specifications generalized} are illustrated in Figure \ref{fig: example main theorem}. While \(X\) is SI from the root towards the leaves, the SI conditions for \(Y\) hold in both directions, except along the predefined path \(P\), where they are directed towards the root, and except at the leaves, where the SI conditions apply outward towards the leaves.
Further, for the edge that lies on the path \(P\) and connects to the leaf, no SI condition is required for \(Y.\) Similarly, for an edge that connects the root with one of its children that, if possible, does not lie on the path \(P\), no SI condition is required for \(X.\)
\item Theorem \ref{thm: supermodular order of tree specifications generalized} extends the comparison results for conditionally independent factor models and for discrete-time Markov processes to Markov tree distributions:
    If \(T\) is a chain, i.e., the edges of the tree are given by \(E=\{(0,1),(1,2),(2,3),...\},\) then condition \eqref{thm11} in Theorem \ref{thm: supermodular order of tree specifications generalized} is \(X_{i+1}\uparrow_{st} X_i\) for \(i\in \{1,2,3,\ldots\}\) and condition \eqref{thm12} simplifies to \(Y_i\uparrow_{st} Y_{i+1}\) for \(i\in \{0,1,2,\ldots\}.\) Hence, Theorem \ref{thm: supermodular order of tree specifications generalized} generalizes Proposition \ref{propMKprocess}.
    In particular, we obtain that condition \eqref{propMKprocess1} in Proposition \ref{propMKprocess} can be skipped for \(i=0,\) when we forgo with PSMD of \(X.\)
    In the case where \(T\) is a star on \(d+1\) nodes, i.e., when \(L=N\setminus \{0\}\) (all nodes except the root are leaves), then the set of edges is given by \(E=\{(0,1),(0,2),\ldots,(0,d)\}.\) In this case, conditions \eqref{thm11} and \eqref{thm12} of Theorem \ref{thm: supermodular order of tree specifications generalized} simplify to \(X_j\uparrow_{st} X_0\) for \(i\in \{1,\ldots,d\}\setminus \{k^*\}\) and \(Y_j\uparrow_{st} Y_0\) for \(j\in \{1,\ldots,d\}\setminus \{\ell\}.\) Hence, Theorem \ref{thm: supermodular order of tree specifications generalized} also generalizes Lemma \ref{lemfacmod} noting that condition \eqref{lemfacmod1} can be omitted for \(i=k^*\), and condition \eqref{lemfacmod2} can be omitted for \(i=\ell.\) In this case, \(X\) and \(Y\) are, in general, no longer PSMD.
    \item Corollary \ref{corsm} provides a simple method to construct and compare PSMD Markov tree distributions: 
   For any set of univariate distribution functions assigned to the nodes and for any set of bivariate SI copulas assigned to the edges of the tree, the implied Markov tree distribution is PSMD.
   Further, Corollary \ref{corsm} yields a flexible method to construct multivariate parametric distributions that are increasing in all of their parameters with respect to the supermodular order. 
   For fixed $F$ and $T$, this construction relies only on families of pointwise increasing \CI copulas assigned to the edges of the tree \(T\). 
   While existing construction methods from the literature rely on the fact that the supermodular order is closed under mixtures or under independent or comonotonic concatenations \cite[Theorem 9.A.3]{Shaked-Shanthikumar-2007}, our method is based on conditional independence along trees.
    \end{enumerate}
\end{remark}

\subsection{Structure of the paper}

The rest of the paper is organized as follows. 
Section \ref{sec: preliminaries} introduces Markov tree distributions and presents the relevant dependence concepts used throughout the paper.
As a consequence of our main result, Section \ref{sec3} provides comparison results when the marginals are in stochastic or convex order.
Section \ref{secSM} discusses, for star-shaped structures, a considerable extension of our supermodular ordering results to non-positive dependence specifications.
Section \ref{secHidMarkov} examines distributional robustness of hidden Markov models with a particular focus on the maximum of a perturbed random walk under model uncertainty.
Various counterexamples that demonstrate the generality of our results are presented in Section \ref{seccex}.
A detailed discussion of the assumptions of Theorem \ref{thm: supermodular order of tree specifications generalized} is given in Appendix \ref{secSI}.
All proofs, including the technical proof of Theorem \ref{thm: supermodular order of tree specifications generalized}, are deferred to Appendix \ref{secproofs}.


\section[Notation and basic concepts]{Notation and basic dependence concepts}\label{sec: preliminaries}

This section provides the notation and dependence concepts used throughout the paper. 
Proposition \nolinebreak \ref{prop: existence of markov tree dirstribution} gives a simple representation of Markov tree distributions in terms of bivariate tree specifications. This representation serves as both a practical tool for constructing Markov tree distributions and a formal framework for various comparison results studied in this paper.
The second part of this section covers the definitions and basic relations of the relevant stochastic orders and positive dependence concepts.

\subsection{Markov tree distributions}

Trees can be used to model simple dependencies between random variables. Each node of a tree represents a random variable, while the edges capture the dependence structure between random variables that are adjacent in the tree \cite{Czado-2019,Lauritzen-1996,Koller-2009}.
As we outline in the sequel, under the Markov property for tree structures, there exists a unique distribution on \(\R^N\)---the Markov tree distribution---fully determined by the univariate distribution functions assigned to the nodes and by the bivariate copulas assigned to the edges of the tree.

\subsubsection{Trees and Markov tree dependence}\label{secdeftree}

We denote by \(N\) an at most countable set of nodes 
which we label with the integers \(N = \{0,1, \dots, d\}\) for some \(d\in \N = \{1,2,\ldots\},\) whenever \(N\) has finitely many elements, and with \(N=\N_0 :=\{0\}\cup\N\)  otherwise. We assume $|N|\geq2$ to avoid cumbersome notation for trivial cases, where \(|N|\) denotes the number of elements of \(N.\)
A graph on \(N\) is a tuple \((N,E),\) where \(E\subset N\times N\) is a set of oriented edges. By abuse of notation we write $\{i,j\}\in E$ if $(i,j)\in E$ or $(j,i)\in E$. 
For \(i,j\in N\) with \(i\ne j,\) a \emph{directed path from $i$ to $j$} is a vector $(i,i_1,\dots,i_m,j)\subseteq N^{m+2}$ of \(m+2\) distinct nodes, \(m\in \N_0,\) such that $(i,i_1),(i_1,i_2),\dots,(i_m,j)\in E$.  Similarly, a (undirected) \emph{path between $i$ and $j$} is a set $\{i_1,\dots,i_m\}\subseteq N$ of \(m\) distinct nodes, \(m\in \N_0,\) such that $\{i,i_1\},\{i_1,i_2\},\dots,\{i_m,j\}\in E$.

In the literature on dependence modeling, trees are sometimes defined as acyclic graphs, where the edges are unordered pairs of nodes \cite{Bedford-Cooke-2002,cooke1997markov,Czado-2019,Joe-2011}.
However, since we generally allow for asymmetric dependence properties, we focus on directed trees (a.k.a.\@ oriented trees or arborescences) with a designated root, which we label, without loss of generality, as \(0\in N.\)
According to the following definition, a tree is a graph in which all nodes can be reached from the root via a unique directed path.
\begin{definition}[Directed tree]\label{def:tree}~\\
	A (directed) \emph{tree} is a graph $T=(N,E)$ 
 such that for all nodes $i\in \bN\backslash\{ 0 \}$ there exists a unique directed path from $0$ to $i$.
\end{definition}
Denote by \(\deg(i) := \#\{j\in N\colon (i,j)\in E \text{ or } (j,i)\in E\}\) the degree of node \(i\in N.\) 
By definition of a tree, we allow a node to have infinite degree, i.e., an infinite number of adjacent nodes. 
Further, any tree has no cycles in the sense that an undirected path between each two nodes $i$ and $j$ is uniquely determined and may also contain the root. We write \(p(i,j)\subseteq N\) (or equivalently $p(j,i)\subseteq N$) for this path, noting that the nodes $i$ and $j$ are not included in $p(i,j)$. 
The \emph{leaves} of a tree are defined as the subset of nodes in $N\backslash\{0\}$ having only one adjacent node, i.e.,
\begin{align}\label{defleaf}
    L := \{ k\in N\setminus \{0\} \mid \deg(k)=1\}\subset N.
\end{align}

The concept of Markov tree dependence uses a tree to model conditional independence between random variables indexed by the nodes of the tree; see \cite{cooke1997markov,Lauritzen-1996,meeuwissen1994tree}. 
Special cases are Markov processes in discrete time and conditionally independent factor models, where the underlying tree is a chain and a star, respectively (see Figure \ref{5dimlabDvine}\,a and b). 
A node $i\in N$ is said to \emph{separate} two disjoint sets $A,B\subset N$ if for every $a\in A$ and $b\in B$ the path between $a$ and $b$ contains $i$.
 
\begin{definition}[Markov tree dependence; {\cite[Definition 5]{meeuwissen1994tree}}]\label{defMTD}~\\
Let \(T=(N,E)\) be a tree.
A distribution \(\mu\) on $\mathbb R^{|N|}$ (resp. $\mathbb R^{\mathbb N_0}$ if $N=\mathbb N_0$) has \emph{Markov tree dependence} (or is a \emph{Markov tree distribution}) with respect to $T$ if there exists a sequence \((X_n)_{n\in N} \sim \mu\) of random variables such that for every two finite disjoint subsets \(A,B\subset N\) and for every \(i\in N\) that separates \(A\) and \(B\), the vectors \(X_A=(X_j)_{j\in A}\) and \(X_B=(X_j)_{j\in B}\) are conditionally independent given \(X_i\).
\end{definition}

Weaker (i.e., non-Markovian) and stronger (i.e., higher-order Markovian) concepts of tree dependence are also discussed in \cite{Bedford-Cooke-2002,meeuwissen1994tree}.

\subsubsection{Bivariate tree specifications and Markov realizations}

For various comparison results,
we use the concept of copulas, a tool that allows us us to study dependence structures between random variables in full generality. More precisely, a \emph{\(d\)-copula} is a \(d\)-variate distribution function \(C\colon [0,1]^d \to [0,1]\) having uniformly on $[0,1]$ distributed univariate marginal distributions. Due to Sklar's theorem, every $d$-variate distribution function $F$ on \(\R^d\) can be decomposed into its univariate marginal distribution functions $F_i$, $1\leq i\leq d$, and a $d$-copula \(C\) such that the joint distribution function can be expressed as the concatenation of these, i.e., 
\begin{align}\label{eqsklar}
    F(x) = C(F_1(x_1),\dots,F_d(x_d))\quad \text{for all } x = (x_1,\ldots,x_d)\in \R^{d}.
\end{align}
In this case, $C$ is called a copula of $F$.
The copula \(C\) is uniquely determined on \(\Ran(F_1)\times \cdots \times \Ran(F_d),\)  where \(\Ran(f)\) denotes the range of a function \(f.\)  Further, for any copula \(C\) and for all univariate distribution functions \(F_1,\ldots,F_d,\) the right-hand side of \eqref{eqsklar} defines a \(d\)-variate distribution function.
If $X=(X_1,\dots,X_d) $ has distribution function \(F,\) we say that $C = C_X$ is a copula of $X$. 
For an overview of the concept of copulas, see e.g. \cite{Durante-2016,Nelsen-2006,Rueschendorf-2013}.

As a consequence of the definition of Markov tree dependence (Definition \ref{defMTD}), for any path \(p(i,j) = \{i_1,\ldots, i_m\}\) from $i$ to $j,$ the conditional distribution \(X_i\mid (X_{i_1},\ldots,X_{i_m},X_j)\) depends only on the random variable \(X_{i_1}\), which is adjacent to \(X_i.\)
This implies that Markov tree distributions are completely specified by the distributions of the bivariate random variables \((X_k,X_\ell),\) \((k,l)\in E,\) that are adjacent in the underlying tree \(T.\) 
Due to Sklar's theorem, each such bivariate distribution function \(F_{k\ell} = F_{X_k,X_\ell}\) can be decomposed into its marginal distribution functions \(F_k\) and \(F_\ell\) and a bivariate copula \(C_{k\ell}\), which describes the dependence structure between \(X_k\) and \(X_\ell\), i.e., \(F_{k\ell}(x,y) = C_{k\ell}(F_k(x),F_\ell(y))\) for all \((x,y)\in \R^2.\)
For a fixed tree, a bivariate tree specification assigns a univariate distribution function to each node and a bivariate copula to each edge of the tree as follows.

\begin{definition}[Bivariate tree specification; {\cite[Definition 4]{meeuwissen1994tree}}]\label{def: Bivariate tree specification}~\\
	A triple \(\cT=(F,T,B)\) is a \emph{bivariate tree specification} if
	\begin{enumerate}[(i)]
		\item $T=(N,E)$ is a tree,
		\item \(F=(F_n)_{n\in N}\) is a family of univariate distribution functions,
		\item \(B=(B_e)_{e\in E}\) is a family of bivariate copulas.
	\end{enumerate}
\end{definition}

For a probability distribution \(\mu\) on \(\R^{|N|},\) if \(d<\infty,\) or on $\mathbb R^{\mathbb N_0}$, if $N=\mathbb N_0$, denote by \(\mu_n\) and \(\mu_{ij},\) \(n,i,j\in N,\) the univariate and bivariate marginal distributions with respect to the components \(n\) and \((i,j),\) respectively. Then \(\mu\) is said to \emph{realize} a bivariate tree specification \(\cT=(F,T,B),\) if for all \(n\in N\) and \(e=(i,j)\in E,\) \(F_n\) is the distribution function of \(\mu_n\) and \(B_e\) is a copula of \(\mu_{ij}.\)

Due to the following proposition, for every bivariate tree specification there exists a unique realizing Markov tree distribution; see \cite{Bedford-Cooke-2002} for the case when Lebesgue densities exist.

\begin{proposition}[Markov realization of a bivariate tree specification]\label{prop: existence of markov tree dirstribution}~\\
	For every bivariate tree specification $\cT=(F,T,B)$ there is a unique distribution $\mu$ that realizes the bivariate tree specification $\cT$ such that $\mu$ has Markov tree dependence with respect to $T$.
\end{proposition}
Recall that, for a bivariate tree specification \(\cT=(F,T,B),\) we denote its uniquely determined Markov realization by
\(\cM(F,T,B)\), and we write \(X\sim \cM(F,T,B)\) for a sequence \(X=(X_n)_{n\in N}\) of random variables with this specification.

\subsection[Dependence orderings/concepts]{Stochastic orderings and positive dependence concepts}
Our comparison results are formulated in terms of integral stochastic orderings, which compare expectations of functions of two random vectors. To this end, consider for a measurable \emph{weight} function \(b\colon\R^d \to [1,\infty)\) (e.g., \(b(x) = \max\{1,\lVert x \rVert^p\}\) for some \(p\in \N_0\)) the class \(\cR_b := \{V = (V_1,\ldots,V_d) \mid \E b(V) < \infty\}\) of \(b\)-integrable random vectors on a probability space \((\Omega,\cA,P).\) Then, for some class  \(\cF\subseteq \{f \colon \R^d\to \R\mid f \text{ measurable}, \sup_{x\in \R^d} |f(x)|/b(x) < \infty\}\) of \(b\)-bounded functions and for \(V,W\in \cR_b,\) define the integral stochastic order
\begin{align*}
	V\prec_{\cF} W \quad\text{if} \quad\E f(V)\leq \E f(W) \quad\text{for all } f\in \cF;
\end{align*}
see \cite{Mueller-1997}.
Recall that a function \(f\colon \R^d \to \R\) is supermodular if \(f(x)+f(y) \leq f(x\wedge y) + f(x\vee y)\) for all \(x,y\in \R^n.\)
The function \(f\) is said to be \emph{increasing/decreasing supermodular} if it is supermodular and componentwise increasing/decreasing. If \(f\) is supermodular and componentwise convex, it is \emph{directionally convex}.
For several examples of such functions, see Table \ref{table_functions}.
If $f$ is sufficiently smooth, then $f$ is supermodular if and only if $\frac{\partial^2}{\partial x_i\partial x_j}f(x)\geq 0$ for all $1\leq i<j\leq d$ and for all \(x.\) Similar properties hold true for sufficiently smooth increasing/decreasing supermodular or directionally convex functions. 

Denote by \(\cF_\uparrow,\) \(\cF_{sm},\) \(\cF_{ism},\) \(\cF_{dsm},\) and \(\cF_{dcx},\) the class of componentwise increasing, supermodular, increasing supermodular, decreasing supermodular, and directionally convex \(b\)-bounded functions,
respectively. 
We will use the following integral stochastic orderings.

\begin{definition}[Stochastic orderings]\label{defdepord}~ 
 \begin{enumerate}[(a)]
     \item \label{defdeporda}  Let \(V = (V_1,\ldots,V_d), W=(W_1,\ldots,W_d)\in \cR_b\).  Then \(V\) is said to be smaller than \(W\) with respect to 
	\begin{enumerate}[(i)]
        \item the \emph{stochastic order}, written \(V\leq_{st} W\),  if $V\prec_{\cF_{\uparrow}} W$,
		\item the \emph{supermodular order}, written \(V\leq_{sm} W\),  if $V\prec_{\cF_{sm}} W$,
        \item the \emph{increasing (decreasing) supermodular order}, written \(V\leq_{ism} W\) (\(V\leq_{dsm} W\)),  if $V\prec_{\cF_{ism}} W$ ($V\prec_{\cF_{dsm}} W$),
		\item the \emph{directionally convex order}, written \(V\leq_{dcx} W\), if $V\prec_{\cF_{dcx}} W$.
    \end{enumerate}
    \item \label{defdepordb} Let \(V=(V_n)_{n\in \N}\) and \(W=(W_n)_{n\in \N}\) be stochastic processes. Let \(\prec\) be one of the orderings in \eqref{defdeporda}. Then \(V\) is said to be smaller than \(W\) with respect to \(\prec\) if for all \(m\in \N\) and all \((n_1,\ldots,n_m)\in \N^m,\) one has \((V_{n_1},\ldots,V_{n_m}) \prec (W_{n_1},\ldots,W_{n_m}).\)
  \end{enumerate}
\end{definition}
For \(d=1,\) the directionally convex order reduces to the well known convex order \(\leq_{cx},\) and the increasing convex order \(\leq_{icx}\) is defined for the class of functions that are increasing and convex. 
Note that the comparison of stochastic processes in part \eqref{defdepordb} is defined through the comparison of their finite-dimensional marginal distributions. This corresponds to the concept of strong comparison of stochastic processes in \cite[Definition 5.1.2]{Mueller-Stoyan-2002}.
We denote by \(V \leq_{lo} W\) and \(V \leq_{uo} W\) the lower and upper orthant order, defined by \(F_V(x)\leq F_W(x)\) and \(\overline{F}_V(x)\leq \overline{F}_W(x)\) for all \(x\in \R^d,\) respectively. 
Some basic relations between the above considered orderings are
\begin{align}\label{smpdo}
	&V\leq_{dcx} W ~ \text{and} ~ V\leq_{ism} W \quad \Longleftarrow \quad V\leq_{sm} W ~~ \Longrightarrow \quad V\leq_{lo} W ~\text{and} ~ V\leq_{uo} W,\\
    \label{eqminmaxsum}
    &V\leq_{ism} W \quad \Longrightarrow \quad \max_{ n} \{V_n\} \geq_{st} \max_{n} \{W_n\} ~~\text{and} 
     ~~ \sum_{n} V_n \leq_{icx} \sum_{n} W_n.
\end{align} 
Further, \(V\leq_{dcx} W\) implies \(\sum_n V_n\leq_{cx} \sum_n W_n\).
If \(\E[V_n] = \E[W_n]\) for all \(n,\) the increasing supermodular order reduces to the supermodular order and the increasing convex order in \eqref{eqminmaxsum} reduces to the convex order.
As a direct consequence of \eqref{smpdo}, the supermodular order is a pure dependence order, i.e., \(V\leq_{sm} W\) implies \(V_i\eqd W_i\) for all \(i\in \{1,\ldots,d\},\) where \(\eqd\) denotes equality in distribution. 
For bivariate random vectors with identical marginal distributions, the orderings in \eqref{smpdo} are equivalent, i.e., if \(d=2,\) then
\begin{align}\label{eqbivord}
 &V\leq_{lo} W &&\Longleftrightarrow &&V\leq_{uo} W &&\Longleftrightarrow && V\leq_{sm} W &&\Longleftrightarrow && V\leq_{dcx} W
\end{align}
whenever \(V_1\eqd W_1\) and \(V_2 \eqd W_2;\) see \cite[Theorem 2.5]{Mueller-2000}. 
Hence, for identical marginals, a verification of the ordering condition \eqref{thm13} in Theorem \ref{thm: supermodular order of tree specifications generalized} reduces to a pointwise comparison of the associated bivariate distribution functions or copulas. \\

For modeling positive dependencies, we make use of several positive dependence concepts. 
To this end, we denote by \(V^\perp:=(V_1^\perp,\ldots,V_d^\perp)\) an independent random vector with the same marginal distributions as $V=(V_1,\dots,V_d)$, i.e., \(V_1^\perp,\ldots,V_d^\perp\) are independent and \(V_i^\perp\eqd V_i\) for all $i\in\{1,\dots,d\}$. 


\begin{definition}[Concepts of positive dependence]\label{defdepnotcop}~\\
	A random vector $V=(V_1,\dots,V_d)$ is said to be
	\begin{enumerate}[(i)]
		\item \emph{positive supermodular dependent} (PSMD) if \(V^{\perp} \leq_{sm} V,\)
		\item  \emph{conditionally increasing} (CI) if 
		\begin{align}\label{defconin}
			V_i\uparrow_{st} (V_j, ~j\in J)
		\end{align}
		for all \(i\in \{1,\ldots,d\}\) and \(J\subseteq \{1,\ldots,d\}\setminus \{i\},\) where \eqref{defconin} means that the conditional distribution \(V_i \mid (V_j=x_j, j\in J)\) is $\leq_{st}$-increasing in \(x_j\) for all  \(j\in J\), $P^{V_J}$-a.s.,  i.e., \(\E[f(V_i)\mid V_j=x_j, j\in J]\) is increasing in \(x_j\) for all \(j\in J\) outside a $V_J$-null set and for all increasing functions \(f\colon \R\to \R\) such that the expectations exist,
        \item \label{defdepnotcop5}  \emph{conditionally increasing in sequence} (\textnormal{CIS}) if  \eqref{defconin} holds for all \(i\in \{2,\ldots,d\}\) and  \(J\subseteq\{1,\ldots,i-1\}\), 
		\item \emph{multivariate totally positive of order \(2\)} (\(\MTP\)) if \(V\) is absolutely continuous with Lebesgue-density $f$ such that $\log(f)$ is supermodular. 
	\end{enumerate}
\end{definition}
The above concepts of positive dependence are defined similarly for probability distributions and distribution functions, and remain invariant under (strictly) increasing transformations; see \cite[Theorem 3.10.19]{Mueller-Stoyan-2002} and \cite{Cai-2012}. 
The concepts are related by 
\begin{align}\label{implposdepcon}
	\MTP ~~~\Longrightarrow ~~~ \text{CI} ~~~\Longrightarrow ~~~ \text{CIS} ~~~\Longrightarrow ~~~ \text{PSMD} , 
\end{align}
where the implications are strict for dimension \(\geq 2.\)

\section[Flexibility of marginals]{Flexibility in the marginal distributions}\label{sec3}

In Theorem \ref{thm: supermodular order of tree specifications generalized} and Corollary \ref{corsm}, we have compared Markov tree distributions
with respect to the supermodular order,
which requires identical marginal distributions.
However, if the marginals of \(X\) and \(Y\) differ or are only partially known, some flexibility of our results with respect to the marginals is desirable.
First, we study flexibility of the marginal distributions in stochastic order, which yields a comparison result with respect to the increasing/decreasing supermodular order. Then we allow flexibility of the marginal distributions in convex order and obtain a comparison result with respect to the directionally convex order. 
Since conditional independence is not solely a copula-based property,
caution is required when comparing Markov tree distributions with different and discontinuous marginal distribution functions; see Example \ref{exp: marginal distribution determines copula}.

\subsection[Supermodular functions]{Inequalities for increasing/decreasing supermodular functions}

In the following result, we additionally compare the marginal distributions in stochastic order. We assume that the ranges of the marginal distribution functions are identical except on Lebesgue null sets. The closure of a set \(A\subseteq [0,1]\) is denoted by \(\overline{A}.\) 
 

\begin{theorem}[Comparison of marginals in stochastic order]\label{thm: maintwo}~\\
    Assume that $X = (X_n)_{n\in N}\sim\cM(F,T,B) $ and $Y = (Y_n)_{n\in N} \sim\cM(G,T,C)$ satisfy the positive dependence conditions \eqref{thm11} and \eqref{thm12} of Theorem \nolinebreak\ref{thm: supermodular order of tree specifications generalized} and assume that \(\overline{\Ran(F_n)} = \overline{\Ran(G_n)}\) for all \(n\in N.\) Then the following statements hold true:
    \begin{enumerate}[(i)]
        \item \label{thm: maintwo1}\(B_e\leq_{lo} C_e\)  for all \(e\in E\) and \(F_n\leq_{st} G_n\) for all \(n\in N\) implies \(X\leq_{ism} Y,\)
        \item \label{thm: maintwo2} \(B_e\leq_{lo} C_e\)  for all \(e\in E\) and \(F_n\geq_{st} G_n\) for all \(n\in N\) implies \(X\leq_{dsm} Y.\)
    \end{enumerate}
\end{theorem}

\begin{remark}\label{rem38}
    \begin{enumerate}[(a)]
    \item 
    As a consequence of Theorem \ref{thm: maintwo}, we obtain distributional robustness for various functionals that are consistent with the increasing resp. decreasing supermodular order. In particular, we obtain distributional robustness of extreme order statistics and sums of random variables with respect to the stochastic order and the increasing convex order as in \eqref{eqminmaxsum}. Note that for bivariate copulas, \(B_e \leq_{lo} C_e\), \(B_e \leq_{uo} C_e\), and \(B_e\leq_{sm} C_e\) are equivalent due to \eqref{eqbivord}.
    \item Due to Example \ref{exp: marginal distribution determines copula}, the assumption \(\overline{\Ran(F_n)} = \overline{\Ran(G_n)}\) cannot be omitted because the univariate marginal distributions generally affect the dependence structure under the Markov property. For the proof of Theorem \ref{thm: maintwo}, we use that the copula of a Markov tree distribution depends only on the ranges of the marginal distribution functions.
    \item \label{rem38b} The assumption \(\overline{\Ran(F_n)} = \overline{\Ran(G_n)}\) is, in particular, satisfied when \(F_n\) and \(G_n\) are continuous. In Example \ref{exGaussobs}, we use the fact that the random variables \(\xi_i^+ := \max\{\xi_i,0\},\) for \(\xi_i\sim N(0,\sigma_i),\) \(i=1,2,\) are comparable in stochastic order and satisfy \(\Ran(F_{\xi_1^+}) = \Ran(F_{\xi_2^+}) = \{0\}\cup [1/2,1).\) More precisely, if \(\sigma_1 \leq \sigma_2,\) then \(\xi_1^+ \leq_{st} \xi_2^+.\)
    \end{enumerate}
\end{remark}

\begin{center}
\begin{table}[t]
\begin{tabular}{l|l} \toprule
 function \(f(x_1,\ldots,x_d)\) & Properties \\ \midrule
 \(\1_{\{x\geq t\}}\) & increasing supermodular  \\
 \(\1_{\{x\leq t\}}\) & decreasing supermodular \\
\midrule
\(P(\xi_1\leq x_1,\ldots,\xi_d\leq x_d)\) & increasing supermodular  \\
 \(P(\xi_1\geq x_1,\ldots,\xi_d\geq x_d)\) &  decreasing supermodular \\ \midrule
 \(\min\{x_1,\ldots,x_d\}\) & increasing supermodular\\
 \(\max\{x_1,\ldots,x_d\}\) &  increasing supermodular, directionally convex \\ \midrule
 \(\1_{\{\min\{x_1,\ldots,x_d\}\geq K\}}\) & increasing supermodular  \\
 \(\1_{\{\max\{x_1,\ldots,x_d\}\leq K\}}\) &  decreasing supermodular \\ \midrule
 \((K-\max\{x_1,\ldots,x_d\})_+\) &  decreasing supermodular  \\
 \((\max\{x_1,\ldots,x_d\}-K)_+\) & increasing supermodular, directionally convex  \\ 
 \((K-\min\{x_1,\ldots,x_d\})_+\) &   decreasing supermodular, directionally convex \\
\((\min\{x_1,\ldots,x_d\}-K)_+\) &  increasing supermodular \\ \midrule
 \(\varphi(\sum_{n=1}^d \alpha_n x_n)\), \(\varphi\) convex &  supermodular, directionally convex \\
 \(\varphi(\sum_{n=1}^d \alpha_n x_n)\), \(\varphi\) increasing convex &  increasing supermodular, directionally convex \\
 \(\varphi(\sum_{n=1}^d  g_n(x_n))\), \(\varphi\) convex &  supermodular\\
 \(\varphi(\sum_{n=1}^d g_n(x_n))\), \(\varphi\) increasing convex &  increasing supermodular
\\ \midrule
\end{tabular}
\caption{
Important examples and classifications of functions implying comparison results for \(\E f(X)\) with respect to the integral stochastic orders considered in this paper, where \((\xi_1,\ldots,\xi_d)\) is a random vector (independent of \(X\)) on a probability space \((\Omega,\cA,P)\) and where \(t\in \R^d,\) \(K\in \R,\) \((y)_+:=\max\{y,0\},\) \(\alpha_n\geq 0,\) and \(g_n\colon \R \to \R\) increasing. Since the integral stochastic orders satisfy various invariance properties, the above examples also apply to the respective transformations of the functions, for example to componentwise increasing/decreasing transformations in the case of the supermodular order; see \cite{Bauerle-1997}. 
}\label{table_functions}
\end{table}
\end{center}

\subsection[Directionally convex functions]{Inequalities for directionally convex functions}

When the marginals are in convex order, a comparison with respect to the directionally convex order can be achieved if the underlying multivariate copula is CI. This is the content of the following well known result, which is not restricted to Markov tree distributions.

\begin{lemma}[Common \CI copula; {\cite[Theorem 4.5]{Mueller-Scarsini-2001}}]\label{lemcCI}~\\
    Let \(U=(U_0,\ldots,U_d)\) and \(V=(V_0,\ldots,V_d)\)  be random vectors having the same $(d+1)$-dimensional copula \(C=C_U=C_V.\)  If \(C\) is CI, then \(U_i\leq_{cx} V_i\) for all \(i\in \{0,\ldots,d\}\) implies \(U\leq_{dcx} V.\)
\end{lemma}
As we demonstrate in Example \ref{exp:counter for ci}, bivariate CI specifications generally do not result in a Markov tree distribution with a CI copula. 
A sufficient condition for the underlying copula to be \textnormal{CI} is provided by the following special case of \cite[Proposition 7.1]{Fallat-2017}.

\begin{lemma}[\(\MTP\) specifications]\label{lemMTP2}~\\
    Let \(V=(V_0,\ldots,V_d) \sim \cM(G,T,C)\) be a Markov tree distributed random vector such that \((V_i,V_j)\) is \(\MTP\) for all \(e=(i,j)\in E.\) Then \(V\) is \(\MTP.\)
\end{lemma}

Combining the above lemmas and Theorem \ref{thm: supermodular order of tree specifications generalized}, we establish the following \(\leq_{dcx}\)-comparison result for Markov tree distributions, which allows a comparison of the marginal distributions in convex order. Note that the bivariate copula specifications of \(Y\) are now assumed to satisfy the stronger positive dependence concept \(\MTP.\) 

\begin{theorem}[Comparison of marginals in convex order]\label{themaindcx}~\\
For a tree \(T=(N,E),\)
let $X = (X_n)_{n\in N}\sim\cM(F,T,B) $ and $Y = (Y_n)_{n\in N} \sim\cM(G,T,C)$ be Markov tree distributed sequences of random variables.
Assume for all \(n\in N\) that the marginal distribution functions \(F_n\) and \(G_n\) are continuous. 
If for all \(e=(i,j)\in E,\)
\begin{enumerate}[(i)]
    \item \label{themaindcx1} $B_{e}$ is SI if $j\neq k^*$ for some fixed child $k^*$ of the root (i.e., for \((0,k^*)\in E\)),
    \item \label{themaindcx2} $C_e$ is \(\MTP\),
    \item \label{themaindcx3} \(B_e \leq_{lo} C_e\) (resp. \(\geq_{lo}\)),
\end{enumerate}
then \(F_n\leq_{cx} G_n\) (resp. \(\geq_{cx}\)) for all \(n\in N\) implies \(X\leq_{dcx} Y\) (resp. \(\geq_{dcx}\)).
\end{theorem}

\begin{remark}
By the continuity assumption on the marginal distribution functions in Theorem \ref{themaindcx}, the copula of \(X\) is the distribution function of \((F_n(X_n))_{n\in N}\), and the copula of \(Y\) is the distribution function of \((G_n(Y_n))_{n\in N}.\) 
In the proof of Theorem \ref{themaindcx}, we apply Lemma \ref{lemMTP2} and compare the marginals of \(X\) and \(Y,\) using that \((G_n(Y_n))_{n\in N}\) is \(\MTP\) (and thus CI) as a consequence of the \(\MTP\) assumptions on \(C_e,\) \(e\in E.\)
\end{remark}

\section{A special property of star structures}\label{secSM}
In this section, we study a significant extension of the comparison results for star structures in Lemma \nolinebreak \ref{lemfacmod} to a comparison of general dependencies. 
This extension builds on the recently established Schur order for conditional distributions which compares conditional distribution functions of bivariate random vectors in terms of their strength of dependence, measured by the variability in the conditioning variable. 
The Schur order for conditional distributions has the fundamental properties that minimal elements characterize independence (i.e., no variability in the conditioning variable) and maximal elements characterize perfect dependence (i.e., maximal variability of the decreasing rearrangements); see \cite{Ansari-Fuchs-2022,Ansari-Rockel-2024}. Here a random variable \(Y\) is said to \emph{perfectly depend} on \(X\) if there exists a Borel measurable function \(f\) (which is not necessarily increasing or decreasing) such that \(Y=f(X)\) almost surely.
Note that perfect dependence is not a symmetric concept, i.e., perfect dependence of \(Y\) on \(X\) does not imply perfect dependence of \(X\) on \(Y.\) 
 From Lemma \ref{lemfacmod} on Markovian star structures, we know that strong positive dependence between \(Y_n\) and \(Y_0\) for all \(n\in \{1,\ldots,n\}\) leads to strong positive dependence among \((Y_1,\ldots,Y_d)\) in the sense of the supermodular order. Similarly, as can easily be verified, strong negative dependence between each \(Y_n\) and \(Y_0\) also implies strong positive dependence among \((Y_1,\ldots,Y_d).\)

 Now, assume that, for every \(n\in \{1,\ldots,d\},\) \(X_n\) is less dependent on the common factor variable \(X_0\) than \(Y_n\) on \(Y_0\) in the sense of the Schur order for conditional distributions and assume that \((Y_n,Y_0)\) has positive dependence in some sense.  Then, a fairly intuitive result for Markovian star structures (Proposition \ref{propSch}) states that \((Y_1,\ldots,Y_d)\) exhibits stronger positive dependence than \((X_1,\ldots,X_d).\) 
Similarly, one might expect for Markov processes that stronger dependence (in the sense of the Schur order) among \((Y_n,Y_{n+1})\) compared to \((X_n,X_{n+1})\) for all \(n\in \N\) would result in stronger dependence among \((Y_n)_{n\in \N}\) compared to \((X_n)_{n\in \N}\)---at least when all \((Y_n,Y_{n+1})\) are conditionally increasing. 
Surprisingly, as we show below, comparison results with respect to the Schur order for conditional distributions cannot be extended from star structures to chain structures and, consequently, not to Markov tree distributions. 
In other words, only for star structures, greater variability in the conditioning variable increases the strength of (positive) dependence of the entire vector in the supermodular order.
Hence, Theorem \ref{thm: supermodular order of tree specifications generalized} is also general in the sense that there is no direct extension to the Schur order for conditional distributions, which we formally define as follows; see \cite{Ansari-Fuchs-2022}.

For integrable functions \(f,\,g\colon (0,1)\to \R\) the \emph{Schur order} \(f\prec_S g\) is defined by
\begin{align}\label{defschurfun}
\begin{split}
f\prec_S g \quad \colon\Longleftrightarrow \quad \int_0^x f^*(t)d \lambda(t) &\leq \int_0^x g^*(t)d \lambda(t) ~~~\text{for all } x\in (0,1) \text{ and} \\ ~~~
\int_0^1 f^*(t)d \lambda(t) &= \int_0^1 g^*(t)d\lambda(t),
\end{split}
\end{align}
where \(h^*\) denotes the decreasing rearrangement 
of an integrable function \(h\colon (0,1) \linebreak\to \R,\) i.e., the essentially uniquely determined decreasing function \(h^*\) such that \(\lambda(h^*\geq w)=\lambda(h\geq w)\) for all \(w\in \R,\) 
where \(\lambda\) denotes the Lebesgue measure on \((0,1);\) 
see e.g. \cite[Section 3.2]{Rueschendorf-2013}.
It is immediately clear that minimal elements in the Schur order are constant functions while, in general, maximal elements do not exist.

The Schur order for conditional distributions is defined by comparing conditional distribution functions in their conditioning variable with respect to the Schur order for functions as defined in \eqref{defschurfun}. To this end, denote by \(q_W\) the (generalized) \emph{quantile function} of a random variable \(W,\) i.e., \(q_W(t):=\inf\{x\in \R\mid F_W(x)\geq t\}\) for \(t\in (0,1).\)

\begin{definition}[Schur order for conditional distributions]~\\
    Let \((U,V)\) and \((U',V')\) be bivariate random vectors with \(V\eqd V'.\) Then the Schur order for conditional distributions is defined by 
\begin{align}\label{defSchurOrder}
    (V|U) \leq_S (V'|U') ~~ \colon \Longleftrightarrow ~~ F_{V|U=q_U(\cdot)}(v) \prec_S F_{V'|U'=q_{U'}(\cdot)}(v) ~~ \text{for all } v\in [0,1].
\end{align}
\end{definition}

By definition, the Schur order for conditional distributions is invariant under rearrangements of the conditioning variable. It compares the variability of conditional distribution functions in the conditioning variables in the sense of the Schur order for functions. Minimal elements characterize independence and maximal elements characterize perfect directed dependence; see \cite[Theorem 3.5]{Ansari-Fuchs-2022}. Further, the Schur order for conditional distributions has the property that, for \(U\eqd U'\) and \(V'\uparrow_{st} U',\) 
\begin{align}\label{propSchurOrder}
      (V|U)\leq_S (V'|U')  \quad \Longrightarrow \quad (U,V)\leq_{sm} (U',V'),
\end{align}
i.e., less variability of the conditional distribution function \(u\mapsto F_{V|U=u}(v)\) in the conditioning variable than \(u\mapsto F_{V'|U'=u}(v)\) in sense of \eqref{defschurfun}  for all \(v\) implies that \((U,V)\) is smaller or equal in the supermodular order than \((U',V')\), provided \(V'\) is stochastically increasing in \(U';\)
see \cite[Proposition 3.17]{Ansari-Rueschendorf-2021}.
Note that in \eqref{propSchurOrder}, there is no positive dependence assumption on \((U,V).\)
If, additionally, \(V\uparrow_{st} U,\) then the reverse direction in \eqref{propSchurOrder} also holds true; see \cite[Proposition 3.4]{Ansari-Fuchs-2022}. In this case, the Schur order is equivalent to the supermodular order, bringing us back to the setting of Lemma \ref{lemfacmod}.

The following result is a version of \cite[Corollary 4(i)]{Ansari-Rueschendorf-2023} and extends \eqref{propSchurOrder} to a vector of conditionally independent random variables. It states that
strengthening the supermodular ordering condition \eqref{lemfacmod3} to the Schur order allows dropping the positive dependence assumption \eqref{lemfacmod1} of Lemma \ref{lemfacmod} on \((X_0,X_i),\) \(i\in \{1,\ldots,d\}.\)

\begin{proposition}[Comparison of star structure based on Schur order]\label{propSch}~\\
Let \(X=(X_0,\ldots,X_d)\) and \(Y=(Y_0,\ldots,Y_d)\) be random vectors such that \(X_1,\ldots,X_d\) are conditionally independent given \(X_0\) and such that \(Y_1,\ldots,Y_d\) are conditionally independent given \(Y_0\) with \(X_0\eqd Y_0.\) Assume for all \(i\in \{1,\ldots,d\}\) that
\begin{enumerate}[(i)]
    \item\label{propSch2} \(Y_i\uparrow_{st} Y_0,\) 
    \item\label{propSch3} \((X_i|X_0)\leq_{S} (Y_i|Y_0).\)
\end{enumerate}
Then it follows that \(X\leq_{sm} Y.\) In particular, \(Y\) is PSMD. 
\end{proposition}


In the following remark, we discuss the distinctive property of star structures that enables the above generalization of Theorem \ref{thm: supermodular order of tree specifications generalized} to the Schur order.

\begin{remark}
\begin{enumerate}[(a)]
    \item 
    Since \((V|U)\leq_S (V'|U')\) is equivalent to \((U,V)\leq_{sm} (U',V')\) whenever \(U\eqd U',\) \(V\uparrow_{st} U\) and \(V'\uparrow_{st} U',\) 
    Proposition \ref{propSch} extends Lemma \ref{lemfacmod} to a broader class of conditionally independent distributions that allow a non-positive dependence structure for \(X.\) 
    Surprisingly, as we show in Example \ref{exp:d}, Proposition \nolinebreak \ref{propSch} cannot be extended from star structures to Markov tree distributions: 
    To be precise, let $T=(N,E)$ for $N=\{0,1,2,3\}$ and $E=\{(0,1),(1,2),(2,3)\}$ be a chain of \(4\) nodes. Then, we can construct Markov tree distributed random variables \(X=(X_n)_{n\in N}\) and \(Y=(Y_n)_{n\in N}\) such that 
    \begin{itemize}
        \item $(X_i,X_{i+1})\leq_{sm}(Y_i,Y_{i+1})$ for all $i\in \{0,1,2\}$,
        \item \((Y_i,Y_{i+1})\) being \(\MTP\) for all $i\in \{0,1,2\}$, and
        \item \((X_j|X_i)\leq_{S}(Y_j|Y_i)\) for all \((i,j)\in E,\)
    \end{itemize} 
    but \(X\not\leq_{lo} Y .\) Hence,  we obtain that less variability of the bivariate specifications \(\{(X_i,X_j)\}_{(i,j)\in E}\) compared to \(\{(Y_i,Y_j)\}_{(i,j)\in E}\)
    with respect to the Schur order, as in \eqref{defSchurOrder}, does \emph{not} imply that the entire vector \(X=(X_n)_{n\in N}\) is smaller than \(Y=(Y_n)_{n\in N}\) in the lower orthant order and thus neither in the supermodular order---even if the specifications \((Y_i,Y_j)_{(i,j)\in E}\) satisfy the strong positive dependence concept \(\MTP;\) see also Figure \ref{fig:counterexample markov}\,d). Consequently, Proposition \ref{propSch} on the rearrangement-based Schur order has no direct extension to Markov tree distributions.
    \item The Schur order for conditional distributions has the fundamental property that it implies an ordering of various dependence measures. A dependence measure is a functional \((X,Y)\mapsto\kappa(Y,X)\) such that (i) \(\kappa(Y,X)\in [0,1],\) (ii) \(\kappa(Y,X)=0\) if and only if \(X\) and \(Y\) are independent, and (iii) \(\kappa(Y,X)=1\) if and only if \(Y\) is perfectly dependent on \(X.\)
    For instance,
    Chatterjee's rank correlation, which has recently attracted a lot of attention in the statistics literature \cite{chatterjee2021,chatterjee2021jasa,Dette-2013}, is consistent with the Schur order for conditional distributions; see \cite{Ansari-Fuchs-2022}.
    By Proposition \ref{propSch}, the Schur order also implies comparison results with respect to the supermodular order.  Roughly speaking, large elements in the Schur order for conditional distributions lead to strong dependencies and, if the dependencies are positive, to large elements in supermodular order. 
\end{enumerate}
\end{remark}

\section[Robust hidden Markov models]{Distributional robustness in hidden Markov models}\label{secHidMarkov}

In this section, we apply our comparison results to hidden Markov models (HMMs), which form a subclass of Markov tree distributions. 
Using Theorem \ref{thm: supermodular order of tree specifications generalized} and the results presented in Section \ref{sec3} allows us to study distributional robustness of various functionals under ambiguity.
As an illustrative example, we consider the maximum of a perturbed random walk and derive uncertainty bands for its distribution function in a setting with model uncertainty.
For a comprehensive overview of hidden Markov models, including the subsequent definitions, we refer to \cite{cappe2009inference, ephraim2002hidden}, along with the literature referenced therein.

 Any HMM \((X,X^*)\) has a functional representation, known as a (general) state-space model, by
\begin{align}\begin{split}\label{defHMM}
    X_{n} &= f_n(X_{n-1},\delta_{n})\quad P\text{-almost surely for } n\in \N,\\
    X_{n}^* &= f_n^*(X_n,\varepsilon_{n}) \quad P\text{-almost surely for } n\in \N_0,
    \end{split}
\end{align}
for some measurable functions \(f_n,f^*_n\colon \R^2\to \R\) and i.i.d.\@ uniformly on $[0,1]$ distributed random variables \(\{\delta_n\}_{n\in\N},\{\varepsilon_n\}_{n\in\N_0},\) that are independent of the initial random variable $X_0$. In the language of Markov tree distributions, a HMM can equivalently be written as a sequence $Z=(Z_n)_{n\in\mathbb N_0}$ of random variables defined by $ Z_{2n}=X_{n}$ and $ Z_{2n+1}=X^*_n,$ for all $n\in\mathbb N_0$,
where \(Z\) follows a Markov tree distribution with respect to the tree 
\begin{align}\label{eq: edges HMM}
 T=(N,E), \text{ with } N = \N_0 \text{ and } E=\big\{(i,j)\in 2\mathbb N_0\times \mathbb N_0\, \vert\,  j-i\in \{1,2\}\big\};   
\end{align}
see Figure \ref{figHMM} for an illustration of the underlying tree structure.

There are two perspectives on the interpretation and application of HMMs. First, in fields such as communication theory, one can view the hidden process $X$ as a signal transmitted via a communications channel. Given the inherent noise in the channel, the receiver perceives the distortion $X^*$ and aims to reconstruct the original signal \(X\); see e.g. \cite{ephraim2002hidden, kailath1998detection}. 
Second, in many models, such as in finance \cite{Mamon-2007,Rossi-2006}, economics \cite{Hamilton-1989}, climatology \cite{Ailliot-2015} or reliability theory \cite{Smyth-1994}, the focus lies directly on the observable process $X^*$, which is driven by a (latent) factor process $X$, the hidden process. For instance, $X^*$ may describe the market price of a stock, while $X$ models an economic factor influencing price fluctuations. 
In particular, any ARMA process has a representation of the form \eqref{defHMM} with linear functions \(f\) and \(f^*;\) see \nolinebreak\cite{Akaike-1974}. 
In the sequel, we adopt the latter perspective, focusing on drawing inferences from the observations of the hidden process.

As an interesting quantity, we consider the maximum of the observations \(X_0^*,\ldots,X_n^*,\) noting that the maximum of a stochastic process is closely related to first passage times. Specifically, the maximum process exceeds a threshold \(s\) at a certain time \(n\) if and only if the first time where the underlying stochastic process enters the interval \((s,\infty)\) is less or equal than \(n.\) 
For an overview of recent literature and for various applications related to first passage times and the maximum of a perturbed random walk, we refer \cite{Araman-2005,Araman-2006,Iksanov-2016,Iksanov-2025,Iksanov-2017a,Iksanov-2017b}.

We make use of our comparison results to establish distributional robustness of \(\max_{0\leq k\leq n}\{X_k^*\}\) in stochastic order under ambiguity of \(X^*.\) 
Our findings are illustrated through uncertainty bands for the associated distribution function \(F_{\max_{0\leq k\leq n}\{X_k^*\}}.\)
Key is that the function 
\begin{align}\label{eqindmax}
    f\colon (x_0,\ldots,x_n) \mapsto \1_{\{\max \{0, x_0,\ldots,x_n\}\leq t\}}
\end{align}
is decreasing supermodular.
As a direct consequence of Theorems \ref{thm: supermodular order of tree specifications generalized}, \ref{thm: maintwo} and \ref{themaindcx}, the following result enables the comparison of various functionals (as outlined in Table \ref{table_functions}) within classes of hidden Markov models.

 \begin{corollary}[Comparison results for hidden Markov models]\label{thmHMM}~\\
    Let \((X,X^*)\) and \((Y,Y^*)\) be HMMs. Assume that \(X_n^*\uparrow_{st} X_n\) for all \(n\in \N\) and that \(X_{n+1}\uparrow_{st} X_n,\) \(Y_n\uparrow_{st} Y_{n+1}\) as well as \(Y_n^*\uparrow_{st} Y_n\) for all \(n\in \N_0.\)
\begin{enumerate}[(i)]
    \item \label{thmHMM1} If \((X_{n},X_{n+1})\leq_{sm} (Y_{n},Y_{n+1})\) and \((X_n,X_n^*)\leq_{sm} (Y_n,Y_n^*)\) for all \(n\in \N_0,\) then \((X,X^*)\leq_{sm} (Y,Y^*).\)
    \item \label{thmHMM2} Assume that \(\overline{\Ran(F_{X_n})} = \overline{\Ran(F_{Y_n})}\) and \(\overline{\Ran(F_{X_n^*})} = \overline{\Ran(F_{Y_n^*})}\) for all \(n\in \N_0.\) If \(C_{X_n,X_{n+1}}\leq_{lo} C_{Y_n,Y_{n+1}},\) \(C_{X_n,X_{n}^*}\leq_{lo} C_{Y_n,Y_{n}^*},\) \(X_n\leq_{st} Y_n\) (resp. \(\geq_{st}\)), and \(X_n^*\leq_{st} Y_n^*\) (resp. \(\geq_{st}\)) for all \(n\in \N_0,\) then \((X,X^*)\leq_{ism} (Y,Y^*)\) (resp. \(\leq_{dsm}\)).
    \item Assume that all marginal distribution functions are continuous. If \(C_{X_n,X_{n+1}}\linebreak \leq_{lo} C_{Y_n,Y_{n+1}}\) and \(C_{X_n,X_{n}^*}\leq_{lo} C_{Y_n,Y_{n}^*}\) with \(C_{Y_n,Y_{n+1}}\) and \(C_{Y_n,Y_{n}^*}\) \(\MTP,\) and if \(X_n\leq_{cx} Y_n\) and \(X_n^*\leq_{cx} Y_n^*\) for all \(n\in \N_0,\) then \((X,X^*)\leq_{dcx} (Y,Y^*).\)
\end{enumerate}
\end{corollary}

\begin{remark}
    The SI assumptions in the above result are not very restrictive: 
Typically, there is a strong positive dependence between the latent variable \(X_n\) and its perturbed observation \(X_n^*.\) 
A standard example, where the SI condition \(X_n^*\uparrow_{st} X_n\) is fulfilled, is the additive error model \(f_n^*(x,z) = x + \sigma z\) for some \(\sigma \geq 0.\) As with many models for stochastic processes, outcomes that are close in time exhibit strong positive dependence. Consequently, also the SI assumptions \(X_{n+1}\uparrow_{st} X_n\) and \(Y_n\uparrow_{st} Y_{n+1}\) are quite natural. Note that Corollary \ref{thmHMM}\eqref{thmHMM1} can also be interpreted as an extension of Lemma \ref{lemfacmod} to a multi-factor model, where the hidden process serves as a common factor process.
\end{remark}

\begin{figure}[t]
	\begin{tikzpicture}[scale=.85, transform shape, remember picture]
         
        \begin{scope}[shift={(0,0)}] 
			\tikzstyle{every circle node} = []
			\node[circle] (a) at (0, 0) {\(X_0\)};
			\node[circle] (b) at +(0: 1.7) {\(X_1\)};
			\node[circle] (c) at +(0: 3.4) {\(X_2\)};
			\node[circle] (d) at +(0: 5.1) {\(X_3\)};
            \node[circle] (ya) at (a |- 0,-1.5) {\(X_0^*\)};
            \node[circle] (yb) at (b |- 0,-1.5) {\(X_1^*\)};
            \node[circle] (yc) at (c |- 0,-1.5) {\(X_2^*\)};
            \node[circle] (yd) at (d |- 0,-1.5) {\(X_3^*\)};
			\draw [->] (a) -- (b) node[pos=0.5,above]{};
			\draw [->] (b) -- (c) node[pos=0.5,above]{};
			\draw [->] (c) -- (d) node[pos=0.5,above]{};
            \draw [-] (a) -- (ya) node[pos=0.5,right]{};
			\draw [->] (b) -- (yb) node[pos=0.5,right]{};
			\draw [->] (c) -- (yc) node[pos=0.5,right]{};
            \draw [->] (d) -- (yd) node[pos=0.5,right]{\(\ \ \ \cdots\)};
		\end{scope}

  \begin{scope}[shift={(7.8,0)}] 
			\tikzstyle{every circle node} = []
			\node[circle] (a) at (0, 0) {\(Y_0\)};
			\node[circle] (b) at +(0: 1.7) {\(Y_1\)};
			\node[circle] (c) at +(0: 3.4) {\(Y_2\)};
			\node[circle] (d) at +(0: 5.1) {\(Y_3\)};
            \node[circle] (ya) at (a |- 0,-1.5) {\(Y_0^*\)};
            \node[circle] (yb) at (b |- 0,-1.5) {\(Y_1^*\)};
            \node[circle] (yc) at (c |- 0,-1.5) {\(Y_2^*\)};
            \node[circle] (yd) at (d |- 0,-1.5) {\(Y_3^*\)};
			\draw [<-] (a) -- (b) node[pos=0.5,above]{};
			\draw [<-] (b) -- (c) node[pos=0.5,above]{};
			\draw [<-] (c) -- (d) node[pos=0.5,above]{};
            \draw [->] (a) -- (ya) node[pos=0.5,right]{};
			\draw [->] (b) -- (yb) node[pos=0.5,right]{};
			\draw [->] (c) -- (yc) node[pos=0.5,right]{};
            \draw [->] (d) -- (yd) node[pos=0.5,right]{\(\ \ \ \cdots\)};
		\end{scope}
    \end{tikzpicture}
    \caption{The graphs illustrate the SI conditions on the hidden Markov processes \((X,X^*)=(X_n,X_n^*)_{n\in \N_0}\) and \((Y,Y^*)=(Y_n,Y_n^*)_{n\in \N_0}\) which lead to the comparison results in Corollary \ref{thmHMM}, where an arrow \(U \to V\) indicates \(V\uparrow_{st} U.\) The root of the underlying tree corresponds to the variable \(X_0\) and \(Y_0,\) respectively. Note that only the SI condition between \(X_0\) and \(X_0^*\) can be dropped, see Proposition \ref{propSIass}.
 }
	\label{figHMM}
\end{figure}

\subsection{A classical random walk}\label{subsecnrw}

As a starting point, we model the hidden process as a classical random walk with independent and standard normal increments, i.e.,
\begin{align}\label{eq: random walk}
    X_n = \sum_{i=1}^n \xi_i, \quad X_0:=0,
\end{align}
for i.i.d.\@ \(N(0,1)\)-distributed random variables \(\{\xi_i\}_{i\in \N}\).
Then, \(X = (X_n)_{n\in\N_0}\) is a discrete-time Gauss-Markov process specified by the means \(\E X_n=0\) and the covariances \(\Cov(X_i,X_j) = \min\{i,j\}\) for all \(i,j,n\in\N_0.\) Equivalently, by Proposition \ref{prop: existence of markov tree dirstribution}, \(X\) can be considered as a discrete-time Markov process with marginal and bivariate copula specifications 
\begin{align}\label{eqnoisyrwcor0}
    X_n &\sim N(0,n),\\
    \label{eqnoisyrwcor}C_{X_n,X_{n+1}} &= C_{\rho_{n,n+1}}^{\Gauss}, \quad \rho_{n,n+1} := \Cor(X_n,X_{n+1}) = \sqrt{n/(n+1)},
\end{align}
where $C_{\rho}^{\Gauss}$ denotes the Gaussian copula with correlation parameter \(\rho.\) For \(\rho\geq 0,\)  it is well known that $C_{\rho}^{\Gauss}$ is CI, which is equivalent to 
\begin{align}\label{eq: random walk2}
    X_{n+1}\uparrow_{st} X_n \quad \text{and} \quad X_n\uparrow_{st} X_{n+1} \quad \text{for all } n\in \N_0.
\end{align}
The random walk \(X\) is represented by the red paths in Figure \ref{fig:paths}, and the distribution function of the associated maximum \(\max_{k\leq d}\{X_k\}\) is illustrated by the upper (red) graph in Figure \ref{fig:HM_plot} for \(d = 200.\)

\subsection[Perturbed random walk]{A perturbed random walk under model uncertainty}\label{secpertrw}

In the sequel, we model the observation process \(X^* = (X_n^*)_{n\in \N_0}\) by incorporating noise in various ways, and analyze how the noise affects the maximum observations.

\begin{example}[Gaussian error model]\label{exGaussobs}
    As a standard model for noise, we choose independent additive Gaussian errors to model the perturbed random walk \(X^* = (X_n^*)_{n\in \N},\) i.e.,
    \begin{align}\label{eqGausssobs}
        X_n^* = X_n + \sigma_n \varepsilon_n, \quad \sigma_n\in [0,\overline{\sigma}_n],
    \end{align}
    for i.i.d. standard normal random variables \((\varepsilon_i)_{i\in \N}\) independent of \(X\) and for some deterministic \(\overline{\sigma}_n > 0.\) The parameter \(\sigma_n\) is assumed to be ambiguous in the sense that it is only known to lie within the interval \([0,\overline{\sigma}_n].\)
    For the family of models in \eqref{eqGausssobs}, it can easily be verified that
    \begin{align}
        \label{eqGausssobs1} X_n^* &\sim N(0,n+\sigma_n),\\
        \label{eqGausssobs2a} C_{X_n,X_n^*} &= C_{\rho_n}^{\Gauss}, \quad \text{with}\quad \rho_n := \sqrt{n/(n+\sigma_n)}.
    \end{align}
    As a direct consequence of \eqref{eqGausssobs}, the correlation of \((X_n,X_n^*)\) decreases with \(\sigma_n\) and lies within the interval \([\underline{\rho}_n,1]\), for 
   \(
        \underline{\rho}_n := \sqrt{n/(n+\overline{\sigma}_n)}.
    \)
    We use our comparison results to show that, for fixed \(d\in \N,\) \(\max_{n\leq d} \{X_n^*\} \) increases in \(\sigma_n\) with respect to the stochastic order. To this end, let
    \(Y_n^* := X_n + \sigma_n'\varepsilon_n\) for some fixed \(\sigma_n'\leq \sigma_n.\) Then we have
    \begin{align}\label{eqGausssobs2}
    F_{X_n^*}(x) &\geq F_{Y_n^*}(x) \quad \text{for all } x\leq 0,\\
        \label{eqGausssobs3} F_{X_n^*}(x) &\leq F_{Y_n^*}(x) \quad \text{for all } x\geq 0,\\
        \label{eqGausssobs4} C_{X_n,X_n^*} &\leq_{lo} C_{Y_n,Y_n^*}
    \end{align}
    for all \(n\in \N.\) For \eqref{eqGausssobs2} and \eqref{eqGausssobs3}, we use the fact that the univariate normal distribution with zero mean is increasing in its variance parameter with respect to the convex order. Note that \eqref{eqGausssobs3} is equivalent to \(\max\{X_n^*,0\} \geq_{st} \max\{Y_n^*,0\},\) see also Remark \ref{rem38}\eqref{rem38b}. 
    For \eqref{eqGausssobs4}, we use that the Gaussian copula family is increasing in its parameter with respect to the lower orthant order; see e.g. \cite[Table 5]{Ansari-Rockel-2024}. 
    Since the function \(f\) in \eqref{eqindmax} is decreasing supermodular, it follows from Corollary \ref{thmHMM}\eqref{thmHMM2} that
    \begin{align*}
        P(\max_{0\leq n\leq d} \{X_n^*\} \leq t ) &= \E [\1_{\max_{n\leq d} \{X_n^*,0\}\leq t}] \\
        &\leq \E [\1_{\max_{n\leq d} \{Y_n^*,0\}\leq t}] = P(\max_{0\leq n\leq d} \{Y_n^*\} \leq t )
    \end{align*}
    for all \(t\in \R,\)
    using \eqref{eqGausssobs3}, \eqref{eqGausssobs4}, and \eqref{eq: random walk2} as well as \(X_k^*\uparrow_{st} X_k\) and \(Y_k^*\uparrow_{st} Y_k\) for all \(k\in \N_0.\) 
    Hence, defining \(\underline{X}^* = (\underline{X}_n^*)_{n\in \N_0}\) and \(\overline{X}^* = (\overline{X}_n^*)_{n\in \N_0}\) by
    \begin{align}\label{defboundsGau}
        \underline{X}_n^* := X_n \quad \text{and} \quad \overline{X}_n^* := X_n + \overline{\sigma}_n \varepsilon_n,
    \end{align}
    we obtain \(P(\max_n\{\underline{X}_n^*\}>t) \leq P(\max_n\{X_n^*\}>t) \leq P(\max_n \{\overline{X}_n^* > t\})\) for all \(t.\) 
    This implies distributional robustness of the extreme order statistics in stochastic order. 
    While \(\underline{X}^*\) represents the setting without error, the process \(\overline{X}^*\) models the scenario with the largest observation errors, as determined by the parameter bounds \(\overline{\sigma}_n.\)
    In Figure \ref{fig:HM_plot}, we illustrate the uncertainty bands for \(\overline{\sigma}_n = 3\) (left plot) and \(\overline{\sigma}_n = 0.3 n\) 
    (right plot).
\end{example}

\begin{figure}[tb]
	\centering

	\begin{minipage}{0.47\textwidth}
		\centering
		\includegraphics[width=\linewidth]{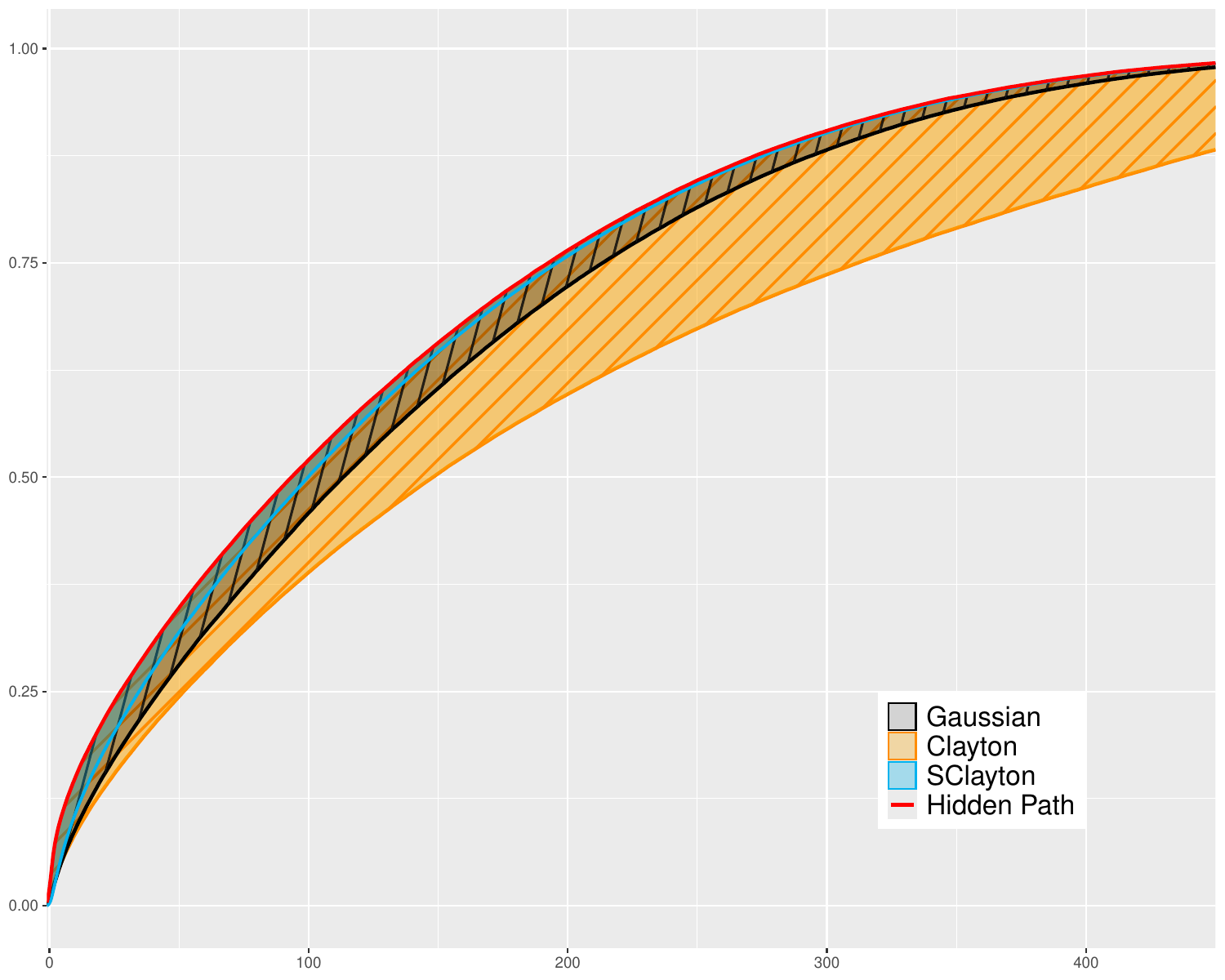}
	\end{minipage}
	\hspace{0.02\textwidth}
	\begin{minipage}{0.47\textwidth}
		\centering
		\includegraphics[width=\linewidth]{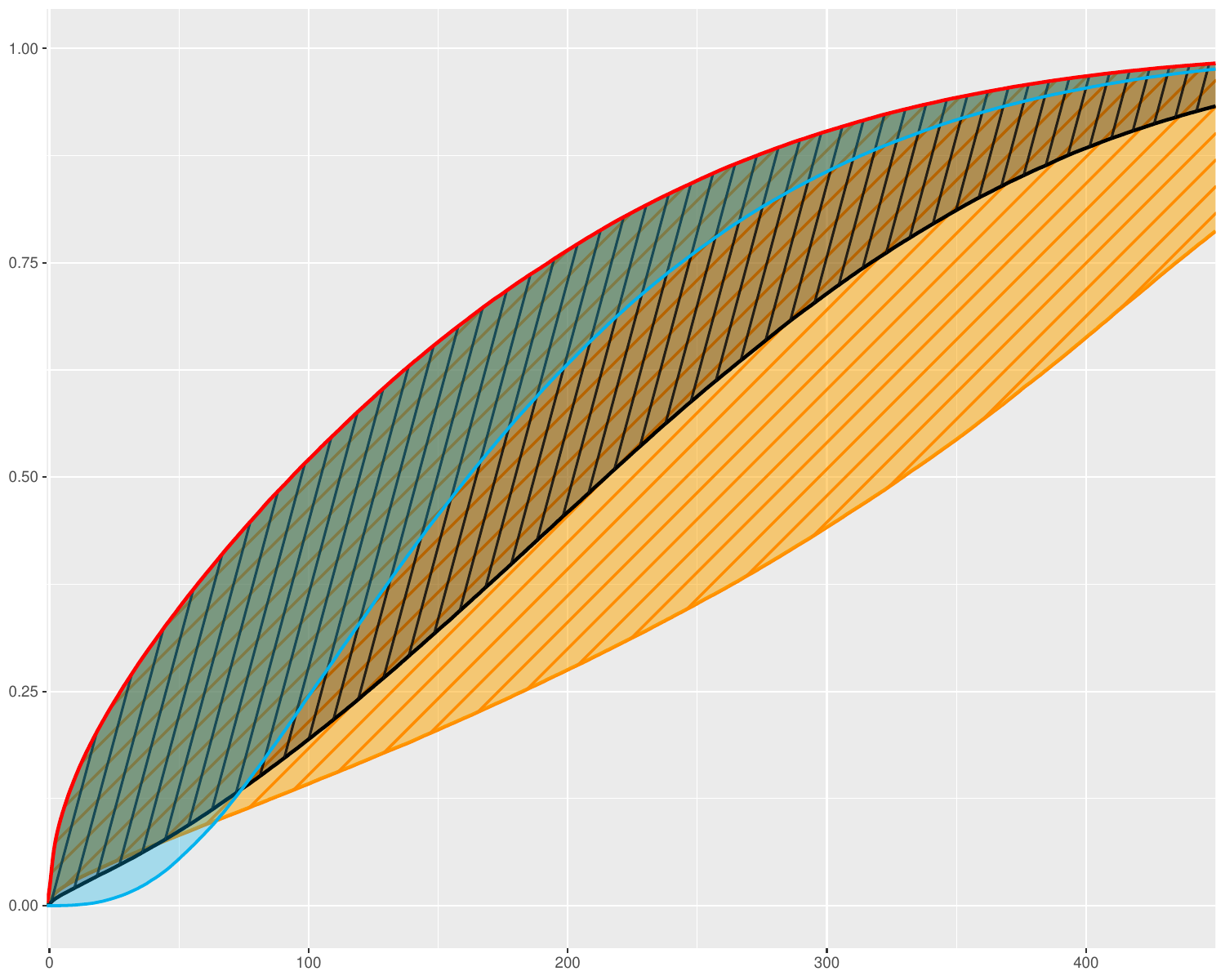}
	\end{minipage}
 \caption{Uncertainty bands of the distribution function $t\mapsto F_{\max\{X_0^*,\ldots,X_d^*\}}(t)$ in the classes of hidden Markov models considered in Theorem \ref{thedisrw} and Examples \ref{exClaycem} and \ref{exsClayton}, where \(d = 200\) and \(\sigma_n = 3\) (left plot) resp. \(\sigma_n = 0.3 n\) (right plot).
 The upper (red) graph is the distribution function of \(\max_{n\leq d}\{\underline{X}_n^*\}\), which corresponds to the hidden Markov model without observation error.
 The black/orange/blue graph corresponds to the distribution function of \(\max_{n\leq d}\{\overline{X}^*_n\}\) in the Gaussian/Clayton/survival Clayton setting, considering maximal uncertainty in Example \ref{exGaussobs}/\ref{exClaycem}/\ref{exsClayton}.
 }
		\label{fig:HM_plot}
\end{figure}

In Example \ref{exGaussobs}, the hidden Markov model \((X,X^*)\) follows a multivariate normal model, for which comparison results are well known; see \cite{Mueller-2001}. 
    However, the bounds \(\underline{X}^* = (\underline{X}_n^*)_{n\in \N}\) and \(\overline{X}^* = (\overline{X}_n^*)_{n\in \N}\) in \eqref{defboundsGau} are applicable to a much more general setting involving the marginal and copula ambiguity sets 
    \begin{align}\label{remGaudus1}
        \cF_n &:= \{F \text{ continuous cdf}\mid F_{N(0,n)}(t) \geq F(t) \geq F_{N(0,n+\overline{\sigma}_n)}(t) ~\forall t\geq 0\},\\
        \label{remGaudus2}\cC_n &:= \{ C \text{ bivariate copula} \mid C \text{ SI and } C\geq_{lo} C_{\underline{\rho}_n}^{Ga}, ~\underline{\rho}_n = \sqrt{n/(n+\overline{\sigma}_n)}\}.
    \end{align}
    Here, \(F_{N(0,n)}\) is the distribution function of \(\underline{X}_n^*\), \(F_{N(0,n+\overline{\sigma}_n)}\) is the distribution function of \(\overline{X}_n^*,\) and \(C_{\underline{\rho}_n}^{Ga}\) is the copula of \((X_n,\overline{X}_n^*).\) 
    The following result summarizes these findings. 
    

\begin{theorem}[Distorted random walk under ambiguity]\label{thedisrw}~\\
Consider the processes \(\underline{X}^*\) and \(\overline{X}^*\) in \eqref{defboundsGau}. Then, for any hidden Markov model \((X,X^*) = (X_n,X_n^*)_{n\in \N_0}\) specified by the random walk \(X\) in \eqref{eq: random walk} and by 
    \(F_{X_n^*}\in \cF_n\) and \(C_{X_n,X_n^*}\in \cC_n\) for all \(n\in \N,\) 
we have
\begin{align}\label{eqthedisrw1}
    \max_{0\leq n\leq d}\{\underline{X}_n^*\} \leq_{st} \max_{0\leq n\leq d}\{X_n^*\} \leq_{st} \max_{0\leq n\leq d}\{\overline{X}_n^*\}.
\end{align}
\end{theorem}

\begin{remark}
    As an extension of the setting in \eqref{eqGausssobs}, Corollary \ref{thmHMM} is also applicable when the observation process is modeled as \(X_n^* = f_n(X_n,\varepsilon_n)\) for some componentwise increasing function \(f\colon \R^2\to \R.\) Then, \(X_n^*\) is still SI in \(X_n\,;\) we refer to Examples \ref{exClaycem} and \ref{exsClayton} for this modified setting.
    Moreover, it is also possible to model the hidden process under ambiguity. For instance, the parameter \(\rho_{n,n+1}\) in \eqref{eqnoisyrwcor} may be assumed to belong to a subinterval of \([0,1].\) Recall that the lower orthant order, the SI property, and copulas are invariant under increasing transformations; see \cite[Theorem 6.G.3]{Shaked-Shanthikumar-2007} and \cite[Proposition 2.12 and Theorem 3.3]{Cai-2012}. Consequently, the setting in Theorem \ref{thedisrw} also applies to transformed processes such as \((\exp(X_n^*))_{n\in \N_0}.\) 
\end{remark}

In Example \ref{exGaussobs} above, the perturbations are modeled as additive errors independent of \(X_n.\) In the sequel, we investigate the behavior of the observation process when the errors depend on the values of the hidden process \(X_n.\)
As an illustrative example, we use Clayton and survival Clayton copulas to specify the dependencies of \((X_n,X_n^*)\), where
\begin{align}\label{eq: Clayton copula}
    C_\theta^{\Cl}(u,v) &:= \max\{(u^{-\theta}+v^{-\theta}-1)^{-1/\theta},0\} \quad\text{ and }\\
    \label{eq: sClayton copula}\quad C_\theta^{\SCl}(u,v) &:= 1-u-v+C_\theta^{\Cl}(u,v), \quad (u,v)\in [0,1]^2
\end{align}
denote the Clayton and survival Clayton copula with parameter \(\theta\in (0,\infty).\) For \(\theta=0\) the Clayton and survival Clayton copula are defined as \(C_\theta^{\Cl}(u,v) = C_\theta^{\Cl}(u,v) = uv\), representing independence. For \(\theta\to\infty,\) they model comonotonicity.
While Clayton copulas exhibit lower tail dependencies, their associated survival copulas exhibit upper tail dependencies; see e.g. \cite{Ansari-Rockel-2024} and Figure \ref{fig:copula_plot}.
\begin{figure}[t]
	\centering

	\begin{minipage}{0.3\textwidth}
		\centering
		\includegraphics[width=\linewidth]{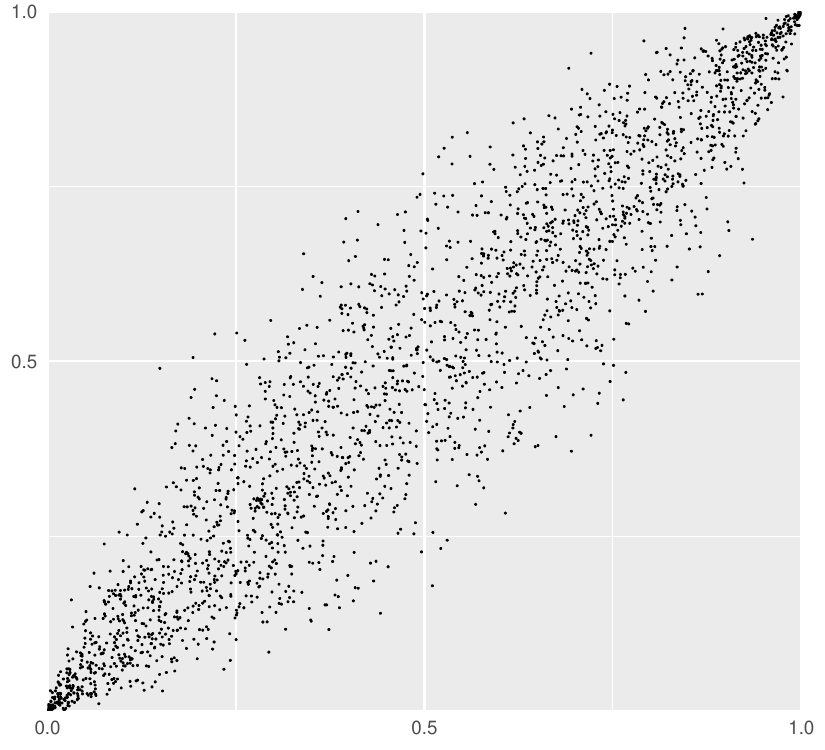}
	\end{minipage}
	\hspace{0.02\textwidth}
	\begin{minipage}{0.3\textwidth}
		\centering
		\includegraphics[width=\linewidth]{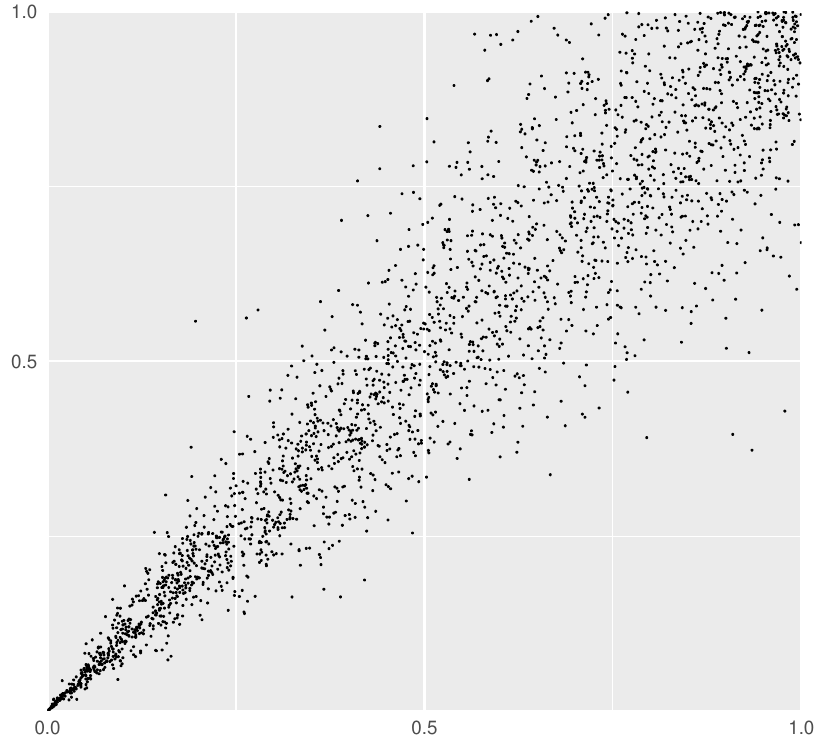}
	\end{minipage}
    \hspace{0.02\textwidth}
	\begin{minipage}{0.3\textwidth}
		\centering
		\includegraphics[width=\linewidth]{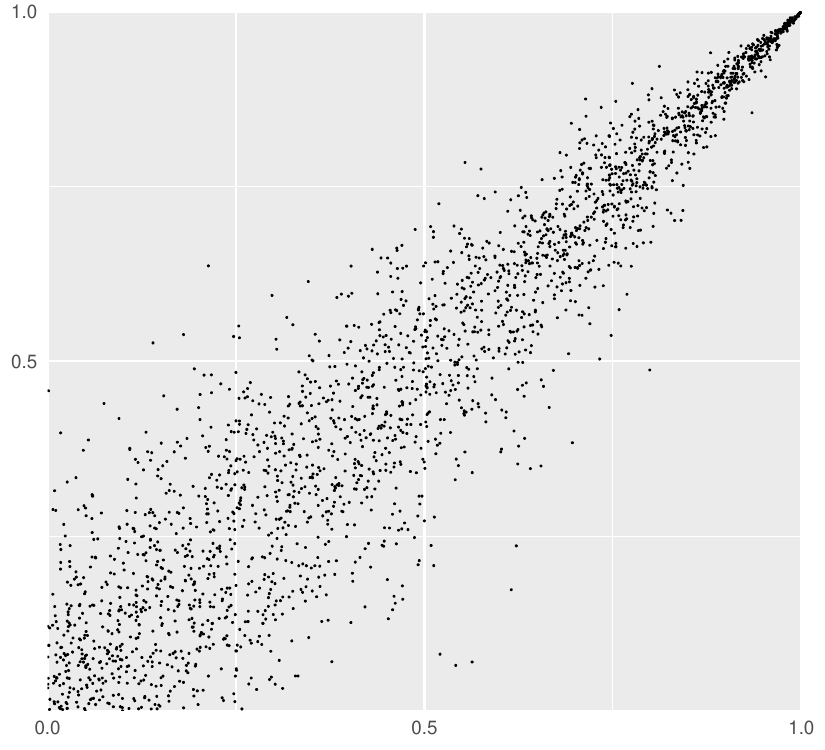}
	\end{minipage}
 \caption{Samples of a Gaussian copula (left), Clayton copula (mid), and survival Clayton copula (right), each having Kendall's tau value \(\tau = 0.795\) which corresponds to the parameter $\rho=\sqrt{9/10}$ for the Gaussian copula and $\theta=7.764$ for the Clayton and survival Clayton copula. The plots indicate that the Clayton copulas exhibit lower tail-dependencies, while the survival Clayton copulas have upper tail-dependencies.}
		\label{fig:copula_plot}
\end{figure}

\begin{example}[Clayton copula error model]\label{exClaycem}
  To analyze the influence of lower tail dependencies on \(\max_{k\leq n}\{X_k^*\}\), we now consider Clayton copulas as
dependence specifications of the observations process. The hidden process \(X\) is modeled as before by the random walk in \eqref{eq: random walk}. For better comparability with the setting in Example \ref{exGaussobs}, we assume that \(X^*\) is specified by the same marginals as in the Gaussian setting (see \eqref{eqGausssobs1}). However, \(X_n^*\) and \(X_n\) are now coupled by a Clayton copula, i.e.,
\begin{align}
    \label{eqClayton1} X_n^* &\sim N(0,n+\sigma_n), \quad \sigma_n\in [0,\overline{\sigma}_n], \\
        C_{X_n,X_n^*} &= C_{\theta_n^*}^{\Cl}, \quad \text{with}\quad \theta^* \in [\underline{\theta}_n,\infty).
\end{align}
For modeling a similar strength of positive dependence, we choose the lower bound for the parameter interval as \(\underline{\theta}_n := \theta(\underline{\rho}_n)\), where
\begin{align}\label{eq: Beta transform }
        \theta(\rho):=\frac{2\arcsin(\rho)}{\pi-2\arcsin(\rho)}
    \end{align}
is the parameter \(\theta\) such that the Clayton copula \(C_\theta^{\Cl}\) has the same Kendall's tau value as the Gaussian copula with parameter \(\rho\,;\) see e.g. \cite[Table 6]{Ansari-Rockel-2024}. Recall that \(\underline{\rho}_n\) is specified in Example \ref{exGaussobs}.
It is well known that any Clayton copula \(C_{\theta}^{\Cl}\) is CI for every \(\theta\geq 0\), and that the Clayton copula family is \(\leq_{lo}\)-increasing in its parameter; see e.g. \cite[Table 3]{Ansari-Rockel-2024}. 
Thus, specifying \(\underline{X}^* = (\underline{X}_n^*)_{n\in \N}\) and \(\overline{X}^* = (\overline{X}_n^*)_{n\in \N}\) by 
\begin{align}
    \underline{X}_n^* &= X_n,\\
    F_{\overline{X}_n^*} &= F_{N(0,n+\overline{\sigma}_n)} \quad \text{and} \quad C_{X_n,\overline{X}_n^*} = C_{\underline{\theta}_n}^{\Cl},
\end{align}
we can derive similar comparison results as in Theorem \ref{thedisrw}, where now the Gaussian copula in \eqref{remGaudus2} is replaced by the Clayton copula \(C_{\underline{\theta}_n}^{\Cl}.\) The uncertainty bands for the distribution function of \(\max_{k\leq n}\{X_k^*\}\) are visualized in Figure \ref{fig:HM_plot}.
\end{example}

\begin{figure}[t]
	\centering
	\begin{minipage}{0.9\textwidth}
		\centering
		\includegraphics[width=\linewidth]{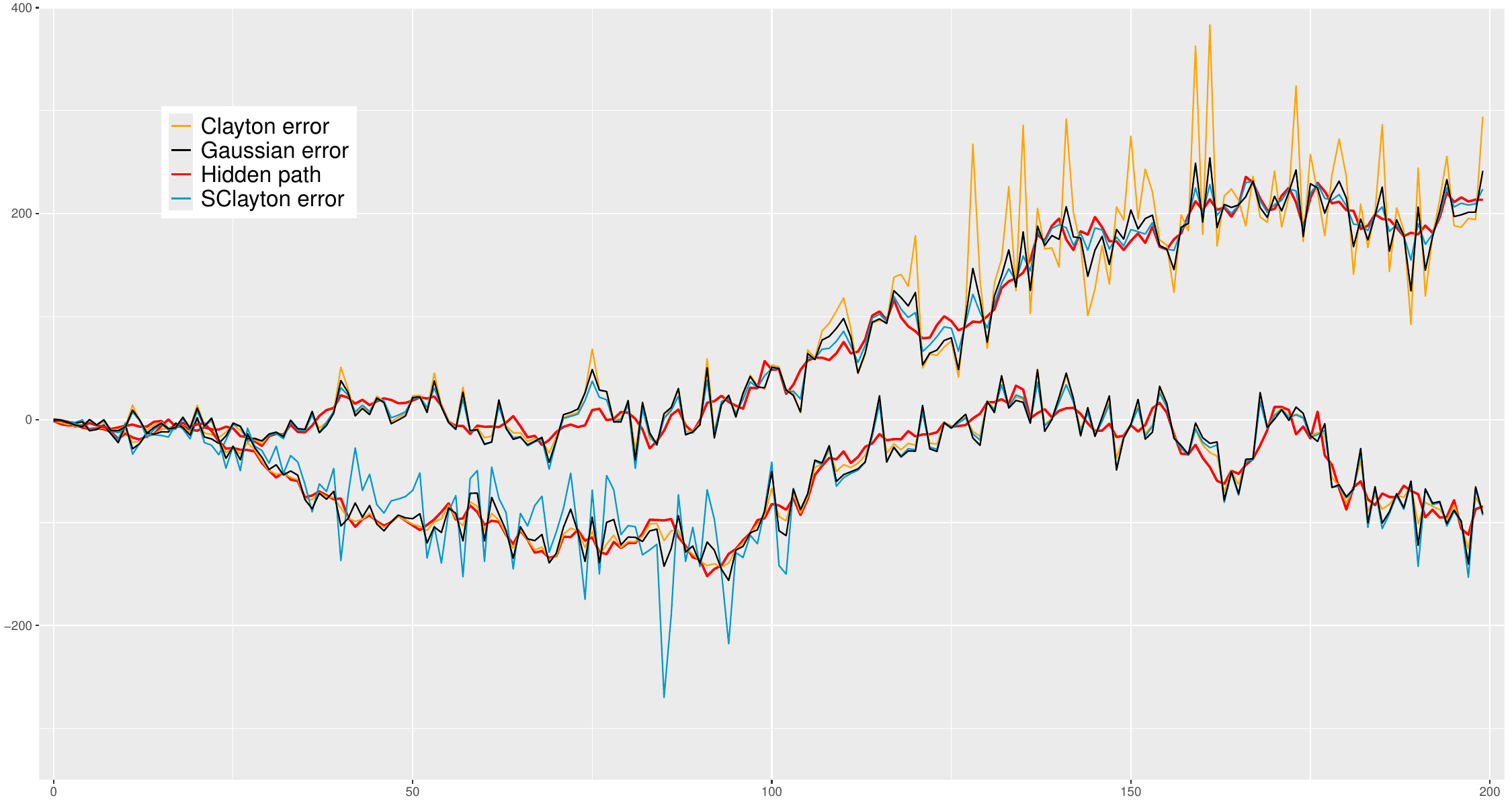}
	\end{minipage}
 \caption{Samples from the hidden process \(X=(X_n)_{n\in \N_0}\) and the observations process \(\underline{X}^* = (\underline{X}^*_n)_{n\in \N_0}\) for \(n\in [0,200]\) in various settings: The red graphs represent two sample paths of the hidden process in \eqref{eq: random walk} (which coincides with the observation process \(\overline{X}^*\) in the absence of errors). The black/orange/blue graphs model the observation \(\underline{X}^*\) in the Gaussian/Clayton/survival Clayton setting under maximal uncertainty, as specified in Example \ref{exGaussobs}/\ref{exClaycem}/\ref{exsClayton}. For this plot, we set \(\overline{\sigma}_n = 3\) for all \(n.\)}
		\label{fig:paths}
\end{figure}

As illustrated in Figure \ref{fig:paths}, \(X_n^*\) exhibits a strong dependence on \(X_n\) for small values of \(X_n.\) Conversely, when \(X_n\) attains large values, \(X_n^*\) becomes less dependent on \(X_n\), resulting in significant fluctuations. Consequently, in the Clayton copula setting, the uncertainty interval for \(F_{\max_{k\leq n}\{X_k^*\}}(t)\) is wider for \(t \gg 0\) (see Figure \ref{fig:HM_plot}), leading to greater uncertainty in the first-passage times.

In the following example, we analyze the influence of upper tail dependencies between the hidden process and the observations. For simplicity, these dependencies are modeled using survival Clayton copulas.

\begin{example}[Survival Clayton copula error model]\label{exsClayton}
    To analyze the influence of upper tail dependencies, we consider a setting similar to Example \ref{exClaycem}, replacing the Clayton copula with the survival Clayton copula defined in \eqref{eq: sClayton copula}. By straightforward symmetry arguments, survival Clayton copulas are also CI and \(\leq_{lo}\)-increasing in their parameter. Hence, the setting from Example \ref{exClaycem} similarly applies to survival Clayton copulas, allowing us to derive uncertainty bands for \(F_{\max_{k\leq n}\{X_k^*\}}\) in an analogous way.
\end{example}

Figure \ref{fig:paths} also illustrates the behavior of the perturbed random walk in the survival Clayton copula setting. Large values of \(X_n\) are strongly coupled with \(X_n^*,\) whereas the observation process exhibits greater fluctuations for small values of \(X_n.\) Consequently, there is less uncertainty in the distribution of \(\max_{k\leq n}\{X_k^*\}\) (see Figure \ref{fig:HM_plot}) and, therefore, also less ambiguity in the first-passage time process.

\section{Counterexamples}\label{seccex}

In this section, we provide various examples that demonstrate the generality of our results.

The following example shows that the SI assumptions on \(X\) and \(Y\) in 'opposite order' in Proposition \ref{propMKprocess} cannot be replaced by SI in the 'same order', i.e., there is no version of Proposition \ref{propMKprocess} under the assumptions that \(X_{i+1}\uparrow_{st} X_i\) and \(Y_{i+1}\uparrow_{st} Y_i\,;\) see also \cite[Example 4.4]{fang1994decrease}. More specifically, we show that even with two additional SI assumptions, as illustrated in Figure \ref{fig:counterexample markov}\,a), no supermodular comparison result can be established.

\begin{example}[Condition \eqref{propMKprocess2} in Proposition \ref{propMKprocess} cannot be replaced by \(Y_{i+1}\uparrow_{st} Y_i\)]\label{exp:a}~\\
Let $T=(N,E)$, $N=\{0,1,2\}$, $E=\{(0,1),(1,2)\}$, be a tree on \(3\) nodes.
    Consider the doubly stochastic matrices $a^{01},a^{12},b^{01},b^{12}\in\mathbb R^{3\times 3}$ given by
    \begin{align*}
    a^{01} = b^{01} =
    \frac{1}{30}
    \begin{pmatrix}
    4 & 4 & 2 \\
    3 & 4 & 3 \\
    3 & 2 & 5 \\
    \end{pmatrix},\quad
    a^{12} = \frac{1}{30}
    \begin{pmatrix}
    4 & 4 & 2 \\
    4 & 3 & 3 \\
    2 & 3 & 5 \\
    \end{pmatrix},\quad
    b^{12} =\frac{1}{30}
    \begin{pmatrix}
    5 & 4 & 1 \\
    3 & 3 & 4 \\
    2 & 3 & 5 \\
    \end{pmatrix}.
    \end{align*}
    Let $X=(X_1,X_2,X_3)$ and $Y=(Y_1,Y_2,Y_3)$ be random vectors which follow a Markov tree distribution with respect to the tree \(T\)
    specified by the bivariate distributions given by 
    \begin{align*}
        P[X_i=k, X_{i+1}=l]=a^{ii+1}_{kl}\text{ and } P[Y_i=k, Y_{i+1}=l]=b^{ii+1}_{kl}
    \end{align*}
    for \(k,l\in \{0,1,2\}\) and \(i\in \{0,1\}.\)
    All marginal distributions of the vectors $X$ and $Y$ are uniform on \(\{0,1,2\},\) i.e., \(X_i\eqd Y_i\sim U(\{0,1,2\})\) for \(i\in \{0,1,2\}.\)
    Further, for \(i\in \{0,1\},\) it holds that \((X_i,X_{i+1})\leq_{lo} (Y_i,Y_{i+1})\) and thus, due to identical marginals, \((X_i,X_{i+1})\leq_{sm} (Y_i,Y_{i+1})\,;\) see \eqref{eqbivord}.
    Moreover, 
    \begin{align*}
        X_{1}&\uparrow_{st} X_0,\phantom{Y_0,Y_2} X_{2}\uparrow_{st}X_1,\phantom{Y_0,Y_2} X_{1}\uparrow_{st}X_2, \\
        Y_{1}&\uparrow_{st} Y_0,\phantom{X_0,X_2} Y_{2}\uparrow_{st} Y_1,\phantom{X_0,X_2} Y_{1}\uparrow_{st}Y_2,
    \end{align*}
    i.e., \(X\) and \(Y\) are \textnormal{CIS} (see Lemma \ref{lem:CIStrees}), and the bivariate subvectors $(X_1,X_2)$ and $(Y_1,Y_2)$ are CI. However, \(Y\) violates 
    condition \eqref{propMKprocess2} of Proposition \ref{propMKprocess} as well as condition \eqref{thm12} of Theorem \ref{thm: supermodular order of tree specifications generalized} since \(Y_0\not\uparrow_{st} Y_1.\) 
    For the lower orthant set $\Theta\coloneqq (-\infty,1]^3$, we obtain 
    \begin{align*}
       P^X(\Theta) = 3\!\!\sum_{i,j,k=1}^2  a^{01}_{ij}a^{12}_{jk}=\frac{1.12}{3}>   \frac{1.11}{3}=3\!\!\sum_{i,j,k=1}^2 b^{01}_{ij}b^{12}_{jk}=P^Y(\Theta).
    \end{align*}
    This shows that $X\nleq_{lo}Y$ and thus, due to \eqref{smpdo}, $X\nleq_{sm}Y$. Hence, condition \eqref{propMKprocess2} of Proposition \ref{propMKprocess} cannot be replaced by \(Y_{i+1}\uparrow_{st} Y_i.\) 
    Similarly, this example highlights the importance of the assumption $Y_i\uparrow_{st}Y_j$ for all $(i,j)\in E$ with $j\notin  L$ in Theorem \ref{thm: supermodular order of tree specifications generalized}.
\end{example}

\begin{figure}[t]
	\begin{tikzpicture}[scale=.75, transform shape]
		\tikzstyle{every circle node} = []

		\begin{scope}[shift={(-1.5,0)}]
			\node[circle] (a) at (0, 0) {\(X_0\)};
			\node[circle] (b) at +(0: 1.7) {\(X_1\)};
			\node[circle] (c) at +(0: 3.4) {\(X_2\)};
			\node[circle] (r) at +(-1,0) {a)};
            \draw[->] (a) -- (b);
		      \draw [->] (b) -- (c);
		\end{scope}
        
		\begin{scope}[shift={(-1.5,-0.8)}]
			\node[circle] (a) at (0, 0) {\(Y_0\)};
			\node[circle] (b) at +(0: 1.7) {\(Y_1\)};
			\node[circle] (c) at +(0: 3.4) {\(Y_2\)};
			\draw [->] (b) -- (a);
			\draw [->] (c) -- (b);
		\end{scope}
  
		\begin{scope}[shift={(4,0)}]
			\node[circle] (a) at (0, 0) {\(X_1\)};
			\node[circle] (b) at +(0: 1.7) {\(X_0\)};
			\node[circle] (c) at +(0: 3.4) {\(X_2\)};
			\node[circle] (r) at +(-1,0) {b)};
            \draw[<-] (a) -- (b);
		      \draw [->] (b) -- (c);
		\end{scope}
        
		\begin{scope}[shift={(4,-0.8)}]
			\node[circle] (a) at (0, 0) {\(Y_1\)};
			\node[circle] (b) at +(0: 1.7) {\(Y_0\)};
			\node[circle] (c) at +(0: 3.4) {\(Y_2\)};
			\draw [<-] (a) -- (b);
			\draw [->] (b) -- (c);
		\end{scope}

		\begin{scope}[shift={(9.5,0)}]
			\node[circle] (a) at (0, 0) {\(X_0\)};
			\node[circle] (b) at +(0: 1.7) {\(X_1\)};
			\node[circle] (c) at +(0: 3.4) {\(X_2\)};
			\node[circle] (r) at +(-1,0) {c)};
            \draw[-] (a) -- (b);
		      \draw [->] (b) -- (c);
		\end{scope}
        
		\begin{scope}[shift={(9.5,-0.8)}]
			\node[circle] (a) at (0, 0) {\(Y_0\)};
			\node[circle] (b) at +(0: 1.7) {\(Y_1\)};
			\node[circle] (c) at +(0: 3.4) {\(Y_2\)};
			\draw [->] (b) -- (a);
			\draw [-] (c) -- (b);
		\end{scope}
	\end{tikzpicture}
	\caption{The graphs illustrate sufficient SI conditions which imply \((X_0,X_1,X_2)\leq_{sm} (Y_0,Y_1,Y_2)\) whenever \((X_i,X_j)\leq_{sm} (Y_i,Y_j)\) for \((i,j)\in E,\) where 
    an arrow \(U \to V\) indicates \(V\uparrow_{st} U.\) The graph in a) corresponds to the setting in Proposition \ref{propMKprocess}, while the graph in b) corresponds to the setting in Lemma \ref{lemfacmod}. The graph in c) generalises the two previous cases and corresponds to the setting in Theorem \ref{thm: supermodular order of tree specifications generalized} noting that there is no positive dependence condition on \((X_0,X_1)\) and \((Y_1,Y_2).\) Similar results are obtained when replacing \(\leq_{sm}\) with \(\geq_{sm}.\)
 }
	\label{fig:counterexample markov2}

	\begin{tikzpicture}[scale=.8, transform shape]
		\tikzstyle{every circle node} = []

		\begin{scope}[shift={(-1.5,0)}]
			\node[circle] (a) at (0, 0) {\(X_0\)};
			\node[circle] (b) at +(0: 1.7) {\(X_1\)};
			\node[circle] (c) at +(0: 3.4) {\(X_2\)};
			\node[circle] (r) at +(-1,0) {a)};
            \draw[->] (a) -- (b);
		      \draw [<->] (b) -- (c);
		\end{scope}
        
		\begin{scope}[shift={(-1.5,-0.8)}]
			\node[circle] (a) at (0, 0) {\(Y_0\)};
			\node[circle] (b) at +(0: 1.7) {\(Y_1\)};
			\node[circle] (c) at +(0: 3.4) {\(Y_2\)};
			\draw [->] (a) -- (b);
			\draw [<->] (b) -- (c);
		\end{scope}
  
		\begin{scope}[shift={(-1.5,-2)}]
			\node[circle] (a) at (0, 0) {\(X_1\)};
			\node[circle] (b) at +(0: 1.7) {\(X_0\)};
			\node[circle] (c) at +(0: 3.4) {\(X_2\)};
			\node[circle] (r) at +(-1,0) {b)};
            \draw[<->] (a) -- (b);
		      \draw [<->] (b) -- (c);
		\end{scope}
        
		\begin{scope}[shift={(-1.5,-2.8)}]
			\node[circle] (a) at (0, 0) {\(Y_1\)};
			\node[circle] (b) at +(0: 1.7) {\(Y_0\)};
			\node[circle] (c) at +(0: 3.4) {\(Y_2\)};
			\draw [->] (a) -- (b);
			\draw [<-] (b) -- (c);
		\end{scope}

		\begin{scope}[shift={(4.5,0)}]
			\node[circle] (a) at (0, 0) {\(X_0\)};
			\node[circle] (b) at +(0: 1.7) {\(X_1\)};
			\node[circle] (c) at +(0: 3.4) {\(\dots\)};
            \node[circle] (d) at +(0: 5.1) {\(X_{d-1}\)};
            \node[circle] (e) at +(0: 6.8) {\(X_d\)};
			\node[circle] (r) at +(-1,0) {c)};
            \draw[<->] (a) -- (b);
		      \draw [<->] (b) -- (c);
            \draw [<->] (c) -- (d);
            \draw[<->] (d) -- (e);
		\end{scope}
        
		\begin{scope}[shift={(4.5,-0.8)}]
			\node[circle] (a) at (0, 0) {\(Y_0\)};
			\node[circle] (b) at +(0: 1.7) {\(Y_1\)};
			\node[circle] (c) at +(0: 3.4) {\(\dots\)};
            \node[circle] (d) at +(0: 5.1) {\(Y_{d-1}\)};
            \node[circle] (e) at +(0: 6.8) {\(Y_d\)};
			\draw [->] (a) -- (b);
			\draw [<->] (b) -- (c);
            \draw [<->] (c) -- (d);
            \draw [<-] (d) -- (e);
		\end{scope}
  
		\begin{scope}[shift={(4.5,-2)}]
			\node[circle] (a) at (0, 0) {\(X_0\)};
			\node[circle] (b) at +(0: 1.7) {\(X_1\)};
			\node[circle] (c) at +(0: 3.4) {\(X_2\)};
			\node[circle] (d) at +(0: 5.1) {\(X_3\)};
			\node[circle] (r) at +(-1,0) {d)};
			\draw [-] (a) -- (b);
			\draw [-] (b) -- (c);
			\draw [-] (c) -- (d);
		\end{scope}

		\begin{scope}[shift={(4.5,-2.8)}]
			\node[circle] (a) at (0, 0) {\(Y_0\)};
			\node[circle] (b) at +(0: 1.7) {\(Y_1\)};
			\node[circle] (c) at +(0: 3.4) {\(Y_2\)};
			\node[circle] (d) at +(0: 5.1) {\(Y_3\)};
			\draw [<->] (a) -- (b) node[pos=0.5,above]{\tiny{\(\operatorname{MTP}_2\)}};
			\draw [<->] (b) -- (c) node[pos=0.5,above]{\tiny{\(\operatorname{MTP}_2\)}};
			\draw [<->] (c) -- (d) node[pos=0.5,above]{\tiny{\(\operatorname{MTP}_2\)}};
		\end{scope}

	\end{tikzpicture}
	\caption{The graphs illustrate the non-sufficient SI conditions in a) Example \ref{exp:a}, b) Example \ref{exp:b}, c) Example \ref{exp:c} and d) Example \ref{exp:d},  where 
    an arrow \(U \to V\) indicates \(V\uparrow_{st} U\) and where \(U \xleftrightarrow{\MTP} V\) indicates that \((U,V)\) is \(\MTP\) which implies, in particular, that \(U\uparrow_{st} V\) and \(V\uparrow_{st} U.\)
    As shown in the respective examples, none of the SI conditions on \(X=(X_n)_{n\in N}\) and \(Y=(Y_n)_{n\in N}\) are sufficient for \(X\leq_{sm} Y\) whenever $(X_i,X_j)\leq_{sm}(Y_i,Y_j)$ for all $(i,j)\in E.$ 
}
	\label{fig:counterexample markov}
\end{figure}

The following example shows that there is no version of Lemma \ref{lemfacmod} when assumptions \eqref{lemfacmod1} and \eqref{lemfacmod2} are replaced by  \((X_0,X_i)\) is \CI and \(Y_0\uparrow_{st} Y_i\) for \(i\in \{1,\ldots,d\}\,\).
These SI conditions correspond to the graph in Figure \ref{fig:counterexample markov}\,b).

\begin{example}[Conditions \eqref{lemfacmod1}-\eqref{lemfacmod2} in Lemma \ref{lemfacmod} cannot be replaced by \((X_0,X_i)\) \CI and \(Y_0\uparrow_{st} Y_i\)]\label{exp:b}  
Let \(T=(N,E),\) \(N=\{0,1,2\},\) \(E=\{(0,1),(0,2)\},\) be a tree on three nodes.
    Consider the doubly stochastic matrices $a^{01},a^{12},b^{01},b^{12}\in\mathbb R^{4\times 4}$ defined by
    \begin{align}\label{eq: matrices from example b}
    \begin{split}
       a^{01} &= \frac{1}{40}\begin{pmatrix}
3 & 3 & 3 & 1 \\
3 & 3 & 3 & 1 \\
3 & 3 & 3 & 1 \\
1 & 1 & 1 & 7 \\
\end{pmatrix},\quad
b^{01} = \frac{1}{40}\begin{pmatrix}
4 & 3 & 2 & 1 \\
5 & 4 & 1 & 0 \\
1 & 2 & 5 & 2 \\
0 & 1 & 2 & 7 \\
\end{pmatrix},\\
a^{02} &= \frac{1}{40}\begin{pmatrix}
4 & 3 & 2 & 1 \\
3 & 3 & 2 & 2 \\
2 & 3 & 3 & 2 \\
1 & 1 & 3 & 5 \\
\end{pmatrix},\quad
b^{02} = \frac{1}{40}\begin{pmatrix}
6 & 2 & 2 & 0 \\
3 & 3 & 1 & 3 \\
1 & 3 & 4 & 2 \\
0 & 2 & 3 & 5 \\
\end{pmatrix}.
    \end{split}
    \end{align}
    Let $X=(X_0,X_1,X_2)$ and $Y=(Y_0,Y_1,Y_2)$ be random vectors which follow a Markov tree distribution with respect to the tree \(T\)
    specified by the bivariate distributions given by
    \begin{align*}
        P[X_0=k, X_{i}=l]=a^{0i}_{kl}\text{ and }P[Y_{0}=k, Y_{i}=l]=b^{0i}_{kl} 
    \end{align*}
    for \(k,l\in \{0,1,2,3\}\) and \(i\in \{0,1\}.\)
    All marginal distributions of the vectors $X$ and $Y$ are uniform on \(\{0,1,2,3\},\) i.e., \(X_i\eqd Y_i\sim U(\{0,1,2,3\})\) for \(i\in \{0,1,2\}.\)
    Further, for \(i\in \{1,2\},\) it holds that \((X_0,X_{i})\leq_{lo} (Y_0,Y_{i})\)  and thus, due to identical marginals, \((X_0,X_{i})\leq_{sm} (Y_0,Y_{i})\,;\) see \eqref{eqbivord}.
Moreover,
    \begin{align*}
        X_{0}&\uparrow_{st} X_1, \phantom{Y_1Y_0} X_{0}\uparrow_{st} X_2, \phantom{Y_1Y_0} X_{1}\uparrow_{st}X_0, \phantom{Y_1Y_0}  X_{2}\uparrow_{st} X_0,\\
        Y_{0}&\uparrow_{st} Y_1, \phantom{X_1X_0} Y_{0}\uparrow_{st} Y_2,
    \end{align*}
    i.e., for \(i\in \{1,2\},\) the common factor variable \(Y_0\) is stochastically increasing \(Y_i\) and $(X_0,X_1)$ as well as $(X_0,X_2)$ are CI.
    For the lower orthant set $\Theta\coloneqq (-\infty,2]^3$, it holds that
    \begin{align}\label{eq: example b calculation}
       P^X(\Theta) = 4\!\!\sum_{i,j,k=1}^3  a^{01}_{ji}a^{12}_{jk} = \frac{2.25}{4}> \frac{2.24}{4} = 4\!\!\sum_{i,j,k=1}^3 b^{01}_{ji}b^{12}_{jk}= P^Y(\Theta).
    \end{align}
    This shows that $X\nleq_{lo}Y$ and thus, due to \eqref{smpdo}, $X\nleq_{sm}Y$. Hence, condition \eqref{lemfacmod2} of Lemma \ref{lemfacmod} cannot be replaced by \(Y_0\uparrow_{st} Y_i.\) Similarly, this example highlights the importance of the assumption $Y_j\uparrow_{st}Y_i$ for all $(i,j)\in E$ with $j\notin P$ in Theorem \ref{thm: supermodular order of tree specifications generalized}.
\end{example}

The SI conditions in the following example correspond to the graph in Figure \ref{fig:counterexample markov}\,c). This example is a comonotonic extension of Example \ref{exp:b} and is also used in the proof of Proposition \ref{propSIass} on the generality of the SI assumptions in Theorem \ref{thm: supermodular order of tree specifications generalized}.  


\begin{example}[A comonotonic extension of Example \ref{exp:b}]\label{exp:c} ~\\ Let $(X_0,X_1,X_2)$ and $(Y_0,Y_1,Y_2)$ be the random vectors in Example \ref{exp:b}. Define the (\(d+1\))-dimensional random vectors $X'$ and $Y'$ by
\begin{align*}
    X'&=(X'_0,\dots,X'_d)\coloneqq(X_1,X_0,\dots,X_0,X_2) \quad \text{ and }\\
    Y'&=(Y'_0,\dots,Y'_d)\coloneqq(Y_1,Y_0,\dots,Y_0,Y_2).
\end{align*}
Then $X'$ and $Y'$ follow a Markov tree distribution with respect to the chain $T=(N,E)$, where $N=\{0,\dots,d\}$ and $E=\{(0,1),\dots,(d-1,d)\}$.
We have that 
    \begin{align*}
        X'_{i+1}&\uparrow_{st} X'_i, &X'_{i}&\uparrow_{st} X'_{i+1}, &&\text{for all } i\in \{0,\dots,d-1\},\\
        Y'_{i+1}&\uparrow_{st} Y'_i, &Y'_{j}&\uparrow_{st} Y'_{j+1}, &&\text{for all } i\in\{0,\dots,d-2\}\text{ and }j\in \{1,\dots,d-1\}.
    \end{align*}
Moreover, it holds that $(X'_i,X'_{i+1})\leq_{lo}(Y'_i,Y'_{i+1})$ for all $i=0,\dots,d-1$. 
Using \eqref{eq: example b calculation} we obtain for the lower orthant set $\Theta\coloneqq (-\infty,2]^{d+1}$ that $P^{X'}(\Theta)>P^{Y'}(\Theta).$ This shows that $X'\nleq_{lo}Y'$ and thus $X'\nleq_{sm}Y'$. This example highlights the importance of the assumption $Y_j\uparrow_{st}Y_i$ for all $(i,j)\in E$ with $j\notin P$ in Theorem \ref{thm: supermodular order of tree specifications generalized} by extending the usage of Example \ref{exp:b} to more complex cases; see also Proposition \ref{prop: justification of the SI assumptions}. 
\end{example}

The following example shows that Theorem \ref{thm: maintwo} fails to hold if the assumption on equal ranges of the marginal distribution functions is not imposed. Additionally, the example highlights that the dependence structure of a random vector determined by a Markov realization of a bivariate tree specification is determined not only by the set of bivariate copulas and the conditional independence assumptions but also by the choice of the marginal distributions.

\begin{example}[General marginals are not sufficient for Theorem \ref{thm: maintwo}]\label{exp: marginal distribution determines copula}~\\
    Let $T$ be a tree with nodes $N=\{0,1,2\}$ and edges $E=\{(0,1),(0,2)\}$. 
    Denote by \(F_{U(0,1)}\) the distribution function of the uniform distribution on \((0,1).\) Consider the marginal specifications \(F_i\) and \(G_i,\) \(i\in \{0,1,2\},\) given by
    \begin{align*}
        F_0 = F_1 = F_2 = G_1 = G_2 = F_{U(0,1)}\quad \text{and} \quad G_0=\mathds{1}_{[0,\infty)}.
    \end{align*}
    Then \(F_0,F_1,F_2,G_1\) and \(G_2\) are continuous distribution functions and \(G_0\) is the distribution function of the Dirac distribution in \(0,\) which is not continuous. 
    Denote by \(M\) the bivariate comonotonicity copula, i.e., \(M(u,v)=\min\{u,v\}\) for \((u,v)\in [0,1],\) and consider the bivariate specifications \(B_e := C_e := M\) for \(e\in E.\)
    Let \(U\) and \(V\) be independent and uniformly on \((0,1)\) distributed random variables.
    Then, for \(F=(F_0,F_1,F_2), G=(G_0,G_1,G_2),\) \(B=(B_e)_{e\in E},\) and \(C=(C_e)_{e\in E},\) the random vectors \(X=(X_0,X_1,X_2):=(U,U,U)\) and \(Y=(Y_0,Y_1,Y_2):=(0,U,V)\) are Markov realizations of \((F,T,B)\) and \((G,T,C),\) respectively, i.e.,
    \begin{align*}
        X\sim \cM(F,T,B) \quad \text{and} \quad Y \sim \cM(G,T,B).
    \end{align*}
    It holds that $F_i\geq_{st} G_i$ for all $i\in \{0,1,2\}$ and \(B_e = C_e\) (which trivially implies \(B_e\leq_{lo} C_e\)) for \(e\in E.\) Further, the bivariate copulas \(B_e\) and \(C_e,\) \(e\in E,\) are CI. Hence, all assumptions of Theorem \ref{thm: maintwo}\eqref{thm: maintwo2} except of continuity of the marginal distributions are satisfied.
    However, it holds for all \((u,v)\in (0,1)\) that 
    \[F_{X_1,X_2}(u,v) = \min\{u,v\} > uv = F_{Y_1,Y_2}(u,v),\] which implies \((X_1,X_2)\nleq_{dsm} (Y_1,Y_2).\) 
    Since the decreasing supermodular order is closed with respect to marginalization, 
    it follows that $X\nleq_{dsm} Y$.\\
    Replacing in the above setting \(G_0\) by \(G_0=\1_{[1,\infty)}\) and considering \(Y=(1,U,V),\) it follows similarly that \((X_1,X_2) \nleq_{ism} (Y_1,Y_2)\) and thus \(X\nleq_{ism} Y.\) Hence, we conclude that the continuity assumption for the marginal specifications in Theorem \ref{thm: maintwo}\eqref{thm: maintwo1} cannot be omitted.

    Note that in this example, \(X_1\) and \(X_2\) are comonotonic while \(Y_1\) and \(Y_2\) are independent. Since the bivariate tree specifications \((F,T,B)\) and \((G,T,C)\) only differ in the first component of the marginal specifications (i.e., \(F_0\ne G_0,\) \(F_1=G_1,\) \(F_2=G_2,\) and \(B=C\)), we conclude that the marginal distributions can also influence the dependence structure of Markov tree distributions. In other words, conditional independence is not solely a copula-based property.  
\end{example}

The following example justifies the \(\MTP\) assumption in Theorem \ref{themaindcx} since, surprisingly, bivariate \CI specifications generally do not result in a Markov realization that is CI.

\begin{example}[A non-\CI Markov tree distribution with bivariate \CI marginals]\label{exp:counter for ci}~\\
    For \(i\in \{1,2\},\) consider the doubly stochastic matrices $a^i = (a^i_{k\ell})_{1\leq k,\ell\leq 3}\in\mathbb R^{3\times 3}$ given by 
    \begin{align*}
    a^1 =
    \frac{1}{30}
    \begin{pmatrix}
    8 & 2 & 0 \\
    2 & 5 & 3 \\
    0 & 3 & 7 \\
    \end{pmatrix},\quad
    a^2 = \frac{1}{30}
    \begin{pmatrix}
    4 & 4 & 2 \\
    4 & 3 & 3 \\
    2 & 3 & 5 \\
    \end{pmatrix}.
    \end{align*}
    Let $X=(X_0,X_1,X_2)$ be a random vector that follows a Markov tree distribution with respect to the tree $T=(N,E)$, $N=\{0,1,2\}$, $E=\{(0,1),(1,2)\}$, specified by the univariate and bivariate distributions through 
    \begin{align}
       \label{exnonCIunivmarg} P[X_i=k] &= 1/3 \quad \text{for } i,k\in \{0,1,2\},\\
       \label{exnonCIbivmarg} P[X_i=k, X_{i+1}=\ell] &=a^{i+1}_{k\ell} \quad\text{for } i\in \{0,1\} \text{ and } k,\ell\in \{0,1,2\}.
    \end{align}
    Note that the univariate distributions in \eqref{exnonCIunivmarg} are the marginals of the bivariate distributions in \eqref{exnonCIbivmarg}.
    From the definition of the matrices \(a^1\) and \(a^2\) it can be seen that the subvectors $(X_0,X_1)$ and $(X_1,X_2)$ are \CI. It even follows that $(X_0,X_2)$ is \CI because, as a consequence of Lemma \ref{lem:CIStrees},  \(X_0\uparrow_{st}(X_1,X_2)\) and \(X_2\uparrow_{st} (X_1,X_0).\) However, the vector $(X_0,X_1,X_2)$ is not $\CI$ because
    \begin{align*}
        P[X_1\geq 1\vert X_0=0,X_2=k]=\begin{cases}
            1-16/20, & \text{ if } k=0,\\
             1-16/19, & \text{ if } k=1,\\
             1-16/22, & \text{ if } k=2,
        \end{cases}
    \end{align*}
    is not increasing in \(k.\) Hence, if the bivariate specifications are \textnormal{CI}, the implied Markov tree distribution is generally not CI.
\end{example}

The following example shows that Proposition \ref{propSch} based on the Schur order for conditional distributions cannot be extended to Markov processes and thus neither to Markov tree distributions.

\begin{example}[Schur order is not applicable to Markov processes]\label{exp:d}~\\
    Define the matrices $a,b\in\mathbb R^{3\times 3}$ by
\begin{align}
a \coloneqq  \frac{1}{30}\begin{pmatrix}
    5 & 2 & 3 \\
    3 & 7 & 0 \\
    2 & 1 & 7
\end{pmatrix}
\quad\text{and}\quad
b\coloneqq  \frac{1}{30}\begin{pmatrix}
    6 & 4 & 0 \\
    3 & 4 & 3 \\
    1 & 2 & 7
\end{pmatrix}.
\end{align}
    Let $X=(X_0,\dots,X_3)$ and $Y=(Y_0,\dots,Y_3)$ be random vectors which follow a Markov tree distribution with respect to the tree $T=(N,E)$ for $N=\{0,1,2,3\}$ and $E=\{(0,1),(1,2),(2,3)\}$ with bivariate distributions specified by 
        \begin{align*}
        P[X_i \in [k,k+s), X_{i+1}\in [\ell, \ell+t)] &=s\,t\,a_{k\ell},\\
        \quad P[Y_i\in [k,k+s), Y_{i+1}\in [\ell,\ell+t)] &=s\,t\,b_{k\ell},
    \end{align*}
    for all \(k,\ell\in \{0,1,2\},\) \(i\in\{0,1,2,3\}\) and $s,t\in[0,1].$
     All marginal distributions of the vectors $X$ and $Y$ are uniform on \([0,3),\) i.e., \(X_i\eqd Y_i\sim U([0,3))\) for all \(i\in \{0,1,2,3\}.\) It holds that $(Y_i,Y_{i+1})$ is $\MTP$ (thus, in particular, \(Y_i\uparrow_{st} Y_{i+1}\) and \(Y_{i+1}\uparrow_{st} Y_{i}\); see \eqref{implposdepcon}) and $(X_i,X_{i+1})\leq_{lo}(Y_i,Y_{i+1})$ for all $i\in \{0,1,2\}$. 
     Further, the bivariate distributions satisfy the Schur order
    \begin{align}\label{eqexschurord}
        (X_i|X_j)\leq_{S}(Y_i|Y_j)\quad \text{and}\quad (X_j|X_i)\leq_{S}(Y_j|Y_i)\quad\text{for all }(i,j)\in E,
    \end{align}
    i.e., for all edges \((i,j)\in E\) and for all \(z\in \R,\) both \(w\mapsto F_{X_i|X_j=q_{X_j}(w)}(z)\) has less variability than \(w \mapsto F_{Y_i|Y_j=q_{Y_j}(w)}(z)\) and \(w \mapsto F_{X_i|X_j=q_{X_j}(w)}(z)\) has less variability than \(w \mapsto F_{Y_i|Y_j=q_{Y_j}(w)}(z).\)
    Since \((Y_i,Y_{i+1})\) exhibits positive dependence (with respect to the strong notion \(\MTP\)) one might expect that \(Y\) exhibits more positive dependence than \(X\) and thus \(X\leq_{lo} Y.\) However, considering the lower orthant set $\Theta\coloneqq (-\infty,2)^4$, we obtain 
    \begin{align*}
       P^X(\Theta) =9\!\!\!\sum_{i_1,\dots,i_4=1}^2  a_{i_1i_2} a_{i_2i_3} a_{i_3i_4}&=\dfrac{1.259}{3} \\
      & >\dfrac{1.256}{3}=9\!\!\!\sum_{i_1,\dots,i_4=1}^2  b_{i_1i_2} b_{i_2i_3} b_{i_3i_4}=P^Y(\Theta),
    \end{align*}
    which shows that $X\nleq_{lo}Y$ and, consequently, $X\nleq_{sm}Y$. 
    Note that Theorem \ref{thm: supermodular order of tree specifications generalized} cannot be applied because \(X_i\not\uparrow_{st} X_j\) so that the vector \(X\) violates assumption \eqref{thm11}. 
 
   Hence, the random vector \(X=(X_0,\ldots,X_3)\) is not smaller in supermodular order than \(Y=(Y_0,\ldots,Y_3)\) even though, for all edges \((i,j)\in E,\) \((X_i,X_j)\) has less variability and is closer to independence than \((Y_i,Y_j)\) with respect to the Schur order for conditional distributions; see \eqref{eqexschurord}. Remarkably, \((X_i,X_j)\) is PSMD and \((Y_i,Y_j)\) is \(\MTP,\) so that all bivariate specifications exhibit positive dependencies. 
   To summarize, this example highlights the critical role of the SI assumptions in Theorem \ref{thm: supermodular order of tree specifications generalized} for both \((X_n)_{n\in N}\) and \((Y_n)_{n\in N}.\) In particular, it shows that the SI assumptions on \((X_n)_{n\in N}\) cannot be omitted or weakened to PSMD, even under the stronger comparison of the bivariate dependence specifications with respect to the Schur order for conditional distributions. Additionally, this example emphases the special properties of star structures, which enable a more general supermodular comparison result based on the Schur order for conditional distributions, as stated in Proposition \ref{propSch}.
\end{example}

\begin{appendices}



\section
{Discussion of the SI assumptions in Theorem \ref{thm: supermodular order of tree specifications generalized}}\label{secSI}
Our main result, Theorem \ref{thm: supermodular order of tree specifications generalized}, provides simple conditions for a supermodular comparison of Markov tree distributions. 
At first glance, the SI conditions on \(X\) and \(Y\) may appear rather unintuitive. As we demonstrate in the sequel, under the necessary ordering assumption \eqref{thm13} of Theorem \ref{thm: supermodular order of tree specifications generalized}, these conditions cannot be omitted or weakened to PSMD.

Figure \ref{fig:counterexample markov2} illustrates sufficient SI conditions in the three-dimensional case which, together with ordering assumption \eqref{thm13}, lead to the supermodular comparison of Markov tree distributions.
Whenever an SI condition in Theorem \ref{thm: supermodular order of tree specifications generalized} is skipped, 
it results in the existence of subvectors with a dependence structure that aligns with (or is even weaker than) the SI conditions depicted in the graphs in Figure \ref{fig:counterexample markov}\,a),b), or c). For these settings, we have provided examples that do not yield a lower orthant comparison result and thus do not imply a supermodular comparison result either. 
Proposition \ref{propSIass} summarizes the findings of this section regarding the necessity of the SI conditions.
In conclusion, for a Markov tree distribution on \(3\) nodes, it is precisely these two SI conditions in Figure \ref{fig:counterexample markov2}\,c) that lead to a supermodular comparison result.

The following remark summarizes Examples \ref{exp:a} -- \ref{exp:c} demonstrating that the SI conditions in Figure \ref{fig:counterexample markov} do not yield comparison results as stated in Proposition \ref{propMKprocess} and Lemma \ref{lemfacmod}.

\begin{remark}
    \begin{enumerate}[(a)]
        \item Condition \eqref{propMKprocess2} in Proposition \ref{propMKprocess} cannot be replaced by \(Y_{i+1}\uparrow_{st} Y_i:\) As shown in Example \ref{exp:a}, there exist Markov tree distributed random vectors \(X=(X_0,X_1,X_2)\) and \(Y=(Y_0,Y_1,Y_2)\) for \(T\) satisfying the SI conditions in Figure \ref{fig:counterexample markov}\,a) such that \((X_i,X_{i+1})\leq_{sm} (Y_i,Y_{i+1})\) but \(X\not\leq_{lo} Y.\)
        Hence, the SI conditions in Figure \ref{fig:counterexample markov}\,a) do not imply a supermodular comparison result.
        \item Conditions \eqref{lemfacmod1}-\eqref{lemfacmod2} in Lemma \ref{lemfacmod} cannot be replaced by \((X_0,X_i)\) being \CI and \(Y_0\uparrow_{st} Y_i:\) As shown in Example \ref{exp:b}, there exist Markov tree distributed random vectors \(X=(X_0,X_1,X_2)\) and \(Y=(Y_0,Y_1,Y_2)\) for \(T\) satisfying the SI conditions in Figure \ref{fig:counterexample markov}\,b) such that \((X_0,X_{i})\leq_{sm} (Y_0,Y_{i})\) for \(i\in \{1,2\}\) but \(X\not\leq_{lo} Y.\)
        \item As an extension of the setting in Figure \ref{fig:counterexample markov}\,b), the SI conditions in Figure \ref{fig:counterexample markov}\,c) are not sufficient to prove a supermodular ordering result as  in Proposition \ref{propMKprocess}; see Example \ref{exp:c}.
    \end{enumerate}
\end{remark}

Using Examples \ref{exp:a} -- \ref{exp:c} summarized in the preceding remark, we show in the following proposition that none of the SI conditions in assumptions \eqref{thm11} and \eqref{thm12} of Theorem \ref{thm: supermodular order of tree specifications generalized} can be dropped or weakened to PSMD. 
Since the supermodular order 
is closed under marginalization (see \cite[Theorem 9.A.9(c)]{Shaked-Shanthikumar-2007}), the proof reduces to the cases considered above.

To this end, let \(T=(N,E)\) be the tree, \(P\) the path, and $k^*$ the specific child of the root as considered in Theorem \ref{thm: supermodular order of tree specifications generalized}. 
For a node \(j^* \in N\backslash \{0\}\), which will be specified in Proposition \ref{prop: justification of the SI assumptions}, we relax the SI conditions in \eqref{thm11} and \eqref{thm12} of Theorem \ref{thm: supermodular order of tree specifications generalized} concerning only one edge: For $(i,j)\in E$ let
\begin{align}
    \tag{$i^*$} \label{thm11 mod}  \phantom{^*}\quad &  X_j\uparrow_{st}X_i\text{ if }j\notin\{k^*,j^*\}\text{, and }(X_{p_{j^*}}^\perp,X_{j^*}^\perp)\leq_{sm}(X_{p_{j^*}},X_{j^*}), \\
    \tag{$ii^*$} \label{thm12 1mod} \phantom{^*}\quad &  Y_i\uparrow_{st}Y_j\text{ if }j\notin L\cup\{j^*\}\text{, }Y_j\uparrow_{st}Y_i\text{ if }j\notin P\text{, and }(X_{p_{j^*}}^\perp,X_{j^*}^\perp)\leq_{sm}(X_{p_{j^*}},X_{j^*}),\\
    \tag{$ii^{**}$} \label{thm12 2mod} \quad &  Y_i\uparrow_{st}Y_j\text{ if }j\notin L\text{, }Y_j\uparrow_{st}Y_i\text{ if }j\notin P\cup\{j^*\}\text{, and }(X_{p_{j^*}}^\perp,X_{j^*}^\perp)\leq_{sm}(X_{p_{j^*}},X_{j^*}). 
\end{align}

    Note that $p_{j^*}\in N$ denotes the parent node of $j^*\in N$; see Definition \ref{def:childs_descendants_ancestors}. We distinguish between three cases. The first one concerns the SI conditions on \(X\) in \eqref{thm11}.
    The second case relates to the SI conditions \eqref{thm12} on $Y$, which states that  $Y_i\uparrow_{st}Y_j$ for $(i,j)\in E$ whenever $j\notin L\cup \{0\}$, where $N\backslash (L\cup \{0\})$ is a nonempty set if and only if $T$ is not a star. In the third case we consider the SI conditions of the second part of \eqref{thm12}, which asserts that $Y_j\uparrow_{st}Y_i$ for $(i,j)\in E$, whenever $j\notin P\cup \{0\}$, where $N\backslash (P\cup \{0\})$ is a nonempty set if and only if $T$ is not a chain. The following result establishes that none of the SI conditions in Theorem \ref{thm: supermodular order of tree specifications generalized} can be omitted.


\begin{proposition}[Generality of the SI assumptions of Theorem \ref{thm: supermodular order of tree specifications generalized}]\label{prop: justification of the SI assumptions}\label{propSIass}~\\
Assume that \(|N|\geq 3.\) None of the SI assumptions in Theorem \ref{thm: supermodular order of tree specifications generalized} can be replaced by the weaker concept of PSMD, 
    i.e., 
    for the tree $T=(N,E)$, the following statements hold true:
    \begin{enumerate}[(a)]
        \item\label{prop41}  For any \(j^*\in N\setminus\{0,k^*\}\), there are Markov tree distributed sequences \(X=(X_n)_{n\in N}\) and \(Y=(Y_n)_{n\in N}\) with $X\nleq_{lo}Y$ such that  $X$ and $Y$ fulfill for all $e=(i,j)\in E$ condition  \eqref{thm11 mod} and conditions \eqref{thm12} and \eqref{thm13} of Theorem \nolinebreak \ref{thm: supermodular order of tree specifications generalized}.
        \item\label{prop42} If \(T\) is not a star, then for any \(j^*\in N\setminus (L\cup\{0\})\), there are Markov tree distributed sequences $X=(X_n)_{n\in N}$ and $Y=(Y_n)_{n\in N}$ with $X\nleq_{lo}Y$, such that  $X$ and $Y$ fulfill for all $e=(i,j)\in E$ condition \eqref{thm12 1mod} and conditions \eqref{thm11} and \eqref{thm13} of Theorem \ref{thm: supermodular order of tree specifications generalized}.
        \item\label{prop43} If \(T\) is not a chain, then for any \(j^*\in N\setminus (P\cup\{0\})\), there are Markov tree distributed sequences $X=(X_n)_{n\in N}$ and $Y=(Y_n)_{n\in N}$ with $X\nleq_{lo}Y$, such that $X$ and $Y$ fulfill for all $e=(i,j)\in E$ condition \eqref{thm12 2mod} and conditions \eqref{thm11} and \eqref{thm13} of Theorem \ref{thm: supermodular order of tree specifications generalized}.
    \end{enumerate}
    In particular, in each of the statements \eqref{prop41}-\eqref{prop43}, it holds that $X\not\leq_{sm}Y$.
\end{proposition}

\section{Proofs}\label{secproofs}

\subsection{Proofs of Section \ref{secintro}}\label{appendix: maintheorem}

In this section, we provide the proof of Theorem \ref{thm: supermodular order of tree specifications generalized} which combines Lemma \ref{lemfacmod} with an extension of the conditioning argument in \cite[Theorem 3.2]{Hu-2000} to a non-stationary setting and to Markovian tree structures. 
Working with trees requires additional tools, which we introduce in the sequel.

To navigate through a tree, we define the commonly used terms for parent, child, descendant, and ancestor relationships as follows.

\begin{definition}[Parent, child, descendant, ancestor, degree]\label{def:childs_descendants_ancestors}~\\
	Let $\bT=(\bN,\bE)$ be a tree with root \(0\in N\).
	\begin{enumerate}[(i)]
		\item  The \emph{parent} of $i\in N \backslash\{0\}$  is defined as the unique element \(p_i\in N\) such that $(p_i,i)\in E$.
		\item  The set of \emph{children} of $i\in N$ is defined by $c_i\coloneqq \{j\in \bN\vert  i=p_j  \}.$
		\item The set of \emph{descendants} of $i\in \bN$ is defined by $d_i\coloneqq \{j\in \bN\vert \text{ there exists}$ $\text{a directed path from $i$ to $j$}\}.$
		\item The set of \emph{ancestors} of $i\in N\backslash\{0\}$ is defined by $a_i\coloneqq \{j\in \bN\vert  i\in d_j \}.$ 
  	\end{enumerate}
\end{definition}
Recall that the degree of a node $i\in N\backslash \{0\}$ is given by $\operatorname{deg}(i)= |c_i|+1$ and $\operatorname{deg}(0)=|c_0|$. 
According to its definition, a node $i\in N$ is not included in the set of its children $c_i$, descendants $d_i$, or ancestors $a_i$. Note that the root \(0\in N\) is the only node without parent. 

The proof of Theorem \ref{thm: supermodular order of tree specifications generalized} is based on induction, which requires a specific enumeration of the nodes of the tree. To this end, we consider the level-order traversal, where the nodes of the tree are visited level by level, starting from the root. This is also known as a breadth-first search. 
Note that in the proofs, we only consider trees with finitely many children, as the ordering results need to be proved only for the finite-dimensional marginal distributions. 
Formally, we say that \(\pi=(\pi_n)_{n\in N}\) is an \emph{enumeration of $N$}  if it is a one-to-one map from $N$ to $N$. Further,  $\pi$ is a \emph{level-order traversal} if it is an enumeration such that for any pair of nodes $i$ and $j$ with $i < j$ (with respect to the order on \(\N\)), the node $\pi_i$ is in a lower or the same level of the tree compared to $\pi_j$, i.e., for all \(i,j,\in N,\)
\begin{align}\label{exp: enumeration of trees} 
     i<j \quad\text{implies}\quad |p[0,\pi_i]|\leq|p[0,\pi_j]|, 
\end{align}
where \(p[0,\pi_i]\) denotes an undirected path from the root to \(\pi\); see \eqref{eq: Notation path} below.
Note that this  enumeration is not uniquely determined.  We illustrate the level-order traversal in Figure \nolinebreak \ref{fig: level order traversal} using a tree with \(8\) nodes distributed across \(3\) levels.
	\begin{figure}[t]
		\centering
		\begin{tikzpicture}[level distance=12mm,
			level 1/.style={sibling distance=25mm},
			level 2/.style={sibling distance=10mm}]
			\node {1}
			child {node {6}
				child {node {5}}
				child {node {7}}
			}
			child {node {3}
				child {node {2}}
				child {node {4}} 
				child {node {8}} 
			};
		\end{tikzpicture}
		\quad
		\begin{tikzpicture}[level distance=12mm,
			level 1/.style={sibling distance=25mm},
			level 2/.style={sibling distance=10mm}]
			\node {$\pi_1$}
			child {node {$\pi_2$}
				child {node {$\pi_4$}}
				child {node {$\pi_5$}}
			}
			child {node {$\pi_3$}
				child {node {$\pi_6$}}
				child {node {$\pi_7$}}
				child {node {$\pi_8$}} 
			};
		\end{tikzpicture}
		\caption{A level-order traversal $\pi=(1,6,3,5,7,2,4,8)$ for the tree $T=(N,E)$ with nodes $N=\{1,\dots,8\}$ and edges $E=\{(1,6),(1,3),(6,5),(6,7),(3,2),(3,4),(3,8)\}$.}\label{fig: level order traversal}
	\end{figure}

To prove Theorem \ref{thm: supermodular order of tree specifications generalized}, we rely on several lemmas presented in the sequel. The following lemma for Markov tree distributions states that, if any random variable is SI in its ancestor random variables, then any random vector corresponding to descendant nodes is SI in the random variable associated with the current node. In the context of Markov processes, this means that if, at each time step, the present random variable is SI in the preceding variable, then all future variables are SI in the present variable; see \cite[Theorem 3.2]{fang1994decrease}. We define, similarly to the undirected path \(p(i,j)\) introduced in Section \ref{secdeftree}, the undirected paths 
\begin{align}
\begin{split}\label{eq: Notation path}
     p[i,j)\coloneqq\{i\}\cup p(i,j), \quad
     p(i,j]\coloneqq\{j\}\cup p(i,j), \quad
     p[i,j]\coloneqq\{i,j\}\cup p(i,j).
\end{split}
\end{align}
Moreover, for a subset of nodes $J\subseteq N $, we write \(x_J\coloneqq(x_j)_{j\in J}\) and $X_J\coloneqq(X_j)_{j\in J}$ to refer to the vector and random vector associated with these components. 
\begin{lemma}\label{lem: stochasticallyincreasing}
	Let $X$ be a sequence of random variables which follow a Markov tree distribution with respect to a tree $T=(N,E)$. Assume that $X_i\uparrow_{st} (X_j,\ j\in a_i)$ for all $i\in N\backslash\{0\}$. Then, for all $i\in N$ and for any finite subset $J\subseteq d_i$, it follows that $(X_j, j\in J)\uparrow_{st}X_i$, i.e., for any $\phi\colon \mathbb R^{|J|}\to \mathbb R$ increasing  and bounded, the mapping
    \begin{align}
        x_i\mapsto \int \phi(x_J)P^{X_J\vert X_i=x_i}(dx_J)
    \end{align}
    is increasing. 
\end{lemma}
\begin{proof}
	We prove the result by induction over $|J|$. Let $|J|=1$, $J=\{j_1\}\subseteq d_i$ for \(i\in N\). Then, by assumption, $X_{k}\uparrow_{st}X_{p_k}$ for all $k \in p(i,j_1]$; see \eqref{eq: Notation path}.
   We conclude by \cite[Theorem 3.2]{fang1994decrease} that $X_{j_1}\uparrow_{st} X_i$.
    For the induction step, assume that the statement holds for all $J\subseteq d_i$ such that
    $|J|=n-1$. Suppose that $J$ has the form $J=\{j_1,\dots,j_n\}\subseteq d_i$. The following part of the proof is divided in two cases.
    First, we assume that there are two nodes in $J$ which are not separated by $i$, i.e., 
    \begin{align*}
      J\cap  \bigcup_{j\in J} p(i,j) \neq \emptyset.    
    \end{align*}
    Then there exists a node $j_\ell\in J$  such that $J\cap p(i,j_{\ell})\neq \emptyset$. In particular, the node $j_{\ell}$ can be chosen such that $j_{\ell}\notin p(i,j)$ for all $j\in J$. Let $j_k\in J$ be the last 
    node in $J$ in the directed path from $i$ to $j_\ell$, i.e.,
    $$j_m\coloneqq\operatorname{arg\,max}\Big\{|p(i,j)|\big\vert\,  j\in J\cap p(i,j_{\ell})\Big\}.$$
    For $K\coloneqq J\backslash \{j_\ell\}$ we have for any $\phi\colon \mathbb R^n\to \mathbb R$ increasing  and bounded that 
    \begin{align}
		&\nonumber \int \phi(x_J)P^{X_J\vert X_i=x_i}(dx_J)\\
		 \nonumber&=\int\int (x_{j_\ell},x_K)P^{X_{j_\ell}\vert X_{K}=x_{K},X_i=x_i}(dx_{j_\ell})P^{X_{K}\vert X_i=x_i}(dx_{K}) \\ 
		\label{eq: stochasticincreasing} &=\int\underbrace{\int \phi(x_{j_\ell},x_K)P^{X_{j_\ell}\vert X_{j_m}=x_{j_m}}(dx_{j_\ell})}_{\coloneqq \psi(x_{j_m},  x_{K\backslash \{j_m\})}} P^{X_{K}\vert X_i=x_i}(dx_{K}),
	\end{align}
    where we use the Markov tree dependence for the second equality.
    Applying \cite[Theorem 3.2]{fang1994decrease} once more, we conclude that  \(\psi\) is an increasing and bounded function in \(x_{j_m}\) for all $x_{K\backslash\{j_m\}}\in \mathbb R^{n-2}$. Moreover, since $\phi$ is increasing, also the function $\psi$ is increasing in $x_{K\backslash\{j_m\}}$ for all $x_{j_m}\in \mathbb R$. Consequently, $\psi$ is increasing in each argument. By the induction hypothesis, we have that $X_{K}\uparrow_{st}X_i$ since $|K|=n-1$ and $K\subseteq d_i$.  Hence, by the definition of the stochastic order, the function in \eqref{eq: stochasticincreasing} is increasing in $x_i$, proving that $X_J\uparrow_{st}X_i$.

    In the second case we assume that $J\cap p(i,j)=\emptyset$ for all $j\in J=\{j_1,\dots,j_n\}$. 
    For $K\coloneqq J\backslash \{j_n\}$ we have for any increasing and bounded $\phi\colon \mathbb R^n\to \mathbb R$ that 
    \begin{align}
        P^{X_J\vert X_i=x_i}=P^{X_{K}\vert X_i=x_i}\otimes P^{X_{j_n}\vert X_i=x_i}
    \end{align}
    by the Markov property. It follows that
    \begin{align}\label{eq: stochasticincreasing2}
	   \begin{split}
		&\int \phi(x_J)P^{X_J\vert X_i=x_i}(dx_J)
		=\int\int \phi(x_J)P^{X_{K}\vert X_i=x_i}(dx_{K})P^{X_{j_n}\vert X_i=x_i}(dx_{j_n}).
	   \end{split}
	\end{align}
    Since $X_{j_n}\uparrow_{st}X_i$ by assumption and, by the induction hypothesis,  $X_{K}\uparrow_{st}X_i$, we obtain that $X_J\uparrow_{st}X_i$ and the result is proven. 
\end{proof}

The following simple lemma, which we were unable to locate in the literature, shows that if all components $X_i$, $i\in N\backslash\{0\}$ of a Markov tree distributed random vector are SI in its parents, then $X_i$ is SI in all of its ancestors.
\begin{lemma}\label{lem:CIStrees}
	Let $X = (X_n)_{n\in N}$ be a random vector that has Markov tree dependence with respect to a tree $T=(N,E)$. If $X_j\uparrow_{st} X_i$ for all edges $(i,j)\in E$, then 
        \begin{align}\label{eq: CIS-T}
        X_i\uparrow_{st} (X_j,\ j\in a_i) \quad \text{for all nodes } i\in N\backslash\{0\}.
        \end{align}
\end{lemma}

\begin{proof}
    In order to prove the statement, it is sufficient to prove that a Markov process $(X_0,\dots,X_n)$ is \CIS if $X_{i+1}\uparrow_{st} X_{i}$ for all $i=0,\dots,n-1$. For $n=2$, the statement is trivial. By induction, we have to prove that $(X_0,\dots,X_{n+1})$ is \CIS if $(X_0,\dots,X_n)$ is \CIS and $X_{n+1}\uparrow_{st} X_{n}$. This follows if $X_{n+1}\uparrow_{st} (X_{j},\ j\in \{1,\dots,n\}).$ By the Markov property it holds that
    \begin{align}
			X_{n+1}\uparrow_{st}(X_j,\ j\in\{0,\dots,n\})\Leftrightarrow X_{n+1}\uparrow_{st} X_{n},
		\end{align}
    and the result is proven. 
\end{proof}

 The following lemma extends \cite[Theorem 3.1]{Hu-2000} from Markov processes to Markov tree distributions. It provides sufficient conditions under which an integral of a supermodular function, with respect to a conditional distribution, results in a supermodular function that depends on both the conditioning variable and the variables that remain unintegrated. 
 Recall that \(p[i,j]\) denotes the undirected path from \(i\) to \(j\), including \(i\) and \(j.\)

\begin{lemma}[Supermodular conditional integration]\label{lem: supermodular}\ 
\begin{enumerate}[(i)]
    \item\label{lem: supermodular i} Let $X$ be a random vector that follows a Markov tree distribution with respect to a tree $T=(N,E)$. Let $i\in N$ and let $J\subseteq d_i$ be a finite set. Assume that $X_{j'}\uparrow_{st}X_{j}$ for all $(j,j')\in E$. Moreover let $K$ be an arbitrary finite index set and $\phi$ be a supermodular function on $\R^{1+|K|+|J|}$. 
	Then  the function $\varphi\colon \R^{|K|+1}\to\R$ given by
	\begin{align}\label{eq: phi supermodular}
		 \varphi(z_{i},z_K) := \int \phi(z_{i},z_K,z_J) P^{X_J\vert X_{i }=z_{i}}(dz_J),
	\end{align}
is supermodular.
\item\label{lem: supermodular ii}
   The assumption $X_{j'}\uparrow_{st}X_{j}$ for all $(j,j')\in E$ in \eqref{lem: supermodular i} can be weakened to $X_{j'}\uparrow_{st}X_j$ for all edges $(j,j')\in\{(j,j')\in E\mid j,j'\in  \bigcup_{u\in J}p[i,u]\}$ of the subtree that consists of all paths from node \(i\) to its descendants in \(J\subseteq d_i.\)
\end{enumerate}
		
\end{lemma}
\begin{proof}
\eqref{lem: supermodular i} 
By the supermodularity of $\phi$ we obtain for $k_1,k_2\in K$, \(k_1\ne k_2,\) that the map
	\begin{align}\label{eq: supermodularity in i1 and i2}
		(z_{k_1},z_{k_2})\mapsto \varphi(z_{i},z_{k_1},z_{k_2},z_{K\backslash\{k_1,k_2\}} )
	\end{align} is supermodular for all $(z_{i},z_{K\backslash\{k_1,k_2\}})\in \mathbb R^{1+|K|-2}$. 
 In the following we show for  $k\in K$  that the map 
	\begin{align}\label{eq: supermodularity in x and i}
		(z_{i},z_{k})\mapsto \varphi(z_{i},z_k,z_{K\backslash\{k\}})
	\end{align}
	is supermodular for all $z_{K\backslash\{k\}}\in \mathbb R^{|K|-1}$.
	We fix $z_{k}<z_{k}'$. By the supermodularity of $\phi$ it follows for all $j\in J$ 
    that the map 
	\begin{align}\label{eq: lemma4 monotonoicity zj}
        z_{j}\mapsto \phi(z_{i},z'_k,z_{K\backslash\{k\}},z_j,z_{J\backslash \{j\}})-\phi(z_{i},z_k,z_{K\backslash\{k\}},z_j,z_{J\backslash \{j\}})
	\end{align}
	is increasing  for all $z_{J\backslash\{j\}}\in \mathbb R^{|J|-1}$ and  $z_{K\backslash\{k\}}\in\mathbb R^{|K|-1}$. Similarly, the map 
    \begin{align}\label{eq: lemma4 monotonicity for i}
        z_{i}\mapsto \phi(z_{i},z'_k,z_{K\backslash\{k\}},z_J)-\phi(z_{i},z_k,z_{K\backslash\{k\}},z_J)
    \end{align}
    is increasing for all $z_{J}\in \mathbb R^{|J|}$ and  $z_{K\backslash\{k\}}\in\mathbb R^{|K|-1}$. By Lemma \ref{lem:CIStrees} the assumption 
    $X_{j'}\uparrow_{st}X_{j}$ for all $(j,j')\in E$ implies that $X_{j}\uparrow_{st}(X_{j'},j'\in a_{j})$  for all $j\in N\backslash\{0\}$. Thus, we can conclude from Lemma \ref{lem: stochasticallyincreasing} that $(X_j, j\in J)\uparrow_{st}X_i$.
    Now, for $z_{i}<z_{i}'$, we obtain that
	\begin{align}\begin{split}\label{eq: lemma4 monotonicity}
	&\varphi(z_{i},z_k',z_{K\backslash \{k\}})-	\varphi(z_{i},z_k,z_{K\backslash \{k\}})\\
		&\leq\int \phi(z_{i},z'_k,z_{K\setminus \{k\}},z_J) P^{X_J\vert X_{i }=z'_{i}}(dz_J)-\int \phi(z_{i},z_k,z_{K\backslash \{k\}},z_J) P^{X_J\vert X_{i }=z'_{i}}(dz_J)\\
		& \leq \varphi(z'_{i},z_k',z_{K\backslash \{k\}})-	\varphi(z'_{i},z_k,z_{K\backslash \{k\}}),
    \end{split}
	\end{align}
 where the first inequality follows by $(X_j, j\in J)\uparrow_{st}X_i$ and by the increasingness of \eqref{eq: lemma4 monotonoicity zj} for \(j\in J.\) The second inequality holds true because the function in \eqref{eq: lemma4 monotonicity for i} is increasing. 
Now, \eqref{eq: lemma4 monotonicity} implies
\begin{align*}
	\varphi(z'_{i},z_k',z_{K\backslash \{k\}})+	\varphi(z_{i},z_k,z_{K\backslash \{k\}})\geq 	\varphi(z'_{i},z_k,z_{K\backslash \{k\}})+	\varphi(z_{i},z'_k,z_{K\backslash \{k\}})
\end{align*}
and we obtain the supermodularity of \eqref{eq: supermodularity in x and i}. Finally, the supermodularity in the cases \eqref{eq: supermodularity in i1 and i2} and \eqref{eq: supermodularity in x and i} yields the supermodularity of \eqref{eq: phi supermodular}, which proves the statement. 

\eqref{lem: supermodular ii} The second statement follows by applying \eqref{lem: supermodular i} for the subtree $T'=(N',E')$ with root $i\in N$, nodes $N'=\bigcup_{u\in J}p[i,u] $ and edges $E'=\{(j,j')\in E \mid j,j'\in  N'\}$.
\end{proof}

The following proposition is a version of Theorem \ref{thm: supermodular order of tree specifications generalized} under additional SI assumptions. 
We briefly outline the idea behind the proof of this proposition.
We start with the distribution of $X=(X_n)_{n\in N}$ and iteratively replace the dependence specifications of \(X\) with those of \(Y\) over all star structures within the tree $T = (N, E)$. This procedure uses a conditioning argument and applies Lemma \ref{lemfacmod} on the supermodular comparison of star structures. 
More precisely, for the induction, we navigate through the tree using a level-order traversal enumeration as described in \eqref{exp: enumeration of trees}. A disintegration argument reduces the proof at each step to the setting of Lemma \ref{lemfacmod}. Finally, by leveraging the transitivity of the supermodular order, we obtain a supermodular comparison between $X$ and $Y$. 
In the proof of Theorem \ref{thm: supermodular order of tree specifications generalized}, we further show that several SI conditions along a specified path can be omitted. 
 \begin{proposition}\label{thm: supermodular order of tree specifications}
Let $X=(X_n)_{n\in N}$ and $Y=(Y_n)_{n\in N}$ be sequences of random variables each following a Markov tree distribution with respect to a tree $T=(N,E)$. Let
	\begin{enumerate}[(i)]
		\item \label{propsmspec1} $X_j\uparrow_{st}X_i$ for all $(i,j)\in E$,
		\item \label{propsmspec2} $Y_i\uparrow_{st}Y_j$ for all $(i,j)\in  E, j\notin L$, and $Y_j\uparrow_{st}Y_i$ for all $(i,j)\in E$, 
		\item \label{propsmspec3} $(X_i,X_j) \leq_{sm}	(Y_i,Y_j)$ (resp.\@ $ \geq_{sm}$) for all \((i,j)\in E\).
	\end{enumerate}
	Then $X\leq_{sm}Y$ (resp.\@ $ \geq_{sm}$).
\end{proposition}

\begin{proof} 
    By definition of the stochastic orderings for stochastic processes (see Definition \ref{defdepord}), we can assume that \(d +1 := |N| < \infty\) and  \(N=\{0,\ldots,d\}\). We only show the case that \(X\leq_{sm} Y\) if $(X_i,X_j) \leq_{sm}	(Y_i,Y_j)$  for all $(i,j)\in E$. 
    In the case where \((X_i,X_j)\geq_{sm} (Y_i,Y_j)\) for all $(i,j)\in E$, it follows analogously that \(X\geq_{sm} Y\) by reversing all inequality signs. \\
    Denote by $F=(F_0,\dots,F_d)$ the vector of univariate marginal distribution functions of $X$ and $Y$. Note that \(F_{X_n} = F_{Y_n}\) for all \(n\in N\) because \((X_i,X_j)\leq_{sm} (Y_i,Y_j)\) implies \(X_i\eqd Y_i\) and \(X_j\eqd Y_j.\) Moreover, for $(i,j)\in E$, denote by $B_{(i,j)}$ and $C_{(i,j)}$ a copula for $(X_i,X_j)$ and $(Y_i,Y_j)$, respectively, and define $B=(B_e)_{e\in E}$ and $C=(C_e)_{e\in E}$. 
    It follows that $X\sim \cM(F,B, T)$ and $Y\sim \cM(F,C, T)$; see 
    Proposition \ref{prop: existence of markov tree dirstribution}. Let $(\pi_n)_{n=1,\dots,d}$ be a level-order traversal of $T$ as defined in \eqref{exp: enumeration of trees}. 
	For $i\in N$, define the set of edges connecting the node \(i\) with its children by $[i,c_i]\coloneqq \{(i,j)\vert j\in c_i \}\subseteq E$. Then \([i,c_i]\) has a star-like structure. Note that if $i\in L$, then $c_i=\emptyset$ and thus $[i,c_i]=\emptyset$. For $k=0,\dots,d-1$, let $D^k=(D^k_{e})_{e\in E}$ be the set of bivariate copulas defined by
\begin{align}\label{defbivcopseq}
	D_e^k\coloneqq 
	\begin{cases}
		C_e&\text{if } e\in [\pi_n,c_{\pi_n}]\text{ for }1\leq  n \leq k,\\
		B_e&\text{if } e\in [\pi_n,c_{\pi_n}]\text{ for } k< n\leq d.\\
	\end{cases}
\end{align}
i.e., the specifications of the edges in the first \(k\)  stars in level-order traversal are replaced by the bivariate copula specifications of \(Y\), and the remaining edges are specified by the bivariate copulas associated to \(X.\)
For \(k\in \{0,\ldots,d-1\},\) consider the random vector $Z^{(k)}\sim \cM(F,T,D^k)$. Note that $D^0=B$ and $D^{d-1}=C$, which implies that $Z^{(0)}\overset{\text{d}}{=}X$ and $Z^{(d)}\overset{\text{d}}{=}Y$. We show that 
	\begin{align}\label{eq: equation for Zi}
	Z^{(k-1)}\leq_{sm} 	Z^{(k)} \quad\text{ for all }k=1,\dots,d-1.
\end{align}
Then the statement follows from the transitivity of the supermodular order.
To prove \eqref{eq: equation for Zi}, fix $k\in\{1,\dots,d-1\}$. If $\pi_k\in L$ we have that $D^{k-1}=D^k$ and thus $Z^{(k-1)}\overset{\text{d}}{=} 	Z^{(k)}$ which trivially implies $Z^{(k-1)}\leq_{sm} Z^{(k)}$.
If $\pi_k\notin L$, define $i\coloneqq\pi_{k}\in N$ and denote by $\{j_1,j_{2},\dots,j_{m} \}\coloneqq c_i$ an enumeration of the children $c_i$ of node \(i.\)  Note  that, by the definition of the bivariate copula sequences in \eqref{defbivcopseq}, \(D^{k-1}\) and \(D^k\) can only differ in the copulas which correspond to  the edges $[i,c_i]$ between node \(i\) and its children, i.e.\@, $D_e^{k-1}=D_e^{k}$ for all $e\in E\backslash [i,c_i]$. 
For proving \eqref{eq: equation for Zi}, it is sufficient to show for any bounded supermodular function \(\phi\colon \R^{d+1}\to \R\) that 
\begin{align}\label{eq: equation for Zi1}
	E[\phi(Z^{(k-1)})]\leq 	E[\phi(Z^{(k)})]\,;
\end{align}
see \cite[Theorem 3.9.14]{Mueller-Stoyan-2002}.
Recall that \(d_i\) denotes the descendants of node \(i\,;\) see Definition \ref{def:childs_descendants_ancestors}.
Then we define the function $\varphi^{(0)}$ on $\mathbb R^{|d_i|+1}$ by 
\begin{align*}
	\varphi^{(0)}(z_i,z_{d_i})\coloneqq \int\phi(z_i,z_{d_i},z_{N\backslash (\{i\}\cup d_i)})P^{Z^{(k)}_{N\backslash (\{i\}\cup d_i)}\vert Z^{(k)}_i=z_i} (dz_{N\backslash (\{i\}\cup d_i)}),
\end{align*}
i.e., we integrate over all variables conditional on \(z_i\) except of the variables associated to node \(i\) and its descendants \(d_i.\) 
Then, we define recursively over the children of node \(i\) for $\ell=1,\dots,m$ the functions \(\varphi^{(l)}\colon \R^{1+m+\sum_{r=\ell+1}^m |d_{j_r}|} \to \R\) by
\begin{align*}
	&\varphi^{(\ell)}(z_i,z_{c_i},z_{d_{j_{\ell+1}}},\dots, z_{d_{j_m}}) 
    \coloneqq \int\varphi^{(\ell-1)}(z_i,z_{c_i},z_{d_{j_{\ell}}},\dots, z_{d_{j_m}})P^{Z_{d_{j_\ell}}^{(k)}\vert Z^{(k)}_{j_\ell}=z_{j_\ell}} (dz_{d_{j_\ell}}),
\end{align*}
i.e., the function $\varphi^{(\ell)}$ is obtained through \(\varphi^{(\ell-1)}\) by integrating over the variables associated to all descendants $d_{j_{\ell}}$ of the $\ell$'th child of $i$ conditioned on the variable associated with $j_\ell$. Finally, $\varphi^{(m)}$ is a function in $(z_i,z_{j_1},\dots,z_{j_m})=(z_i,z_{c_i})$.

To prove supermodularity of $\varphi^{(0)}$, we will apply Lemma \ref{lem: supermodular}. Recall that \(p[0,i]\) denotes the path from \(0\) up to and including node \(i\). From the assumptions \eqref{propsmspec2} and \eqref{propsmspec1} we know, in particular, that 
\begin{align}
\begin{split}\label{eq: Zk SI assumptions}
    Z^{(k)}_{j'}&\uparrow_{st}Z^{(k)}_{j''}\text{ for all }(j',j'')\in E \cap (p[0,i])^2 \quad \text{and }\\
Z^{(k)}_{j''}&\uparrow_{st}Z^{(k)}_{j'}\text{ for all }(j',j'')\in E \backslash((p[0,i])^2\cup (\{i\}\cup d_i)^2),
\end{split}
\end{align}
where, for a set \(S,\) the Cartesian product \(S\times S\) is denoted by \(S^2.\)
Consider the subtree $T_0=(N_0,E_0)$ 
with nodes $N_0 := N\backslash d_i$ and edges $E_0 := E\cap N_0^2$. The nodes $N_0$ can be relabeled such that $i$ becomes the new root in $T_0$. By \eqref{eq: Zk SI assumptions} the Markov tree distributed sequence $(Z^{(k)}_n)_{n\in N_0}$ (with respect to the tree $T_0$) satisfies the assumptions of Lemma \ref{lem: supermodular} \eqref{lem: supermodular ii}  and by setting $K=d_i$ and $J=N_0\backslash \{i\}$ (here  $d_i$ still denotes the set of descendents of $i$ with respect to $T$)  we obtain the supermodularity of $\varphi^{(0)}$.

To prove the supermodularity of $\varphi^{(\ell)}$, \(\ell\in \{1,\ldots,m\},\) we assume by induction that $\varphi^{(\ell-1)}$ is supermodular.
By definition of $Z^{(k)}$ and assumption $\eqref{propsmspec1}$  we have  that 
\begin{align}
    Z^{(k)}_{j''}\uparrow_{st}Z^{(k)}_{j'}\text{ for all }(j',j'')\in E\cap (d_{j_{\ell}})^2.
\end{align}
Hence, we obtain supermodularity of $\varphi^{(\ell)}$ by Lemma  \ref{lem: supermodular} \eqref{lem: supermodular ii} by setting $K=\{i\}\cup (c_i\backslash \{j_\ell\})\cup d_{j_{\ell+1}}\cup\dots\cup d_{j_m}$ and $J=d_{j_\ell}$.
 In particular, we obtain that  $\varphi^{(m)}$ is supermodular.

Next, we show that 
\begin{align}\label{eqpropsmspec1}
    E[\phi(Z^{(k-1)})]
    &= E[\varphi^{(m)}(Z^{(k-1)}_i,Z^{(k-1)}_{c_i})] \quad \text{and}\\
    \label{eqpropsmspec2}E[\phi(Z^{(k)})]&=E[\varphi^{(m)}(Z^{(k)}_i,Z^{(k)}_{c_i})].
\end{align}
From Markov tree dependence of $Z^{(k)}$ and $Z^{(k-1)}$ and since the sets of bivariate copulas $D^{k-1}$ and $D^{k}$ are identical for all edges $E\backslash [i,c_i]$, we have 
{\small
\begin{align}\label{eq: phi0_markov}
	\begin{split}
			\varphi^{(0)}(z_i,z_{d_i})&=\int\phi(z_i,z_{d_i},z_{N\backslash (\{i\}\cup d_i)})P^{Z^{(k)}_{N\backslash (\{i\}\cup d_i)}\vert Z^{(k)}_i=z_i,Z^{(k)}_{d_i}=z_{d_i}} (dz_{N\backslash (\{i\}\cup d_i)})\\
		&=\int\phi(z_i,z_{d_i},z_{N\backslash (\{i\}\cup d_i)})P^{Z^{(k-1)}_{N\backslash (\{i\}\cup d_i)}\vert Z^{(k-1)}_i=z_i,Z^{(k-1)}_{d_i}=z_{d_i}} (dz_{N\backslash (\{i\}\cup d_i)})
	\end{split}
\end{align}
}
and, for \(\zz = (z_i,z_{c_i},z_{d_{j_{\ell}}},\dots, z_{d_{j_m}}),\) 
\begin{align}
		& \quad \quad \varphi^{(\ell)}(z_i,z_{c_i},z_{d_{j_{l+1}}},\dots, z_{d_{j_m}})\nonumber\\ \label{eq: phi i_markov}
		&=\int\varphi^{(\ell-1)}(\zz)P^{Z_{d_{j_\ell}}^{(k)}\vert Z^{(k)}_{i}=z_i,Z^{(k)}_{c_i}=z_{c_i}, Z^{(k)}_{d_{j_{\ell+1}}}=z_{d_{j_{\ell+1}}},\dots ,Z^{(k)}_{d_{j_{m}}}=z_{d_{j_{m}}}} (dz_{d_{j_\ell}})\\ \nonumber
		&=\int\varphi^{(\ell-1)}(\zz)P^{Z_{d_{j_\ell}}^{(k-1)}\vert Z^{(k-1)}_{i}=z_i,Z^{(k-1)}_{c_i}=z_{c_i}, Z^{(k-1)}_{d_{j_{\ell+1}}}=z_{d_{y_{\ell+1}}},\dots ,Z^{(k-1)}_{d_{j_{m}}}=z_{d_{j_{m}}}} (dz_{d_{j_\ell}})
\end{align}
for all $\ell=1,\dots,m$.
From \eqref{eq: phi0_markov}, \eqref{eq: phi i_markov}, and the disintegration theorem we obtain
\begin{align}
\begin{split}\label{eq:disintegration calc}
    E[\phi(Z^{(k-1)})]&= E[\varphi^{(0)}(Z^{(k-1)}_i,Z^{(k-1)}_{d_i})] \\
    &= E[\varphi^{(1)}(Z^{(k-1)}_i,Z^{(k-1)}_{c_i},Z^{(k-1)}_{d_{j_2}},\ldots,Z^{(k-1)}_{d_{j_m}})]\\
    &= \cdots \\
    &= E[\varphi^{(m-1)}(Z^{(k-1)}_i,Z^{(k-1)}_{c_i},Z^{(k-1)}_{d_{j_m}})]
    = E[\varphi^{(m)}(Z^{(k-1)}_i,Z^{(k-1)}_{c_i})],     
\end{split}
\end{align}
which shows \eqref{eqpropsmspec1}. The proof for Equation \eqref{eqpropsmspec2} follows similarly replacing $k-1$ by $k$ in \eqref{eq:disintegration calc}.
Hence, we obtain  from \eqref{eqpropsmspec1} and \eqref{eqpropsmspec2} that
\begin{align*}
	E[\phi(Z^{(k-1)})]&=E[\varphi^{(m)}(Z^{(k-1)}_i,Z^{(k-1)}_{c_i})]
    =E[\varphi^{(m)}(X_i,X_{c_i})]\\
    &\leq E[\varphi^{(m)}(Y_i,Y_{c_i})]
    =E[\varphi^{(m)}(Z^{(k)}_i,Z^{(k)}_{c_i})]
    =E[\phi(Z^{(k)})],
\end{align*}
For the second and third equality, we use that \((Z^{(k-1)}_i,Z^{(k-1)}_{c_i})\eqd (X_i,X_{c_i})\) and \((Z^{(k)}_i,Z^{(k)}_{c_i}) \eqd (Y_i,Y_{c_i}),\) which follows from the definition of  $Z^{(k-1)}$ and $Z^{(k)}$. 
For the inequality, we apply Lemma \ref{lemfacmod} using on the one hand that, by Markov tree dependence, the random variables \(X_{c_i}=(X_{j_1}, \ldots,X_{j_m})\) are conditionally independent given \(X_{i}\), and the random variables \(Y_{c_i}= (Y_{j_1}, \ldots,Y_{j_m})\) are conditionally independent given \(Y_i\). On the other hand, \(X_{j_\ell}\uparrow_{st} X_i\) by assumption \eqref{propsmspec1}, \(Y_{j_\ell}\uparrow_{st} Y_i\) by assumption \eqref{propsmspec2}, and \((X_i,X_{j_\ell})\leq_{sm} (Y_i,Y_{j_\ell})\) for \(\ell\in \{1,\ldots,m\}.\)

This finally shows \eqref{eq: equation for Zi1} and thus \eqref{eq: equation for Zi} for fixed $k\in \{1,\dots,d-1\}$. As we can repeat this procedure for each $k\in \{1,\dots,d-1\}$, the result is proven. 
\end{proof}

The comparison result formulated in Proposition \ref{thm: supermodular order of tree specifications} extends Lemma \ref{lemfacmod} to tree structures. However, it does not generalize Proposition \ref{propMKprocess} as the SI assumptions in Proposition \ref{thm: supermodular order of tree specifications} are more stringent when applied to chains. 
In the following proof of Theorem \ref{thm: supermodular order of tree specifications generalized}, it turns out that these SI assumptions on $Y$ can be omitted in one direction along a specified path \(P\) starting at the root node and---provided the path has finite length---terminating at a leaf node. Consequently, Theorem \ref{thm: supermodular order of tree specifications generalized} also extends Proposition \ref{propMKprocess} to tree structures.
Additionally, we can omit the SI assumption on \(X\) for one of the edges connected with the root. 

\begin{proof}[Proof of Theorem \ref{thm: supermodular order of tree specifications generalized}]
Similar to the proof of Proposition \ref{thm: supermodular order of tree specifications}, we can assume that $d+1\coloneqq |N|<\infty$. 
We show that $X\leq_{sm}Y$. The proof for $X\geq_{sm}Y$ follows analogously. 
    Denote by \(F=(F_n)_{n\in N}\) for \(F_n=F_{X_n}=F_{Y_n}\) the marginal distribution functions of $X$ and $Y$. Moreover, for $e=(i,j)\in E$, denote by \(B_e\) and \(C_e\) any copula for $(X_i,X_j)$ and $(Y_i,Y_j)$, respectively, and define $B=(B_e)_{e\in E}$ and $C=(C_e)_{e\in E}$. Then $X\sim \cM(F,B, T)$ and $Y\sim \cM(F,C, T)$; see 
    Proposition \ref{prop: existence of markov tree dirstribution}.
    We enumerate the path $P=p(0,\ell]=:\{\ell_1,\dots, \ell_m\}$ such that   $(\ell_{j},\ell_{j+1})\in  E$ for all $j=0,\dots,m-1$,  where $\ell_0:= 0$ and \(\ell_m= \ell.\)

   In the following, we partition the tree $T$ along the nodes $k^*,0,\ell_1,\dots,\ell_m$ into \(m+1\) disjoint subtrees, denoted as $T_0,\dots, T_m,$ with corresponding disjoint sets of nodes $N_0, \dots, N_m\subseteq N$ and edges $E_0,\dots, E_m\subseteq E$. 
   As we will see, due to the selection of the subtrees, the subvectors 
   $(X_n)_{n\in N_k}$ and $(Y_n)_{n\in N_k}$ fulfill the assumptions of Proposition \ref{thm: supermodular order of tree specifications} for all $k=0,\dots,m$.
   We start with the definition of the set of nodes $N_0,\dots,N_m$ by
    \begin{align*}
        N_0&\coloneqq\begin{cases}
		    \{k^*\}\cup d_{k^*}&\text{if }\operatorname{deg}(0)> 1,\\
            \emptyset &\text{if }\operatorname{deg}(0)= 1,\\
		\end{cases}
  \quad\text{and}\\[4pt]
  N_k&\coloneqq N\backslash \Big(\{\ell_k\}\cup d_{\ell_k} \cup \bigcup_{j=0}^{k-1}N_j\Big)\quad\text{for }k=1,\dots,m.
	\end{align*}
 Recall that \(k^*\) is due to the assumptions of Theorem \ref{thm: supermodular order of tree specifications generalized} a child node of the root \(0\) and lies in \(P\) if and only if \(\deg(0)=1.\) Note also that, for a node \(i,\) the set \(d_i\) is defined as the set of descendants of \(i.\)
    Each set of nodes $N_k$ forms a tree $T_k = (N_k,E_k)$ along with its corresponding set of edges $E_k$  defined by 
    \begin{align*}
    E_k\coloneqq \{(i,j)\in E\mid i,j\in N_k\},\text{ for }k=0,\dots,m.
    \end{align*}
    Note that some set of nodes $N_k$ may be empty or a singleton, depending on the structure of the tree $T$. It holds that 
    \begin{align*}
        \bigcup\limits_{k=0}^m N_k=\biguplus\limits_{k=0}^m N_k=N \text{ and }\bigcup\limits_{k=0}^m E_k=\biguplus\limits_{k=0}^m E_k=E\backslash \Big(\{(0,k^*)\cup\bigcup_{j=1}^{m}\{(\ell_{j-1},\ell_{j})\} \Big),
    \end{align*}
    i.e., \(\{N_k\}_{k= 0,\ldots,m}\) is a partition of \(N\) and \(\{E_k\}_{k=0,\ldots,m}\) is a partition of the set of edges \(E\) without the edges that are contained in the specified path \(P\) and without the edge between the root node and \(k^*.\)
 In the upcoming steps, we will successively replace the bivariate copulas of $X$ with those of $Y$. 
 For $e\in E$, consider the sequences \(A^1=(A^1_e)_{e\in E}\) and $A^2=(A^2_e)_{e\in E}$ of bivariate copulas given by
 	\begin{align*}
 A_e^{1}&\coloneqq 
	\begin{cases} 
		C_e &\text{if } e=(0,k^*), \\
		B_e &\text{else, } 
	\end{cases} \quad \text{and} \quad
	A_e^{2}\coloneqq 
		\begin{cases} 
		C_e &\text{if } e\in \{(0,k^*)\}\cup E_0 \\
		B_e &\text{else. }
	\end{cases}
    \end{align*}
    Furthermore, we define
    the sequence of bivariate copulas  \(D^{0}=(D^{0}_e)_{e\in E}\) by
\begin{align*}
    D^0_e&\coloneqq 
	\begin{cases} 
		A_e^2 &\text{if }\operatorname{deg}(0)> 1,\\
		  B_e  &\text{if }\operatorname{deg}(0)= 1,\\
	\end{cases}
\end{align*}
and for $k=1,\dots,m$ the sequences \(D^{2k-1}=(D^{2k-1}_e)_{e\in E}\) and \(D^{2k}=(D^{2k}_e)_{e\in E}\) of bivariate copulas by
	\begin{align*}
  D_e^{2k-1}&\coloneqq 
	\begin{cases} 
		B_e &\text{if } e\in  \bigcup_{u=k+1}^m E_u \cup \bigcup_{u=k}^m\{(\ell_{u-1},\ell_u)\}  ,\\
		C_e &\text{else}, 
	\end{cases}\\
	D_e^{2k}&\coloneqq 
		\begin{cases} 
		B_e &\text{if } e\in \bigcup_{u=k+1}^{m} E_u \cup \bigcup_{u=k+1}^m\{(\ell_{u-1},\ell_u)\} ,\\
		C_e &\text{else}. 
	\end{cases}
    \end{align*}
    Note that, by definition, $D^{2m}=C$. Keep in mind that our aim is to replace all dependencies described by $(B_e)_{e\in E}$  by the dependencies described by $(C_e)_{e\in E}$. The difference between the copula sets $D^{2k-1}$ and $D^{2k}$ lies in the replacement of the copula $B_{(\ell_{k-1},\ell_k)}$ by $C_{(\ell_{k-1},\ell_k)}$. On the other hand, the difference between $D^{2k}$ and $D^{2k+1}$ is given by the replacement of the set of copulas $(B_e)_{e\in E_{k+1}}$ by $(C_e)_{e\in E_{k+1}}$.
    Let 
    \begin{align}
        W^{(u)}& \sim \cM(F,T,A^u) \quad \text{for } u=1,2 \quad \text{and}\\
       \label{defZvau} Z^{(v)}&\sim \cM(F,T,D^v) \quad \text{for } v=0,\dots,2m.
    \end{align}
   be random vectors on $\mathbb R^{d+1}$. The following part of the proof is divided into two steps. In the first step we show that $X\leq_{sm} Z^{(0)}$, and in the second step we show that $ Z^{(0)}\leq_{sm} Y$.\\
    \emph{Step 1:} 
    If $\deg(0)< 2\,$, then $X\eqd Z^{(0)}$ which trivially implies \(X\leq_{sm} Z^{(0)}.\)
    For the other case, assume that \(\deg(0)\geq 2.\) We prove the supermodular comparison 
 \begin{align}\label{eq: firstpartinmaintheoremproof}
 X \leq_{sm} W^{(1)} \leq_{sm} W^{(2)}\eqd Z^{(0)}.
 \end{align}
 To show the first inequality in \eqref{eq: firstpartinmaintheoremproof}, let $\phi\colon \mathbb R^{d+1}\to\mathbb R$ be a bounded supermodular function. Since, by assumption \eqref{thm11}, $X_j\uparrow_{st}X_i$  for all $(i,j)\in E_0$ and $X_j\uparrow_{st}X_i$ for all $(i,j)\in E\backslash(E_0\cup\{(0,k^*)\})$,  we can apply Lemma \ref{lem: supermodular}\nolinebreak \eqref{lem: supermodular ii} twice and find that the function $\widetilde\varphi^{(1)}\colon \R^2 \to \R$ given by 
 \begin{align}
   \widetilde\varphi^{(1)}(x_0,x_{k^*})\coloneqq \int \int \phi(x)P^{X_{d_{k^*}}|X_{k^*}=x_{k^*}}(dx_{d_{k^*}})P^{X_{N\backslash (N_0\cup\{0\})}\vert X_0=x_0}(dx_{N\backslash (N_0\cup\{0\})})
 \end{align}
 is supermodular. Note that we do not need any SI assumption on $(X_0,X_{k^*})$. By the  disintegration theorem and the Markov property,
 we have that
    \begin{align} \begin{split}\label{eq:tilde Z1}
		E[\phi(X)]&=E[\widetilde\varphi^{(1)}(X_0,X_{k^*})],\quad \text{and }\\
    E[\phi(W^{(1)})]&=E[\widetilde\varphi^{(1)}(W^{(1)}_0, W^{(1)}_{k^*})]=E[\widetilde\varphi^{(1)}(Y_0, Y_{k^*})],
    \end{split}
	\end{align}
where the last equality follows from \((W^{(1)}_0,W^{(1)}_{k^*})\eqd (Y_0, Y_{k^*}).\)
Then the assumption $(X_0,X_{k^*})\leq_{sm} (Y_0,Y_{k^*})$ together with \eqref{eq:tilde Z1}  implies $X \leq_{sm} W^{(1)}$. \\
To show the second inequality in \eqref{eq: firstpartinmaintheoremproof}, again, let $\phi\colon \mathbb R^{d+1}\to\mathbb R$ be a bounded supermodular function. Since, by assumption \eqref{thm11}, $X_j\uparrow_{st}X_i$ for all $(i,j)\in E\backslash(E_0\cup\{(0,k^*)\})$ and by assumption \eqref{thm12}, $Y_0\uparrow_{st} Y_{k^*}$, we have that $W^{(1)}_j\uparrow_{st}W^{(1)}_i$ for all $(i,j)\in E\backslash(E_0\cup\{(0,k^*)\})$ and  $W^{(1)}_0\uparrow_{st} W^{(1)}_{k^*}$. Due to Lemma \ref{lem: stochasticallyincreasing}, we conclude that the function $\widetilde\varphi^{(2)}$ given by 
 \begin{align}
   \widetilde\varphi^{(2)}(x_{N_0})\coloneqq \int \phi(x)P^{W^{(1)}_{N\backslash N_0}|W^{(1)}_{k^*}=x_{k^*}}(dx_{N\backslash N_0})
 \end{align}
    is supermodular. By the disintegration theorem and the Markov property, we obtain  
    \begin{align} 
        \begin{split}\label{eq:tilde Z2}
            E[\phi(W^{(1)})]&=E[\widetilde\varphi^{(2)}(W^{(1)}_{N_0})]=E[\widetilde\varphi^{(2)}(X_{N_0})],\\  E[\phi(W^{(2)})]&=E[\widetilde\varphi^{(2)}(W^{(2)}_{N_0})]=E[\widetilde\varphi^{(2)}(Y_{N_0})],
        \end{split}
	\end{align}
 where the second equality in the first and second line follow from \(W^{(1)}_{N_0}\eqd X_{N_0}\) and \(W^{(2)}_{N_0}\eqd Y_{N_0}\,\), respectively.
  Assumptions \eqref{thm11} -- \eqref{thm13}, imply that the subvectors $(X_n)_{n\in N_0}$ and $(Y_n)_{n\in N_0}$ fulfill the assumptions of Proposition \ref{thm: supermodular order of tree specifications}. Thus, \eqref{eq:tilde Z2} implies $ W^{(1)} \leq_{sm} W^{(2)}$, which proves \eqref{eq: firstpartinmaintheoremproof} using $W^{(2)}\eqd Z^{(0)}$ by the definition of \(W^{(2)}\) and \(Z^{(0)}.\)\\
\emph{Step 2:} We aim to show that 
  \begin{align}\label{eq: secondpartinmaintheoremproof}
 Z^{(0)}\leq_{sm} Z^{(1)}\leq_{sm}\dots\leq_{sm}Z^{(2m)}=Y,
 \end{align}
 The proof of \eqref{eq: secondpartinmaintheoremproof} is divided into two parts. First, we show that
  \begin{align}\label{eq:Zk-2}
     Z^{(2k-2)}\leq_{sm}Z^{(2k-1)},\quad\text{for all }k=1,\dots,m,
 \end{align}
and then show that 
 \begin{align}\label{eq:Zk-1}
     Z^{(2k-1)}\leq_{sm}Z^{(2k)},\quad\text{for all }k=1,\dots,m.
 \end{align}
To prove \eqref{eq:Zk-2} let $\phi\colon \mathbb R^{d+1}\to\mathbb R$  be a bounded and supermodular function and $k\in \{1,\dots,m\}$. If $|N_k|\leq 1$ we have that $Z^{(2k-2)}\eqd Z^{(2k-1)}$ and thus $Z^{(2k-2)}\leq_{sm} Z^{(2k-1)}$.  
So assume that \(|N_k|>1.\)
Since, by assumption \eqref{thm12}, $Y_j\uparrow_{st}Y_i$ for all $(i,j)\in \bigcup_{u=0}^{k-1} E_u$ and $Y_{k^*}\uparrow_{st} Y_0$ if $\deg(0)>1$ and $Y_i\uparrow_{st}Y_j$ for all $(i,j)\in \bigcup_{u=1}^{k-1}\{(\ell_{u-1},\ell_{u})\} $, we have that $Z^{(2k-2)}_j\uparrow_{st}Z^{(2k-2)}_i$ for all $(i,j)\in \bigcup_{u=0}^{k-1} E_u $, and $Z^{(2k-2)}_{k^*}\uparrow_{st} Z^{(2k-2)}_0$ if $\deg(0)>1$,  and $Z^{(2k-2)}_i\uparrow_{st}Z^{(2k-2)}_j$ for all $(i,j)\in \bigcup_{u=1}^{k-1}\{(\ell_{u-1},\ell_{u})\} $ . Further, by assumption \eqref{thm11}, $X_j\uparrow_{st}X_i$ for all $(i,j)\in \bigcup_{u=k+1}^{m}E_u\cup \bigcup_{u=k}^{m}\{(\ell_{u-1},\ell_u)\}$ and, consequently, $Z^{(2k-2)}_j\uparrow_{st}Z^{(2k-2)}_i$ for all $(i,j)\in \bigcup_{u=k+1}^{m}E_u\cup \bigcup_{u=k}^{m}\{(\ell_{u-1},\ell_u)\}$. Note that if $k=1$ and $\operatorname{deg}(0)= 1$, it holds that $|N_1|=\{0\}$, which belongs to the previous case.  Now, we obtain from Lemma \ref{lem: supermodular} \eqref{lem: supermodular ii}, applied on the subtree with nodes $(N\backslash N_k)\cup \{\ell_{k-1}\}$ and root $\ell_{k-1}$ by setting $K=N_k\backslash\{\ell_{k-1}\}$ and $J=N\backslash N_k$, that the function $\varphi^{(2k-1)}$ given by
 \begin{align}
     \varphi^{(2k-1)}(x_{N_{k}})\coloneqq \int \phi(x)P^{Z^{(2k-2)}_{N\backslash N_k}\vert Z^{(2k-2)}_{\ell_{k-1}}=x_{\ell_{k-1}}}(dx_{N\backslash N_k})
 \end{align}
 is supermodular. Then we obtain
    \begin{align}
    \begin{split}\label{eq:Z1}
		  E[\phi(Z^{(2k-2)})]&=E[\varphi^{(2k-1)}(Z^{(2k-2)}_{N_k})]=E[\varphi^{(2k-1)}(X_{N_k})],\text{ and } \\E[\phi(Z^{(2k-1)})]&=E[\varphi^{(2^{k-1})}(Z^{(2k-1)}_{ N_k})]=E[\varphi^{(2k-1)}(Y_{N_k})],
    \end{split}
	\end{align}
where the first  equality in each line follows by the disintegration theorem and the Markov  property and the second ones since $X_{N_k}\eqd Z_{N_k}^{(2k-2)}$ and $Y_{N_k}\eqd Z_{N_k}^{(2k-1)}$ by definition of \(Z^{(2k-2)}\) and \(Z^{(2k-1)}\) in \eqref{defZvau}.
 Assumptions \eqref{thm11} -- \eqref{thm13} of Theorem \ref{thm: supermodular order of tree specifications generalized}, imply that the subvectors $(X_n)_{n\in N_k}$ and $(Y_n)_{n\in N_k}$ fulfill the assumptions of Proposition \ref{thm: supermodular order of tree specifications} which implies \(E[\varphi^{(2k-1)}(X_{N_k})] \leq E[\varphi^{(2k-1)}(Y_{N_k})]\). Together with \eqref{eq:Z1}, we obtain \eqref{eq:Zk-2}.

To prove \eqref{eq:Zk-1}, let again $\phi\colon \mathbb R^{d+1}\to\mathbb R$  be a bounded and supermodular function and $k\in \{1,\dots,m\}$. Since, by assumption \eqref{thm12}, $Y_j\uparrow_{st}Y_i$ for all $(i,j)\in \bigcup_{u=0}^{k} E_u $, and $Y_{k^*}\uparrow_{st} Y_0$ if $\deg(0)>1$, and $Y_i\uparrow_{st}Y_j$ f.a. $(i,j)\in \bigcup_{u=1}^{k-1}\{(\ell_{u-1},\ell_{u})\} $, we have that $Z^{(2k-1)}_j\uparrow_{st}Z^{(2k-1)}_i$ for all $(i,j)\in \bigcup_{u=0}^{k} E_u $ and $Z^{(2k-1)}_{k^*}\uparrow_{st} Z^{(2k-1)}_0$ if $\deg(0)>1$ and $Z^{(2k-1)}_i\uparrow_{st}Z^{(2k-1)}_j$ for all $(i,j)\in \bigcup_{u=1}^{k-1}\{(\ell_{u-1},\ell_{u})\} $.
Further, by assumption \eqref{thm11}, $X_j\uparrow_{st}X_i$ for all $(i,j)\in \bigcup_{u=k+1}^{m}E_u\cup \bigcup_{u=k+1}^{m}\{(\ell_{u-1},\ell_u)\}$ and, thus, $Z^{(2k-1)}_j\uparrow_{st}Z^{(2k-1)}_i$ for all $(i,j)\in \bigcup_{u=k+1}^{m}E_u\cup \bigcup_{u=k+1}^{m}\{(\ell_{u-1},\ell_u)\}$. We apply Lemma \ref{lem: supermodular} \eqref{lem: supermodular ii} twice
and find that the function $\varphi^{(2k)}$ 
given by
 \begin{align}
 \begin{split}
  &\varphi^{(2k)}(x_{\ell_{k-1}},x_{\ell_k}) \\&\coloneqq \int \int \phi(x)P^{Z^{(2k-1)}_{N\backslash (\{\ell_{k-1},\ell_{k}\}\cup d_{\ell_k})}|Z^{(2k-1)}_{\ell_{k-1}}=x_{\ell_{k-1}}}(dx_{N\backslash (\{\ell_{k-1},\ell_k\}\cup d_{\ell_k})})\\
  & \hspace{6cm} P^{Z^{(2k-1)}_{ d_{\ell_k}}\vert Z^{(2k-1)}_{\ell_k}=x_{\ell_k}}(dx_{d_{\ell_k}})
 \end{split}
 \end{align}
 is supermodular. Again, we obtain
 \begin{align}
    \begin{split}\label{eq:Z2}
		  E[\phi(Z^{(2k-1)})]&=E[\varphi^{(2k)}(Z^{(2k-1)}_{\ell_{k-1}},Z^{(2k-1)}_{\ell_k})]=E[\varphi^{(2k)}(X_{\ell_{k-1}},X_{\ell_k})],\text{ and } \\E[\phi(Z^{(2k)})]&=E[\varphi^{(2k)}(Z^{(2k)}_{\ell_{k-1}},Z^{(2k)}_{\ell_k})]=E[\varphi^{(2k)}(Y_{\ell_{k-1}},Y_{\ell_k})].
    \end{split}
	\end{align}
 By assumption \eqref{thm13} it holds that $(X_{\ell_{k-1}},X_{\ell_k})\leq_{sm}(Y_{\ell_{k-1}},Y_{\ell_k})$, which together with \eqref{eq:Z2} implies \eqref{eq:Zk-1}.
 
By \eqref{eq:Zk-2} and \eqref{eq:Zk-1} we conclude \eqref{eq: secondpartinmaintheoremproof}. Finally, $X\leq_{sm}Z^{(0)}$, \eqref{eq: secondpartinmaintheoremproof} and the transitivity of the supermodular order imply that $X\leq_{sm} Y$, which proves the first statement. 

To show positive supermodular dependence of $X$, we apply the first statement on $X$ and $Y:=X^{\perp}$.
Note that the random vector $X^{\perp}$ fulfills assumption \eqref{thm12} trivially. Since \(X\) fulfills assumption \eqref{thm11}, we have for $(i,j)\in E\backslash \{(0,k^*)\}$ that $X_{j}\uparrow_{st}X_i$, which implies $(X_i,X_j)\geq_{sm}(X^{\perp}_i,X^{\perp}_j) = (Y_i,Y_j)$; see \eqref{implposdepcon}. 
Hence, with the additional assumption of positive supermodular dependence of $(X_0, X_{k^*})$, we obtain \((X_i,X_j)\geq_{sm} (Y_i,Y_j)\) for all \((i,j)\in E,\) i.e., assumption \eqref{thm13} holds for all $(i, j) \in E$. 
From the first part, we now conclude that $X \geq_{sm} Y = X^\perp$. \\
A similar argument yields positive supermodular dependence of $Y$ under assumption \eqref{thm12} and positive supermodular dependence of \((Y_i,Y_j)\) for \(j\in P\cap L.\) 
\end{proof}

\begin{proof}[Proof of Corollary \ref{corsm}:]
Since the supermodular order is a pure dependence order, it requires the marginal distributions to coincide. Hence, from condition \eqref{thm13} of Theorem \ref{thm: supermodular order of tree specifications generalized} (i.e., \((X_i,X_j)\leq_{sm} (Y_i,Y_j)\) for all \((i,j)\in E\)) it follows that 
\begin{align*}
    F_{X_n} = F_{Y_n} =:F_n \quad\text{for all } n\in N.
\end{align*}
Further, as a consequence of \eqref{eqbivord}, there exist bivariate copulas for $(X_i,X_j)$ and $(Y_i,Y_j)$ that are pointwise ordered, i.e.,
\begin{align*}
    C_{X_i,X_j} \leq_{lo} C_{Y_i,Y_j}\quad \text{for all } e=(i,j)\in E.
\end{align*}
Condition \eqref{thm11} of Theorem \ref{thm: supermodular order of tree specifications generalized} (i.e., \(X_j\uparrow_{st} X_i\) for $(i,j)\in E$, \(j\ne k^*\)), means that the bivariate random vector \((X_i,X_j)\) is CIS. Since, for bivariate random vectors, the SI property is invariant under increasing transformations, it follows that there exists a bivariate copula \(B_e\) such that 
\begin{align*}
    B_e \text{ is SI} \quad \text{and} \quad B_e=C_{X_i,X_j} \text{ on } \Ran(F_i)\times \Ran(F_j) \quad\text{for } e=(i,j)\in E,\, j\ne k^*.
\end{align*}
Let \(D^\top\) denote the transpose of a bivariate copula \(D,\) defined as \(D^\top(u,v):=D(v,u)\) for all \((u,v)\in [0,1]^2.\) 
Then condition \eqref{thm12} of Theorem \ref{thm: supermodular order of tree specifications generalized} (i.e., \(Y_i\uparrow_{st} Y_j\) if \(j\notin L\), and \(Y_j\uparrow_{st} Y_i\) if \(j\notin P\), \(e=(i,j)\in E\)) means that there exists a bivariate copula \(C_e\) such that 
\begin{align*}
\begin{cases}
C_e \text{ is \CI}, & \text{for } j\notin L\cup P, \\
    C_e^\top \text{ is SI}, & \text{for } j\in P\setminus L, \\
    C_e \text{ is SI}, & \text{for } j\in L\setminus P, 
\end{cases}
    &&\quad \text{and} \quad C_e = C_{Y_i,Y_j} \text{ on } \Ran(F_i)\times \Ran(F_j).
\end{align*}
Finally, we have decomposed the distributions of $X$ and $Y$ into bivariate tree specifications with identical marginal distribution functions $F=(F_n)_{n\in N}$, and with pointwise ordered SI and \textnormal{CI} copulas $B=(B_e)_{e\in E}$ and $C=(C_e)_{e\in E}$.
Due to Proposition \ref{prop: existence of markov tree dirstribution}, every bivariate tree specification yields a uniquely determined Markov tree realization, which proves the results incorporating some additional SI conditions.
\end{proof}

   


\subsection{Proofs of Section \ref{sec: preliminaries}}\label{appendix: A}

\begin{proof}[Proof of Proposition \ref{prop: existence of markov tree dirstribution}]
 For $|N|<\infty$, we prove the statement by an induction over the number of nodes  $N=\{0,\dots,d\}$. Starting with $|N|=2$, the statement is trivial. For the induction step, assume that the statement is true for each tree $T$ with nodes $N=\{0,\dots,n\}$ for a fixed \(n\in \{1,\ldots,d-1\}.\) We need to show the statement for \(N=\{0,\dots,n+1\}\). By re-enumeration we may assume that the node $n+1$ is a leaf in $N$ with parent node $n$, i.e., $c_{n+1}=\emptyset$ and $p_{n+1}=n$; see Definition \ref{def:childs_descendants_ancestors}.  We define the subtree  $T'=(N',E')$ with nodes $N'=\{0,\dots,n\}$ and edges $E'=\{(j,k)\in E \mid j,k\neq n+1\}$, which is a tree with $|N'|-1=n$ nodes. By the induction hypothesis, there is a Markov tree distribution $\mu_{\cT'}$  specified through $((F_j)_{j\in N'},T',(B_e)_{e\in E'})$. Let $(X_{n}, X_{n+1})$ be a bivariate random vector having distribution function $F_{X_n,X_{n+1}}:= B_{n,n+1}(F_{n},F_{n+1})$ defined through Sklar's theorem. For Borel sets $A\in \mathscr B(\mathbb R^{n+1})$ and $B\in \mathscr B (\mathbb R)$, we define the distribution $\mu_{\cT}$ on \(\cB(\R^{n+2})\) by 
	\begin{align}\label{eqconmkd}
		\mu_{\cT}(A\times B) :=\int_{A}P^{X_{n+1}\vert X_{n}}(B\vert x_n)\mu_{\cT'}(dx_0,\dots,dx_n).
	\end{align}
 Similar to the proof of \cite[Theorem 2]{meeuwissen1994tree}, it follows that \(\mu_\cT\) is the unique distribution with the given specifications that has Markov tree dependence. For $|N|=\infty$, the distribution $\mu_{\cT}$ can be obtained by the Kolmogorov extension Theorem (see \cite[Theorem 6.16]{kallenberg1997foundations}), where the assumption of projective families is satisfied as a consequence of the disintegration theorem.
\end{proof}

\subsection{Proofs of Section \ref{sec3}}\label{appendix: B}


For the proofs of Theorems \ref{thm: maintwo} and \ref{themaindcx}, we make use of the following lemma.

\begin{lemma}\label{prop: Copula of bivariate tree specifications}
	Let $X = (X_0,\ldots,X_d)\sim \cM((F_0,\dots,F_d),T,C)$ be a random vector with marginal specifications $F_0,\dots,F_d$. Further, let $U= (U_0,\ldots,U_d)\sim \cM((G_0,\dots,G_d),T,C)$ be a random vector with distribution function $F_U$, where $G_i$ denotes the distribution function of the rank-transformed random variable \(F_i(X_i).\) Then we have
    \begin{align}
        \label{prop: Copula of bivariate tree specifications i} F_X(x_0,\dots,x_d)=F_U(F_0(x_0),\dots,F_d(x_d))
    \end{align} 
    for all \((x_0,\ldots,x_d)\in \R^{d+1}.\)
\end{lemma}

\begin{proof}[Proof of Lemma \ref{prop: Copula of bivariate tree specifications}]
    We prove the statement by an induction over the number of nodes $N=\{0,\dots,d\}$. Starting with $|N|=2$, the statement is trivial. For the induction step, assume that the statement is true for each tree $T$ with nodes $N=\{0,\dots,n\}$ for a fixed \(n\in \{1,\ldots,d-1\}.\) We need to show the statement for \(N=\{0,\dots,n+1\}\). By re-enumeration we may assume that $n+1\in N$ is a leaf in $N$ with parent node $n$, i.e., $c_{n+1}=\emptyset$ and $p_{n+1}=n$; see Definition \ref{def:childs_descendants_ancestors}. We define the subtree  $T'=(N',E')$ with nodes $N'=\{0,\dots,n\}$ and edges $E'=\{(j,k)\in E|j,k\neq n+1\}$ which is a tree with number of nodes $|N'|-1=n$. By \eqref{eqconmkd} and the transformation formula, using that $F_i^{-1}\circ F_i(X_i) = X_i $  $P$-a.s. (see e.g. \cite[Lemma A.1(ii)]{Ansari-Rueschendorf-2021}), we have that 
    \begin{align}\label{eq: propositionforcopula}
    \begin{split}
                &F_{(X_0,\dots,X_{n+1})}(x_0,\dots,x_n,x_{n+1})\\&=\int_{(-\infty,x_0]\times\dots\times (-\infty,x_n]}F_{X_{n+1}\vert X_n=y_n}(x_{n+1})P^{(X_0,\dots,X_n)}(dy_0,\dots,dy_n)\\
        &=\int_{(-\infty,F_0(x_0)]\times\dots\times (-\infty,F(x_n)]}F_{X_{n+1}\vert X_n=F_n^{-1}(u_n)}(x_{n+1})P^{(F_0(X_0),\dots,F_n(X_n))}(\de \uu)
    \end{split}
    \end{align}
    for \(\de\uu = (du_0,\dots,du_n).\)
Since copulas are invariant under componentwise increasing transformations, it follows that the random vectors $(F_0(X_0),\dots,F_n(X_n))$ and $(X_0,\dots,X_n)$ have the same copula (see e.g. \cite[Theorem 3.3]{Cai-2012}). Moreover, by the induction hypothesis the distribution function $F_{U_0,\dots,U_n}$ coincides with the copula of $(X_0,\dots,X_n)$ on \(\Ran(F_0)\times \cdots\times \Ran(F_d).\) By uniqueness of the copula on \(\Ran(F_0)\times \cdots\times \Ran(F_d),\) we obtain that  
$(F_0(X_0),\dots,F_n(X_n))\eqd (U_0,\dots,U_n)$. Then we have
\begin{align*}
    F_{X_{n+1}\vert X_n=F_n^{-1}(u_n)}(x_{n+1})
    & =\partial_1^{F_n} C_{n,n+1}(F_n(F_n^{-1}(u_n)),F_{n+1}(x_{n+1})) \\
    & =\partial_1^{G_n} C_{n,n+1}(G_n(G_n^{-1}(u_n)),F_{n+1}(x_{n+1})) \\
    & =\partial_1^{G_n} C_{n,n+1}(G_n(G_n^{-1}(u_n)),G_{n+1}(F_{n+1}(x_{n+1}))) \\
    & =F_{U_{n+1}\vert U_n=u_n}(F_{n+1}(x_{n+1})),
\end{align*}
for Lebesgue-almost all \(u_n\in \Ran(F_n)\) and for all \(x_{n+1}\in \Ran(F_{n+1}),\) where the first and last equality hold true by \cite[Theorem 2.2]{Ansari-Rueschendorf-2021}.
Here \(\partial_1^{H}\) denotes a generalized partial derivative operator that takes the partial derivative with respect to the first component and with respect to a univariate distribution function \(H,\) see \cite[Section 2.1]{Ansari-Rueschendorf-2021}.  
For the second equality, we use the assumptions that \((X_n,X_{n+1})\) and \((U_n,U_{n+1})\) have the same copula, and that \(\overline{\Ran(F_n)} = \overline{\Ran(G_n)}\), noting that \(\partial_1^H\) has the property that it only depend on \(\overline{\Ran(H)}.\) For the third equality, we use that \(G_{n+1}(v) = v\) for all \(v\in \Ran(G_{n+1}).\)
Hence, we obtain that \eqref{eq: propositionforcopula} equals 
\begin{align*}
&\int_{(-\infty,F_0(x_0)]\times\dots\times (-\infty,F(x_n)]}F_{U_{n+1}\vert U_n=u_n}(F_{n+1}(x_{n+1}))P^{(U_0,\dots,U_n)}(du_0,\dots,du_n)\\
&=F_{(U_0,\dots,U_{n+1})}(F_0(x_0),\dots,F_d(x_{n+1}))
\end{align*}
for all \((x_0,\ldots,x_{n+1})\in \R^{n+2},\)
which proves the statement. 
\end{proof}

\begin{proof}[Proof of Theorem \ref{thm: maintwo}:]
By the definition of the orderings for stochastic processes, Definition \ref{defdepord} \eqref{defdepordb}, we may assume that $|N|<\infty$.
    Let \(Z = (Z_n)_{n\in N}\sim\cM(F,T,C).\) Then we have \(X\leq_{sm} Z\) due to Theorem \ref{thm: supermodular order of tree specifications generalized} noting that, due to \eqref{eqbivord}, \(B_e\leq_{lo} C_e,\) \(X_i \eqd Z_i,\) and \(X_j\eqd Z_j\) implies \((X_i,X_j)\leq_{sm} (Z_i,Z_j)\) for \(e = (i,j).\) 
    Now, observe that \(Z\) and \(Y\) have the same bivariate copula specifications. Since \(\overline{\Ran(F_n)} = \overline{\Ran(G_n)}\) and thus \(F_n(Z_n)\eqd G_n(Y_n)\) for all \(n\), we obtain from Lemma \ref{prop: Copula of bivariate tree specifications} that \(Z\) and \(Y\) have a common copula.    
    Using the assumption that \(Z_n\leq_{st} Y_n\) (resp. \(Z_n\geq_{st} Y_n\)) for all \(n,\) it follows that \(Z\leq_{st} Y\) (resp. \(Z\geq_{st} Y\)) see \cite[Theorem 4.1]{Mueller-Scarsini-2001}. Hence, for \(f\) increasing (decreasing) supermodular, we obtain \(\E f(X) \leq \E f(Z) \leq \E f(Y)\) whenever the expectations exist.
\end{proof}

\begin{proof}[Proof of Theorem \ref{themaindcx}]
Again, we may assume that $|N|<\infty$. Let $Z=(Z_n)_{n\in N}\sim \mathcal{M}(F,T,C)$ be a random vector having the same univariate marginals as \(X\) and the same bivariate dependence specifications as \(Y.\) 
By assumption \eqref{themaindcx2} and \eqref{implposdepcon} the  copulas $C_e$, \(e\in E,\) are \CI and thus $(Z_i,Z_j)$ is \CI for all $(i,j)\in E$. We deduce from Theorem \ref{thm: supermodular order of tree specifications generalized} that $X\leq_{sm}Z$ and consequently $X\leq_{dcx}Z$. 
When verifying the assumptions for Theorem \nolinebreak \ref{thm: supermodular order of tree specifications generalized}, it is worth noting that distinguishing between the cases where $\operatorname{deg}(0) \geq 2$ and $\operatorname{deg}(0) < 2$ is not necessary since $(Y_i,Y_j)$ is \CI for all $(i,j)\in E$. Consequently, in the case where $\operatorname{deg}(0) \geq 2$, it is always possible to select the path $P$ in the assumptions of Theorem \ref{thm: supermodular order of tree specifications generalized} in such a way that $k^*\notin P.$ 
By the continuity of $F_n$ and $G_n$ for all $n\in N$ and Lemma \ref{prop: Copula of bivariate tree specifications}, the vectors
$Z$ and $Y$ have the same copula. Using the MTP$_2$ assumption, we obtain from Lemma \ref{lemMTP2}, that this copula is MTP$_2$, and thus CI. Hence, it follows with Lemma \ref{lemcCI}, using the assumption \(F_n\leq_{cx} G_n\) for all $n\in N$, that $Z\leq_{dcx}Y$. The transitivity of the directionally convex order implies $X\leq_{dcx}Y$. 
\end{proof}

\subsection{Proofs of Section \ref{secSM}}

\begin{proof}[Proof of Proposition \ref{propSch}]
    Using the notation in \cite{Ansari-Rueschendorf-2023}, there exists for all \(x\in \R\) a Lebesgue-null set \(N_x\subset [0,1]\) such that \(F_{X_i|X_0=q_{X_0}(u)}(x) = \partial_2^{F_{X_0}} C_{X_i,X_0}(F_{X_i}(x),u)\) and \(F_{Y_i|Y_0=q_{Y_0}(u)}(x) = \partial_2^{F_{Y_0}} C_{Y_i,Y_0}(F_{Y_i}(x),u)\) for all \(u\in (0,1)\setminus N_x;\) see \cite[Theorem 2.2]{Ansari-Rueschendorf-2021}. Hence, the statement follows from \cite[Corollary 4(i)]{Ansari-Rueschendorf-2023} because \((X_i|X_0)\leq_S (Y_i,Y_0)\) implies \(\partial_2^{F_{X_0}} C_{X_i,X_0}(F_{X_i}(x),\cdot) \prec_S \partial_2^{F_{Y_0}} C_{Y_i,Y_0}(F_{Y_i}(x),\cdot)\) for all \(x\in \R.\)
    \end{proof}

\subsection{Proofs of Section \ref{secHidMarkov} }\label{appendix: HMM}

\begin{proof}[Proof of Corollary \ref{thmHMM}]
    Hidden Markov models follow a Markov tree distribution with respect to the tree $T=(\N_0,E)$, where $E$ is defined in \eqref{eq: edges HMM}. The first statement then follows by Theorem \ref{thm: supermodular order of tree specifications generalized}. The second statement is a direct consequence of Theorem \ref{thm: maintwo}. The third statement follows from Theorem \ref{themaindcx}.
\end{proof}


\subsection{Proofs of Appendix \ref{secSI}}\label{appendix: D}

\begin{proof}[Proof of Proposition \ref{prop: justification of the SI assumptions}]

\eqref{prop41} We consider two cases. If $j^*\in d_{k^*}$, i.e., $j^*$ is an element of the descendants of $k^*$, we define the vectors by $(X_0,X_{k^*},X_{p(k^*,j^*)},X_{j^*})\coloneqq -Y'$ and $(Y_0,Y_{k^*},Y_{p(k^*,j^*)},Y_{j^*})\coloneqq -X'$, where $X'$ and $Y'$ are given by Example \nolinebreak \ref{exp:c}. Then we extend these vectors by independent random variables (note that the independence copula is \textnormal{CI}) to the Markov tree distributed sequences $X=(X_i)_{i\in N}$ and $Y=(Y_i)_{i\in N}$. These vectors fulfill assumption \eqref{thm11 mod} as well as assumptions \eqref{thm12} and \eqref{thm13} of Theorem \ref{thm: supermodular order of tree specifications generalized}. 
Since \((X_0,X_{k^*},X_{p(k^*,j^*)},X_{j^*})\not\leq_{lo} (Y_0,Y_{k^*},Y_{p(k^*,j^*)},Y_{j^*})\), we obtain from the closure of the lower orthant order under marginalization that $ X \nleq_{lo}Y $ and thus \(X\nleq_{sm} Y.\) In the case where $j^*\notin d_{k^*}$, which implies $\operatorname{deg}(0)\geq 2$, we define the vectors $(X_{k^*},X_{0},X_{p(0,j^*)},X_{j^*})\coloneqq -Y'$ and $(Y_{k^*},Y_{0},Y_{p(0,j^*)},Y_{j^*})\coloneqq -X'$, and proceed as in the first case. 

\eqref{prop42} Since \(T\) is not a star, we have $N\setminus (L\cup\{0\}) \ne \emptyset$. The node \(j^*\in N\setminus (L\cup\{0\})\)  has a parent node  $i\coloneqq p_{j^*}$ and at least one child node $k\in c_{j^*}$. Define the subvectors $(X_i,X_{j^*},X_k)$ and  $(Y_i,Y_{j^*},Y_k)$ as in Example \nolinebreak \ref{exp:a} and extend these $3$-dimensional vectors by independent random variables (note that the independence copula is \textnormal{CI}) to the random vectors $X=(X_i)_{i\in N}$ and $Y=(Y_i)_{i\in N}$. Then \(X\) and \(Y\) fulfill  assumption \eqref{thm12 1mod} as well as assumptions \eqref{thm11} and \eqref{thm13} of Theorem \ref{thm: supermodular order of tree specifications generalized}. We obtain from Example \ref{exp:a} that \((X_i,X_{j^*},X_k)\nleq_{lo} (Y_i,Y_{j^*},Y_k)\) and thus $X\nleq_{lo} Y$ and $X\nleq_{sm} Y$. This shows the second statement.

\eqref{prop43} Since $T$ is not a chain, we have $N\backslash (P\cup \{0\})\neq \emptyset$. For $j^*\in N\backslash (P\cup \{0\})$  we consider the directed path $(0,\dots,j^*)$ from the root $0$ to $j^*$. Let $i$ be the last node in $(0,\dots,j^*)$ such that $i\in P\cup\{0\}$, i.e., the unique node $i\in \{0,\dots,j^*\}\cap (P\cup\{0\}) $ such that $c_i\cap \{0,\dots,j^*\}\cap P=\emptyset$. Since \(i\in P\) and $j^*\notin P$ the node $i$ has a child node $k\in  P\cap c_i$. Define the subvectors $(X_k,X_{p(k,j^*)},X_{j^*})$  and  $(Y_k,Y_{p(k,j^*)},Y_{j^*})$ as in Example \ref{exp:c} and extend this $3$-dimensional vectors by independent random variables (note that the independence copula is \textnormal{CI}) to the random vectors $X=(X_i)_{i\in N}$ and $Y=(Y_i)_{i\in N}$. Then \(X\) and \(Y\) fulfill  assumption \eqref{thm12 2mod} as well as assumptions \eqref{thm11} and \eqref{thm13} of Theorem \ref{thm: supermodular order of tree specifications generalized}. Due to Example \ref{exp:c}, we obtain $(X_k,X_{p(k,j^*)},X_{j^*})\nleq_{lo}(Y_k,Y_{p(k,j^*)},Y_{j^*})$ and consequently $X\nleq_{lo} Y$ and $X\nleq_{sm} Y$. This shows the third statement and the result is proven.
\end{proof}

\end{appendices}

\section*{Acknowledgments}
The authors would like to thank two anonymous referees and an Associate
Editor for careful reading of the manuscript.

\section*{Funding}
The first author gratefully acknowledges the support of the Austrian Science Fund (FWF)
projects 10.55776/PAT1669224 ``SORT: Stochastic Orders for Functional Dependence'' and P 36155-N ``ReDim: Quantifying Dependence via Dimension Reduction'', as well as the support of the WISS 2025 project ’IDA-lab Salzburg’ (20204-WISS/225/197-2019 and 20102-F1901166-KZP), and the support of the Early Career Project 'VeRAin' (IPE191001\_01) funded by the University of Salzburg.









\end{document}